\documentclass[reqno,12pt]{amsart}

\usepackage[utf8]{inputenc}
\usepackage{amsmath}
\usepackage{amsthm}
\usepackage{amssymb}
\usepackage{amsfonts,mathrsfs}
\usepackage{stmaryrd}
\usepackage{amsxtra}  
\usepackage{epsfig}
\usepackage{verbatim}
\usepackage{enumerate}
\usepackage{xcolor}
\usepackage[all]{xy}
\usepackage{graphicx}
\usepackage{enumitem}

\usepackage{tikz}
\usepackage{pgfplots}
\pgfplotsset{compat=1.11}
\usepgfplotslibrary{fillbetween}
\usetikzlibrary{intersections}

\usepackage{mathtools, hyperref}
\usetikzlibrary{shapes}
\usetikzlibrary{plotmarks}
\usepackage{bm}

\theoremstyle{plain}
\newtheorem{theorem}[equation]{Theorem}
\newtheorem{proposition}[equation]{Proposition}
\newtheorem{lemma}[equation]{Lemma}
\newtheorem{corollary}[equation]{Corollary}
\newtheorem{definition}[equation]{Definition}

\theoremstyle{remark}
\theoremstyle{remark}
\newtheorem{remark}[equation]{Remark}
\numberwithin{equation}{section}

\newcommand{\dbarstar}{\bar{\partial}^{\star}}

\newcommand{\dbar}{\bar \partial}

\newcommand{\abs}[1]{\left\vert#1\right\vert}
\newcommand{\norm}[1]{\left\Vert#1\right\Vert}
\newcommand{\ol}{\overline}

\newcommand{\wt}{\widetilde}

\newcommand{\ipr}[1]{\left( #1 \right)}

\newcommand{\xdownarrow}[1]{%
	{\left\downarrow\vbox to #1{}\right.\kern-\nulldelimiterspace}
}
\makeatletter
\newcommand*{\dashdownarrow}{%
	\mathrel{%
		\mathpalette\dasharrow@vert{-90}%
	}%
}
\newcommand*{\dashuparrow}{%
	\mathrel{%
		\mathpalette\dasharrow@vert{90}%
	}%
}
\newcommand*{\dasharrow@vert}[2]{%
	\sbox0{$#1\vcenter{}$}%
	\sbox2{$#1\dashrightarrow\m@th$}%
	\dimen@=1.2\dimexpr\ht2-\ht0\relax
	\sbox2{\raisebox{-\ht0}{\unhcopy2}}%
	\ht2=\z@
	\dp2=\z@
	\vcenter{\hbox to 2\dimen@{\hfill\rotatebox{#2}{\box2}\hfill}}%
}
\makeatother

\makeatletter
\renewcommand*\env@matrix[1][\arraystretch]{%
	\edef\arraystretch{#1}%
	\hskip -\arraycolsep
	\let\@ifnextchar\new@ifnextchar
	\array{*\c@MaxMatrixCols c}}
\makeatother

\newcommand{\ch}{{\mathcal H}}

\newcommand{\cl}{{\mathcal L}}
\newcommand{\cm}{{\mathcal M}}

\newcommand{\cp}{{\mathcal P}}

\newcommand{\ct}{{\mathcal T}}

\newcommand{\sC}{{\mathscr C}}

\newcommand{\vp}{\varphi}

\newcommand{\C}{{\mathbb C}}

\newcommand{\Z}{{\mathbb Z}}

\newcommand{\hwedge}[2]{\langle #1 \wedge \overline{#2}, h\rangle}
\newcommand{\sqi}{\sqrt{\scalebox{1.5}[0.9]{-}1}}

\allowdisplaybreaks[4]

\usepackage{setspace}

\begin{document}
\thispagestyle{empty}

\title{Direct Images and Hilbert Fields}
\author{Pranav Upadrashta}

\begin{abstract}
    In this article we undertake a careful study of the curvature operator of a field of Bergman spaces associated to a holomorphic fibration of K\"ahler manifolds. We focus on the case when the fibers are smoothly bounded pseudoconvex domains. 
\end{abstract}
\maketitle

\section{Introduction}
We are interested in the geometry of families of Bergman spaces associated to a holomorphic submersion. More precisely, let $\wt{X}$ be an $n+m$-dimensional K\"ahler manifold, let $B$ be an $m$-dimensional complex manifold and let $\wt{\pi}:\wt{X}\to B$ be a holomorphic submersion. Let $X\subset \wt{X}$ be a smoothly bounded domain whose fibers $X_b := \wt{\pi}^{-1}(b)\cap X$ have smooth pseudoconvex boundary. Let $E\to \wt{X}$ be a holomorphic line bundle with a smooth hermitian metric $h$. There is a family of Bergman spaces $\ch\to B$ associated to the submersion $\pi :=\wt{\pi}|_{X}$, which assigns to each $b \in B$ the Bergman space
\[
\ch_b = \left\{ f\in \Gamma(X_b, \mathcal{O}(K_{X_b}\otimes E|_{X_b}))  : (f,f)_b := \sqi^{n^2}\int_{X_b} \hwedge{f}{f} < \infty\right\}.
\]
We prove the following theorem about the family $\ch\to B$.
\begin{theorem}\label{Bgeneral}
Let $(E,h)\to X\to B$ be as above and let the ambient manifold $\wt{X}$ be Stein. Assume that for every $b\in B$ the fiber $X_b$ is such that the Neumann operator $\mathcal{N}_b^{(0,1)}$ acting on $(0,1)$-forms is compact and $X_b$ admits a neighborhood of Stein domains in $\wt{X}_b$. If $(E,h)$ has nonnegative curvature and $X$ is pseudoconvex then $\ch$ has nonnegative Nakano curvature.
\end{theorem}

In fact Theorem \ref{Bgeneral} holds under less restrictive assumptions on the geometry of the fibration $X\to B$. Defining the curvature of $\ch$ associated to such general fibrations is a subtle matter involving technicalities. The main reason for these technical difficulties is that $\ch \to B$ need not be a Hilbert bundle as it may fail to be locally trivial. However, using the ideas developed in \cite{LS2014} and \cite{V2021} we can equip $\ch$ with additional structure so that its \emph{curvature} is well-defined. The additional structure is that of an \emph{iBLS field}. (See Section \ref{iBLS} for details.)

The geometry of $\ch \to B$ is much simpler when $X$ is a product. Suppose that $\wt{X} = \C^n \times B$ where $B$ is a domain in $\C^m$ and $X = \Omega \times B \subset \wt{X}$ where $\Omega \subset \C^n$ is a smoothly bounded pseudoconvex domain. Let $(E,h)\to \wt{X}$ be the trivial line bundle with a nontrivial metric. In this case the metric is just a function, which we assume to be smooth on $\C^n\times \ol{B}$ for the sake of simplicity. Then the Hilbert spaces $\ch_b$ are all canonically isomorphic as topological vector spaces, so $\ch \to B$ is a trivial vector bundle. Also, the curvature of the Chern connection associated to the $L^2$-metric on $\ch$ is a smooth $(1,1)$-form taking values in $\text{End}(\ch)$. Here $\text{End}(\ch)$ denotes the bundle morphisms of $\ch \to B$ that restrict to continuous operators on the fiber $\ch_b$ for every $b\in B$. It is a theorem of B. Berndtsson (\cite[Theorem 1.1]{B2009}) that if $(E,h)$ has nonnegative curvature then the curvature of the Chern connection on $\ch$ is Nakano nonnegative. 

Returning to the case of a general submersion $\wt{\pi}:\wt{X}\to B$, Theorem \ref{Bgeneral} was proved by X. Wang (\cite{W2017}) under the assumption that $X$ has a defining function $\rho$ such that $\rho|_{\wt{\pi}^{-1}(b)}$ is strictly plurisubharmonic on a neighborhood of $X_b$ for every $b\in B$. In this case the compactness of $\mathcal{N}_b^{(0,1)}$ goes back to the work of J. Kohn.

In the next theorem we obtain a formula for the curvature of $\ch$.
\begin{theorem}[= Theorem \ref{smoothcurv}]\label{smoothcurvintro}
Let $\wt{\pi}:\wt{X}\to B$ be a holomorphic submersion, let $(E,h)\to \wt{X}$ a hermitian holomorphic line bundle and let $X\subset \wt{X}$ be a bounded domain so that the fibers $X_b := X\cap \wt{\pi}^{-1}(b)$ are smoothly bounded pseudoconvex domains. Assume that $\ch\to B$ is an iBLS field. Then the curvature of $\ch$ is given by
\begin{align}
    \ipr{\Theta^{\mathcal{\ch}}_{\sigma_1\ol{\sigma_2}}f_1,f_2}_b &= \ipr{\Theta^E_{\xi_{\sigma_1}\ol{\xi}_{\sigma_2}} u_1,u_2}_b - \sqi^{n^2} \int_{X_b} \hwedge{\dbar \xi_{\sigma_1}\lrcorner u_1}{\dbar \xi_{\sigma_2}\lrcorner u_2} + \int_{\partial X_b} \partial \dbar\rho(\xi_{\sigma_1},\ol{\xi}_{\sigma_2})\{u_1,u_2\}dS_b \nonumber \\
    &\quad \quad - \ipr{P_b^{\perp}L^{1,0}_{\xi_{\sigma_1}}u_1, P_b^{\perp}L^{1,0}_{\xi_{\sigma_2}}u_2}_b. \label{hcurv}
\end{align}
Here $\sigma_1, \sigma_2$ are holomorphic $(1,0)$-vector fields on an open set $U$ containing $b$ and $\xi_{\sigma_1}, \xi_{\sigma_2}$ are their horizontal lifts to $X$, i.e., $d\wt{\pi}(\xi_{\sigma_i})=\sigma_i$ and $\xi_{\sigma_i}(x) \in T^{1,0}_{\partial X, x}$ for all $x\in \partial X\cap \wt{\pi}^{-1}(U)$. Also $f_1,f_2 \in \ch_b$ are such that $f_i = \iota_{X_b}^*u_i$ for a twisted $(n,0)$-form $u_i$ that is smooth on $\ol{X}$ and satisfies 
\begin{equation}\label{ufcond} 
\iota_{X_b}^*L^{0,1}_{\ol{\xi}_{\sigma_i}}u_i \in \ch_b \quad \text{and}  \quad \iota_{X_b}^*L^{1,0}_{\xi_{\sigma_i}}\dbar u_i=0.
\end{equation}
\end{theorem}
A few clarifying remarks about notation. For a $(1,0)$-vector field $\eta$ on $X$, the form $L^{1,0}_{\eta}u$ is the twisted Lie derivative of an $E$-valued form $u$ given by
    \[
    L^{1,0}_{\eta}u  = \nabla^{1,0}(\eta\lrcorner u) + \eta \lrcorner \nabla^{1,0}u.
    \]
    Likewise $L^{0,1}_{\ol{\eta}}u$ is 
    \[
    L^{0,1}_{\ol\eta}u = \dbar(\ol\eta\lrcorner u) + \ol\eta\lrcorner \dbar u.
    \]
    Here $\nabla = \nabla^{1,0}+\dbar$ is the Chern connection for $(E,h)$.
    The map $\iota_{X_b}:X_b\to X$ denotes the natural inclusion and $P_b:\Gamma(X_b, L^2(K_{X_b}\otimes E|_{X_b}))\to \ch_b$ denotes the Bergman projection. 
    For a local $(1,0)$-vector field $\eta = \eta_k \frac{\partial}{\partial z^k}$ (here and below the standard summation convention is in place) we use the notation 
    \[
    \dbar \eta\lrcorner u = \dbar \eta^k \wedge \ipr{ \frac{\partial}{\partial z^k}\lrcorner u}.
    \]
    The notation $\{\cdot,\cdot\}$ denotes the inner product induced on $E|_{X_b}$-valued forms by metric $h_b$ on $E|_{X_b}$ and the K\"ahler metric $\omega_b$ on $T^{1,0}_{X_b}$. Finally $\rho$ is a defining function for $X$ so that the defining function $\rho_b := \rho|_{\wt{\pi}^{-1}(b)}$ for $X_b$ satisfies $\abs{d\rho_b}=1$ on $\partial X_b$, and $dS_b$ is the induced volume form on $\partial X_b$, i.e., $d\rho_b \wedge dS_b = dV_{\omega_b}|_{\partial X_b}$. 

\begin{remark}
The right hand side of \eqref{hcurv} is independent of the choice of horizontal lifts $\xi_{\sigma_1},\xi_{\sigma_2}$ of the  vector fields $\sigma_1, \sigma_2$. (see Theorem \ref{liftinv}). The element $(\Theta^{\mathcal{\ch}}_{\sigma_1\ol{\sigma}_2}f)(b) \in \ch_b$ depends only on the values $\sigma_i(b)$ of the vector fields $\sigma_i$ at $b$. The sections $f$ of $E|_{X_b}\otimes K_{X_b} \to X_b$ which arise as restrictions of $E$-valued $(n,0)$-forms $u$ that satisfy \eqref{ufcond} are in the domain of the operator $\Theta^{\ch}_{\sigma_1(b)\ol{\sigma}_2(b)}$ and the element $\Theta^{\mathcal{\ch}}_{\sigma_1(b)\ol{\sigma}_2(b)}f$
is independent of the choice of $u$ that extends $f$, provided $u$ satisfies \eqref{ufcond}. Informally, this means that $\Theta^{\ch}$ is a (possibly nonsmooth) operator valued $(1,1)$-form.
\end{remark}

When the line bundle $(E,h)\to X$ has non-negative curvature, we are able to obtain the following estimate.

\begin{theorem}[=Theorem \ref{2ndfunprop}]\label{2ndfunpropintro}
Let $(E,h)\to X\to B$ be as in Theorem \ref{smoothcurvintro}. Suppose that $(E,h)$ has nonnegative curvature. Let $f_1,\cdots, f_m \in \ch_b$ such that $f_i = \iota_{X_b}^*u_i$ for twisted $(n,0)$-form $u_i$ satisfying \eqref{ufcond}, where $\sigma_1, \cdots, \sigma_m$ are holomorphic $(1,0)$-vector fields on an open set containing $b$ and $\xi_{\sigma_1},\cdots, \xi_{\sigma_m}$ are their horizontal lifts to $X$. Then we have the following estimate
\begin{equation}\label{2ndfunestimateintro}
    \norm{P_b^{\perp}\left(\sum_{j=1}^m L^{1,0}_{\xi_{\sigma_j}} u_j\right)}_{\ch_b}^2 \leq \sum_{j,k=1}^m \left[\ipr{\Theta^E_{\xi_{\sigma_j}\ol{\xi}_{\sigma_k}}f_j,f_k}_{\ch_b} + \ipr{ \dbar\xi_{\sigma_j}\lrcorner f_j,\dbar\xi_{\sigma_k}\lrcorner f_k}_{\ch_b}\right].
\end{equation}
\end{theorem}
In view of the estimate \eqref{2ndfunestimateintro} and Theorem \ref{smoothcurvintro}, to obtain Nakano positivity of $\Theta^{\ch}$ we would like to have 
\[
- \sum_{j,k=1}^m \sqi^{n^2} \int_{X_b} \hwedge{\dbar \xi_{\sigma_j}\lrcorner u_j}{\dbar \xi_{\sigma_k}\lrcorner u_k} = \sum_{j,k=1}^m \ipr{ \dbar\xi_{\sigma_j}\lrcorner u_j,\dbar\xi_{\sigma_k}\lrcorner u_k}_{\ch_b}.
\]
Equivalently, with 
\begin{equation}\label{kappa}
    \kappa = \iota_{X_b}^*\left(\sum_{j=1}^m \dbar\xi_{\sigma_j}\lrcorner u_j\right)
\end{equation}
we wish to show that 
\begin{equation}\label{kappaprim}
-\sqi^{n^2}\hwedge{\kappa}{\kappa}=\{\kappa,\kappa\}.
\end{equation}
The identity \eqref{kappaprim} holds when $\kappa$ is a primitive $E$-valued $(n-1,1)$-form (see \cite[Corollary 1.2.36]{huyb}). Recall that a twisted $(n-1,1)$-form $v$ on $X_b$ is said to be primitive if $\omega_b\wedge v = 0$. If the Neumann operator $\mathcal{N}_b^{(0,2)}$ acting on $(0,2)$-forms on $X_b$ is globally regular, then we can get horizontal lifts $\xi_\tau$ to $X$ of a $(1,0)$-vector field $\tau$ on a neighbourhood of $b$ so that $\iota_{X_b}^*(\dbar \xi_{\tau}\lrcorner u)$ is primitive on $X_b$ for every twisted $(n,0)$-form $u$ on $X$ (see Lemma \ref{primlift}). These lifts, called \emph{primitive horizontal lifts} were introduced in \cite{V2021} in the setting of proper fibrations. Even though we assume compactness of $\mathcal{N}_b^{0,1}$ in Theorem \ref{Bgeneral}, the global regularity of $\mathcal{N}_b^{0,2}$ is sufficient to prove Nakano positivity of $\ch$. Recall that compactness of $\mathcal{N}_b^{0,1}$ implies the compactness of $\mathcal{N}_b^{0,2}$ (see \cite[Proposition 4.5]{Strbook}), and that compactness of $\mathcal{N}_b^{0,2}$ implies global regularity is a result of Kohn-Nirenberg (\cite{KoNi}). When $n=1$, i.e., the fibers $X_b$ are Riemann surfaces, every $(0,1)$-form is primitive so every horizontal lift $\xi_\tau$ is a primitive horizontal lift of $\tau$. Thus, Theorem \ref{Bgeneral} follows for one-dimensional fibers even though $\mathcal{N}_b^{0,2}$ does not exist.

In the case of the trivial fibration $\pi: X\to B$, i.e., when $X= \Omega \times B$, the tangent bundle of $X$ splits canonically as $T_X^{1,0} = p_1^*T^{1,0}_{\Omega}\oplus \pi^*T^{1,0}_B$ where $p_1:X\to \Omega$ is the projection onto the first factor. A vector field $\tau \in \Gamma(B, T^{1,0}_B)$ can be lifted as $\xi_\tau = (0,\tau)$ with respect to this splitting. These lifts satisfy $\dbar \xi_{\tau} \equiv 0$, so they are primitive horizontal lifts. Thus, one does not need the global regularity of $\mathcal{N}^{0,2}$ for $\Omega$ to prove Theorem \ref{Bgeneral}. Another feature of the trivial fibration is that these primitive horizontal lifts lie in the kernel of the Levi form of $\partial X$, i.e., $\partial \dbar\rho(\xi_\tau, \ol{\xi}_\tau)=0$ for lifts $\xi_\tau = (0,\tau)$ as above. Since the curvature formula \eqref{hcurv} is independent of the choice of horizontal lifts, it follows that the formula \eqref{hcurv} for a trivial fibration does not contain a term involving integral over $\partial \Omega$.

The quantity $\kappa$ (as in \eqref{kappa}) appearing in the formula for curvature of $\ch$ is related to the deformation of complex structure on $X_b$. Observe that in the case of trivial fibration (when all the fibers are biholomorphic), for any vector field $\tau$ on $B$ we were able to produce lifts $\xi_\tau$ such that $\dbar \xi_\tau \equiv 0$, so the term containing $\kappa$ is also absent from the curvature formula of a trivial fibration. An interpretation of \eqref{hcurv} is that for families of domains, the curvature of $(E,h)\to X$, the Levi form of $\partial X$ in directions transverse to the fibers of $X\to B$ and the deformation of complex structure on the fibers contribute to the curvature of $\ch$.  
The relation between the deformation of complex structure on $X_b$ and the curvature formula with regard to the contribution of $\kappa$ is better understood when $\pi: X\to B$ is a proper map (see \cite[Theorem 4]{V2021}, \cite[Theorem 1.1]{BerndtssonStrict} for instance). In order to get a better understanding of how the positivity of $\Theta^{\ch}$ relates to the geometry of the family $(E,h)\to X\to B$ it would be desirable to get an exact formula for the square norm of the \emph{second fundamental map}, which is the quantity appearing on the left hand side of \eqref{2ndfunestimateintro} (see Section \ref{iBLS} for precise definition). 

When $\wt{X}$ is a Stein manifold we obtain a couple of exact formulae for the square norm of the second fundamental map, which we record below. However, it seems difficult to determine the positivity of $\Theta^{\ch}$ from either of the formulae. 
\begin{theorem}[= Theorem \ref{2ndfunexact}]\label{2ndfunexactintro}
Let $(E,h)\to X\to B$ be as in Theorem \ref{smoothcurvintro}. Also suppose that the ambient manifold $\wt{X}$ is Stein and that the Neumann operator $N_b^{n,q}$ acting on $E|_{X_b}$-valued $(n,q)$-forms is globally regular for $1\leq q \leq n$. Then  
\begin{align}
    \ipr{P_b^{\perp} L^{1,0}_{\xi_{\sigma_1}}u_1,P_b^{\perp}L^{1,0}_{\xi_{\sigma_2}}u_2}_{\ch_b} &=  -\frac{1}{2}  \ipr{N^{(n,1)}_b\iota_{X_b}^*\dbar\nabla^{1,0}(\xi_{\sigma_1} \lrcorner u_1), \iota_{X_b}^*(\xi_{\sigma_2}\lrcorner \Theta^{E})u_2+\iota_{X_b}^*\nabla^{1,0}(\dbar\xi_{\sigma_2}\lrcorner u_2)}_{\ch_b}  \nonumber \\
    &\quad- \frac{1}{2} \ipr{ \iota_{X_b}^*(\xi_{\sigma_1}\lrcorner \Theta^{E})u_1+\iota_{X_b}^*\nabla^{1,0}(\dbar\xi_{\sigma_1}\lrcorner u_2), N^{(n,1)}_b\iota_{X_b}^*\dbar\nabla^{1,0}(\xi_{\sigma_2} \lrcorner u_2)}_{\ch_b}, \label{2ndfunex}
\end{align}
where $u_i$ are $E$-valued $(n,0)$-forms satisfying \eqref{ufcond} such that $\iota_{X_b}^* u\in \ch_b$ and $\xi_{\sigma_i}$ are horizontal lifts of holomorphic $(1,0)$-vector fields $\sigma_i$ on an open set containing $b$.  
\end{theorem}

\begin{theorem}[= Theorem \ref{2ndfunexact2}]
Let $(E,h)\to X \to B$ be as in Theorem \ref{smoothcurvintro}. Also suppose that the ambient manifold $\wt{X}$ is Stein and that the Neumann operator $N_b^{n,q}$ acting on $E|_{X_b}$-valued $(n,q)$-forms is globally regular for $1\leq q \leq n$. Further assume that for all $b\in B$ the Neumann operator $\mathcal{N}^{(0,2)}_b$ acting on $(0,2)$-forms on $X_b$ is globally regular. Let $u$ be an $E$-valued $(n,0)$-form satisfying \ref{ufcond} such that $\iota_{X_b}^*u\in \ch_b$. 
Then 
\begin{equation*}
   \norm{P_b^{\perp}L^{1,0}_{\xi_\sigma}u}^2_{\ch_b} = \ipr{\alpha, \iota_{X_b}^*(\xi_\sigma\lrcorner \Theta^E)u}_{\ch_b} + \ipr{\dbar \Lambda_{\omega_b}\alpha, \iota_{X_b}^*(\dbar\xi_\sigma \lrcorner u)}_{\ch_b} + \int_{\partial X_b}\left\{\iota_{X_b}^*(\ol{\xi}_\sigma \lrcorner \partial \dbar \rho)\wedge\Lambda_{\omega_b}\alpha ,u   \right\}dS_b, 
\end{equation*}
where $\alpha := \sqi N^{(n,1)}_b \iota_{X_b}^*\dbar \nabla^{1,0} (\xi_\sigma \lrcorner u)$. Here $\sigma$ is a holomorphic $(1,0)$-vector field on a neighborhood of $b$ and $\xi_{\sigma}$ is its primitive horizontal lift. 
\end{theorem}

\subsection{Scope of results}
The main theorems in this paper are Theorem \ref{smoothcurvintro} and Theorem \ref{2ndfunpropintro}. These theorems rely on the technical assumption that $\ch \to B$ is an iBLS field. The following are examples of fibrations $X\to B$ for which this assumption holds. (See Section \ref{finalremarks} for details.)
\begin{enumerate}
\item Let $B\subset \C^m$ be the unit ball, let $\wt{X} = B\times \C^n$ and let $E\to \wt{X}$ be the trivial line bundle with a nontrivial metric. Let $\Omega \subset \C^n$ be a smoothly bounded pseudoconvex domain. Take $X = B\times \Omega \subset \wt{X}$ and $\pi :X\to B$ to be the projection onto the first factor. This scenario was studied in \cite{B2009}. 

\item Let $E\to \wt{X}$ be a hermitian holormophic line bundle over an $n+m$-dimensional complex manifold $\wt{X}$ and let $\wt{\pi}:\wt{X}\to B$ be a holomorphic submersion, where $B\subset \C^m$ is the unit ball. Take $X\subset \wt{X}$ to be the domain $X = \{ \rho <0\}$, where $\rho$ is a smooth real valued function on $\wt{X}$ such that for all $b\in B$ (i) $\rho|_{\wt{\pi}^{-1}(b)}$ is a strictly plurisubharmonic function in a neighborhood of the closure of $X_b := X\cap \wt{\pi}^{-1}(b)$ and (ii) $d\rho(x)\neq 0$ for all $x \in \partial X_b$. Take $\pi:X\to B$ to be the map $\wt{\pi}|_X$. This scenario was studied in \cite{W2017}.

\item Let $\wt{X} = B\times \C^n$ where $B\subset \C^m$ is a domain, let $E\to \wt{X}$ be the trivial line bundle with nontrivial metric and let $\wt{\pi}:\wt{X}\to B$ be the projection onto the first factor. Take $X\subset \wt{X}$ to be the domain $X = \{ \rho <0\}$, where $\rho$ is a smooth real valued function on $\wt{X}$ such that for all $b\in B$ (i) $X_b$ is a bounded domain in $\C^n$ (ii) $\rho|_{\wt{\pi}^{-1}(b)}$ is a plurisubharmonic function and (iii) $d\rho(x)\neq 0$ for all $x \in \partial X_b$. Take $\pi:X\to B$ be the map $\wt{\pi}|_X$.

\item Let $\wt{X} = B\times \C^n$ where $B\subset \C^m$ is a domain, let $E\to \wt{X}$ be the trivial line bundle with nontrivial metric and let $\wt{\pi}:\wt{X}\to B$ be the projection onto the first factor. Take $X\subset \wt{X}$ to a smoothly bounded domain such that the domains $X_b := X\cap \wt{\pi}^{-1}(b) \subset \C^n$ admit good Stein neighborhood bases in the sense of \cite{straube2001}. Take $\pi:X\to B$ be the map $\wt{\pi}|_X$.
\end{enumerate}

\subsubsection{An example with compact fibers}\label{compactexample}
Here we reproduce an example from \cite{V2021} of a fibration $\pi: X\to B$ whose fibers are compact manifolds (without boundary). It can be shown that for any line bundle $(E,h)\to X\to B$ the associated family of Hilbert spaces is an iBLS field when the fibers are manifolds without boundary (see \cite[Proposition 4.15]{V2021}).

Let $\mathbb{H} \subset \C$ denote the upper half plane and let $X = \C\times \mathbb{H}/\sim$ where $(z,b) \sim (z',b')$ if $b=b'$ and there exist $m, n \in \Z$ such that $z-z' = m+ nb$. Then $\pi: X\to \mathbb{H}$ is the map induced by the projection $\C\times \mathbb{H} \to \mathbb{H}$. Thus, the fiber $X_b$ is the torus corresponding to the lattice spanned by $\{1,b\}$. Let $D_a$ be the smooth hypersurface on $X$ obtained from $\{a\}\times \mathbb{H} \subset \C\times \mathbb{H}$ after passing to the quotient. Let $E\to X$ be the line bundle associated to the divisor $D:= D_0- D_{\sqi}$. The divisor $\Delta_b := D\cap X_b$ is trivial if there are $m,n \in \Z$ such that $\sqi = m+nb$ and otherwise is $[0]-[\sqi] \in X_b$. Let $B\subset \mathbb{H}$ denote the discrete set of $b$ for which there are $m,n \in \Z$ such that $\sqi = m+nb$. Then it can be shown that 
\[
\ch_b = 
\begin{cases}
\C \quad \text{if} \quad b\in B \\
\{0\} \quad \text{if} \quad b\in \mathbb{H}\setminus B.
\end{cases}
\]

The geometry of families $\ch \to B$ associated to a holomorphic fibration $\pi:X\to B$ has been studied extensively when the fibers $X_b$ are compact manifolds without boundary. The articles \cite{LS2014} and \cite{V2021} do an excellent job of giving an almost exhaustive list of references on this topic, so we refer the reader to those for the current state of affairs. In \cite{LS2014} the authors develop the formalism of Hilbert fields and study the case of fibrations $X\to B$ when the fibers $X_b$ are open manifolds. The article \cite{V2021} discusses the geometry of families $\ch \to B$ associated to a proper submersion $X\to B$, when $E\to X$ is a vector bundle. In this scenario, the fibers of $\ch \to B$ are all finite dimensional vector spaces, but $\ch \to B$ need not be vector bundle as the example in Section \ref{compactexample} shows.

\subsection{Organization of the paper} In Section 2 we recall definitions regarding abstract Hilbert fields and maps between them. The key idea is that of an iBLS field, and we show that the curvature of an iBLS field is a possibly nonsmooth operator valued $2$-form. Section 3 contains details of the geometric setup that we are interested in. We also obtain identities involving Lie derivatives that will be needed in curvature calculations. In Section 4 we define the Hilbert fields of interest to us and define the curvature of a family of Bergman spaces $\ch$. In this section, we show that the curvature of $\ch$ is independent of various choices involved in defining it. Sections 5 and 6 contain the proofs of the main results. 

\subsection*{Acknowledgements}
I am grateful to my advisor Dror Varolin for numerous discussions relating to this article in particular and complex geometry in general. I cannot thank him enough for his kindness, his constant encouragement in pursuit of all forms of math and for giving me the complete creative freedom to pursue whichever ideas I liked to their logical conclusion. It has been a pleasure working with him.

\section{Abstract Hilbert Fields}\label{absHilb}
In this section, we collect the basic definitions about Hilbert fields. We are mainly interested in the notion of an \emph{iBLS field} and defining its curvature. The main motivation for such a notion comes from the fact (see Section \ref{finalremarks})that under reasonable assumptions, the family of Bergman spaces associated to a holomorphic submersion is an iBLS field. For a more thorough discussion of the geometry of abstract Hilbert fields we refer the reader to \cite{V2021} and \cite{LS2014}. 

\begin{definition}[Hilbert field]
Let $B$ be a smooth manifold. 

\begin{enumerate}[label=(\roman*)]
    \item A \emph{Hilbert field} $p:\cl \to B$ is a set theoretic map whose fibers are Hilbert spaces. Thus, each fiber $\cl_b := p^{-1}(b)$ comes equipped with an inner product $(\cdot, \cdot)_{\cl_b}$.
    
    \item A \emph{section} $\mathfrak{f}$ of a Hilbert field $\cl\to B$ is a set theoretic map $\mathfrak{f}:B\to \cl$ such that $\mathfrak{f}(b) \in \cl_b$ for every $b\in B$.
\end{enumerate}
\end{definition}
In order to talk about differentiable sections we need to equip $\cl$ with more structure.

\begin{definition}
Let $B$ be a smooth manifold and $\cl \to B$ be a Hilbert field. 
\begin{enumerate}[label=(\roman*)]
    \item A \emph{smooth structure} for the Hilbert field $\cl \to B$ is a sheaf $\sC^{\infty}(\cl)$ whose associated presheaf assigns to each open set $U\subset B$ a $\sC^{\infty}(U)$-module of sections $\Gamma(U,\sC^{\infty}(\cl))$ such that for each $b\in B$ the stalk $\sC^{\infty}(\cl)_b$ is mapped by the evaluation-at-$b$ map to a dense subspace of $\cl_b$.
    
    \item The metric $(\cdot,\cdot)_{\cl}$ on $\cl\to B$ is said to be \emph{smooth} if the function $b \mapsto (\mathfrak{f}(b),\mathfrak{g}(b))_{\cl_b}$ belongs to $\sC^{\infty}(U)$ for all $\mathfrak{f, g} \in \Gamma(U, \sC^{\infty}(\cl))$ and every open set $U\subset B$.
    
    \item Let $\cl \to B$ be a Hilbert field with a smooth structure $\sC^{\infty}(\cl)$. A \emph{connection} $\nabla^{\cl}$ for $\cl\to B$ is a map
    $\nabla^{\cl}: \Gamma(U,\sC^{\infty}(T_B\otimes \C))\times \Gamma(U, \sC^{\infty}(\cl)) \to \Gamma(U,\sC^{\infty}(\cl))$ for every open $U\subset B$ satisfying
    \begin{equation}
    \nabla^{\cl}_{f\sigma + \tau}\mathfrak{f} = f\nabla^\cl_\sigma \mathfrak{f} + \nabla^\cl_\tau \mathfrak{f} \quad \text{and} \quad \nabla^\cl_\sigma (f\mathfrak{f}) = (\sigma f)\mathfrak{f} + f\nabla^\cl_\sigma \mathfrak{f} \label{Leibniz}
    \end{equation}
    for all complex vector fields $\sigma, \tau$ on $U$, for all $f \in \sC^{\infty}(U)$ and all $\mathfrak{f} \in \Gamma(U, \sC^{\infty}(\cl))$.
    
    \item A connection $\nabla^{\cl}$ for $\cl\to B$ is said to be \emph{compatible with the metric} if 
    \begin{align}
        \sigma (\mathfrak{f},\mathfrak{g}) = \ipr{ \nabla^\cl_\sigma \mathfrak{f,g}} + \ipr{\mathfrak{f}, \nabla^\cl_{\ol \sigma} \mathfrak{g}} \label{compatibleconnection},
    \end{align}
    for every open set $U\subset B$, all complex vector fields $\sigma$ on $U$ and sections $\mathfrak{f,g} \in \Gamma(U,\sC^{\infty}(\cl))$.
    
    \item A \emph{smooth Hilbert field} is a Hilbert field with a smooth structure, a smooth metric and a metric compatible connection. 
\end{enumerate}
\end{definition}
\begin{remark}
It is possible to develop a theory of smooth structures and  connections for a family of topological vector spaces $\cl \to B$ independently of any metric on $\cl$. One is then led to the notion of a quasi-Hilbert field, as defined in \cite{V2021}. Each fiber $\cl_b$ of a quasi-Hilbert field $\cl \to B$ is a topological vector space, which forgets the metric that gives rise to the topology. We avoid the language of quasi-Hilbert fields, as the language of Hilbert fields is sufficient to describe the results in this paper.
\end{remark}
Pursuing the analogy of finite rank vector bundles and Hilbert fields $\cl \to B$ with a connection $\nabla^{\cl}$ further, we can define the curvature of $\nabla^\cl$ as below. 

\begin{definition}
Let $\cl \to B$ be a Hilbert field with smooth structure $\sC^{\infty}(\cl)$ and let $\nabla^{\cl}$ be a connection for $\cl$. The map $\Theta(\nabla^{\cl}): \Gamma(U, \sC^{\infty}(\Lambda^2T_B\otimes \C))\times \Gamma(U, \sC^{\infty}(\cl)) \to \Gamma(U, \sC^{\infty}(\cl))$ defined by
\begin{align*}
    \Theta(\nabla^{\cl})(\sigma, \tau)\mathfrak{f} = \nabla^{\cl}_\sigma \nabla^{\cl}_\tau \mathfrak{f} - \nabla^{\cl}_\tau \nabla^{\cl}_\sigma \mathfrak{f} - \nabla^{\cl}_{[\sigma,\tau]}\mathfrak{f},
\end{align*}
is called the \emph{curvature} of $\nabla^{\cl}$. Here $U \subset B$ is an open set, $\sigma, \tau \in \Gamma(U, \sC^{\infty}(T_B\otimes \C))$ and $\mathfrak{f}\in \Gamma(U, \sC^{\infty}(\cl)$.
\end{definition}
The curvature of a finite rank vector $E\to B$ is an $\text{End}(E)$-valued $2$-form on $B$. In order to obtain a similar description for the curvature of Hilbert fields with a connection, we introduce the following definitions.

\begin{definition}[Maps of Hilbert fields]
Let $\cl, \cm$ be Hilbert fields over a smooth manifold $B$.
\begin{enumerate}[label=(\roman*)]

    \item A set theoretic map $\ct:\cl \to \cm$ is said to be a \emph{map of Hilbert fields} if it is fiber preserving and linear on fibers. We write $\ct\in \text{Lin}(\cl, \cm)$. Thus, a map of Hilbert fields $\ct:\cl \to \cm$ is a collection $\{\ct_b\}_{b\in B}$ of linear operators $\ct_b \in \text{Lin}(\cl_b,\cm_b)$.
    
    \item A map of Hilbert fields $\ct:\cl\to \cm$ is said to be a \emph{morphism} if $\ct_b$ is a bounded linear operator for every $b\in B$.
    
    \item A map of Hilbert fields $\ct:\cl \to \cm$ is said to be \emph{densely defined} (resp. \emph{closed}, \emph{closable}) if $\ct_b$ is a densely defined (resp. \emph{closed}, \emph{closable}) operator for each $b\in B$.
    
    \item Suppose that $\sC^{\infty}(\cl)$ and $\sC^{\infty}(\cm)$ are smooth structures for $\cl$ and $\cm$ respectively. A Hilbert field map $\ct \in \text{Lin}(\cl,\cm)$ is said to be \emph{smooth} if for all $b\in B$ and all $\mathfrak{f} \in \sC^{\infty}(\cl)_b$ we have $\ct \mathfrak{f}\in \sC^{\infty}(\cm)_b$.
\end{enumerate}
\end{definition}
It follows that a smooth map is densely defined. However, a smooth map need not be a morphism and a morphism need not be smooth. 

\begin{definition}
Let $\cl, \cm \to B$ be Hilbert fields with smooth structures $\sC^{\infty}(\cl)$ and $\sC^{\infty}(\cm)$.

\begin{enumerate}[label=(\roman*)]
    \item A map of sheaves $\mathbb{L}: \sC^{\infty}(\cl)\to \sC^{\infty}(\cm)$ is said to be $\sC^{\infty}$-linear if for all $b\in B$ we have
    \[
    \mathbb{L}(f\mathfrak{f}) = f\mathbb{L}\mathfrak{f} \quad \text{for all } f\in \sC^{\infty}_{B,b} \text{ and all }\mathfrak{f} \in \sC^{\infty}(\cl)_b.
    \]
    
    \item A map of sheaves $\mathbb{L}: \sC^{\infty}(\cl)\to \sC^{\infty}(\cm)$ is said to be tensorial at $b\in B$ if \[
    \mathfrak{f}\in \sC^{\infty}(\cl)_b \quad \text{and} \quad \mathfrak{f}(b) = 0 \implies (\mathbb{L}\mathfrak{f})(b) = 0.
    \]
    We say that $\mathbb{L}$ is \emph{tensorial} if $\mathbb{L}$ is tensorial at every point of $B$.
\end{enumerate}
\end{definition}

 In contrast to the situation of a finite rank vector bundle, the curvature $\Theta^{\cl}_{\sigma\tau}$ of a connection $\nabla^{\cl}$ along complex vector fields $\sigma, \tau \in \Gamma(B, \sC^{\infty}(T_B\otimes \C))$ need not be a morphism in general. However, following lemma says that the curvature $\Theta^{\cl}_{\sigma\tau}$ is tensorial. Moreover, it says that $\Theta^{\cl}_{\sigma\tau}$ is a densely defined closable smooth map from the Hilbert field $\cl$ to itself. The smoothness of $\Theta^{\cl}_{\sigma\tau}$ is a consequence of the definition of the connection $\nabla^\cl$.

\begin{lemma}[{\cite[Lemma 2.2.4]{LS2014}}]\label{LSlemma}
\begin{enumerate}[label=(\arabic*)]
    \item $\Theta^{\cl}_{\sigma\tau}\mathfrak{f}(b)$ depends only on $\sigma(b), \tau(b)$ and $\mathfrak{f}(b)$, hence induces a densely defined operator on $\cl_b$, denoted $\Theta^{\cl}_{\sigma(b)\tau(b)}$.
    
    \item The adjoint of $\Theta^{\cl}_{\sigma(b)\tau(b)}$ is an extension of $-\Theta^{\cl}_{\ol{\sigma}(b)\ol{\tau}(b)}$. In particular, the adjoint is densely defined, and so $\Theta^{\cl}_{\sigma(b)\tau(b)}$ is closable.
\end{enumerate}
\end{lemma}

\begin{definition}[Subfields]
 Let $\cl \to B$ be a Hilbert field.
\begin{enumerate}[label=(\roman*)]
    \item A \emph{Hilbert subfield} $\ch$ of $\cl$ is a Hilbert field $\ch \to B$ such that $\ch_b \subset \cl_b$ is a closed subspace for all $b\in B$, consequently the metric $(\cdot, \cdot)_{\ch_b}$ is the restriction of the metric $(\cdot, \cdot)_{\cl_b}$ to $\ch_b$. 
    
    \item Every Hilbert subfield $\ch \subset \cl$ is equipped with the \emph{orthogonal projector}, i.e., the unique morphism $\cp_{\ch}: \cl \to \ch$ such that $\cp_{\ch, b} \in \text{Lin}(\cl_b,\ch_b)$ is the orthogonal projection onto $\ch_b$ for every $b\in B$.
    
    \item Suppose that $\cl$ is a smooth Hilbert field. A \emph{smooth Hilbert subfield} $\ch$ is a smooth Hilbert field so that the smooth structure $\sC^{\infty}(\ch)$ is the sheaf associated to the presheaf
    \[
    U \mapsto \Gamma(U, \sC^{\infty}(\ch)) := \{ \mathfrak{f}\in \Gamma(U, \sC^{\infty}(\cl)) : \mathfrak{f}(b) \in \ch_b \text{ for all } b\in U\}.
    \]
    
    \item We say that a smooth subfield $\ch \subset \cl$ is \emph{regular} if $\cp_{\ch}:\cl \to \ch$  is a smooth morphism.
\end{enumerate}
\end{definition}
When $\ch \subset \cl$ is a smooth Hilbert subfield, we can also look at the Hilbert subfield $\ch^{\perp} \subset \cl$. In general, $\ch^{\perp}$ need not be a smooth subfield of $\cl$. However if $\ch \subset \cl$ is a regular subfield, then $\ch^{\perp}$ is a regular subfield of $\cl$.

\begin{definition}[Second fundamental map of a smooth subfield]\label{2ndfundef}
Let $\cl \to B$ be a smooth Hilbert field and $\ch \subset \cl$ be a smooth subfield. The map $N^{\cl/\ch}:\Gamma(B,\sC^{\infty}(T_B\otimes \C)\times \Gamma(B, \sC^{\infty}(\ch))\to \Gamma(B, \sC^{\infty}(\cl))$ given by
\[
N^{\cl/\ch}(\sigma)\mathfrak{f} = \nabla^{\cl}_{\sigma}\mathfrak{f} - \nabla^{\ch}_{\sigma}\mathfrak{f}, \quad \text{for } \mathfrak{f} \in \Gamma(B,\sC^{\infty}(\ch)), \sigma \in \Gamma(B, \sC^{\infty}(T_B\otimes \C))
\]
is called the \emph{second fundamental map} of $\ch$  in $\cl$.
\end{definition}
When $\ch \subset \cl$ is a regular subfield, the connection $\nabla^{\cl}$ induces a connection $\nabla^{\ch}$ on $\ch$ as $\nabla^{\ch} := \cp_{\ch}\nabla^{\cl}$. The second fundamental map of the induced connection is 
\begin{equation}\label{secondfunmotivation}
N^{\cl/\ch}(\sigma)\mathfrak{f} = \nabla^{\cl}_{\sigma}\mathfrak{f} - \cp \nabla^{\cl}_{\sigma}\mathfrak{f} = \cp^{\perp}\nabla^{\cl}_{\sigma}\mathfrak{f}.
\end{equation}

The quantity on the right hand side of equation \eqref{secondfunmotivation} is analogous to the familiar second fundamental form one encounters in the geometry of finite dimensional vector bundles, which explains the terminology `second fundamental map' in Definition \ref{2ndfundef}. In the next proposition we compare the curvature of a connection $\nabla^{\cl}$ with that of the induced connection on a regular subfield $\ch \subset \cl$. This proposition is well known in the case of finite rank holomorphic vector bundles (due to P. Griffiths), and is present in \cite{B2009} in the setting of Hilbert bundles. The proof for the finite rank vector bundle situation carries over to our setting with minor modifications.

\begin{proposition}\label{propgauss}
Let $\mathcal{L}\to B$ be a smooth Hilbert field with a metric compatible connection $\nabla^{\cl}$ and let $\mathcal{H}$ be a regular subfield of $\cl$. Let $\nabla^{\ch} := \cp \nabla^{\cl}$ denote the induced connection on $\ch$. Then the curvatures $\Theta^{\ch}$ of $\nabla^{\ch}$ and $\Theta^{\cl}$ of $\nabla^{\cl}$ are related as
\begin{equation}\label{pregauss}
    \Theta^\mathcal{H}_{\sigma_1 \sigma_2}\mathfrak{f} = \cp \Theta^{\mathcal{L}}_{\sigma_1\sigma_2}\mathfrak{f} - \cp \left(\nabla^{\mathcal{L}}_{\sigma_1}\cp^{\perp}\nabla^{\mathcal{L}}_{\sigma_2}\mathfrak{f} - \nabla^{\mathcal{L}}_{\sigma_2}\cp^{\perp}\nabla^{\mathcal{L}}_{\sigma_1}\mathfrak{f} \right),
\end{equation}
where $U\subset B$ is an open set, $\sigma_1, \sigma_2 \in \Gamma(U,\sC^{\infty}(T_B\otimes \C))$ and $\mathfrak{f}\in \Gamma(U, \sC^{\infty}(\ch))$. Here $\cp^{\perp}:\cl \to \cl$ is the smooth map given by $\cp_{b}^{\perp}= \text{Id}_{\cl_b}-\cp_{b}$ on fibers. Moreover, for sections $\mathfrak{f,g} \in \Gamma(U, \sC^{\infty}(\ch))$ and vector fields $\sigma_1, \sigma_2 \in \Gamma(U,\sC^{\infty}(T_B\otimes \C))$ we have
\begin{equation}\label{progauss2}
    \ipr{\Theta^\mathcal{H}_{\sigma_1 \sigma_2}\mathfrak{f}, \mathfrak{g}} = \ipr{\Theta^\mathcal{L}_{\sigma_1 \sigma_2}\mathfrak{f}, \mathfrak{g}} + \ipr{\cp^{\perp}\nabla^{\cl}_{\sigma_2}\mathfrak{f},\cp^{\perp}\nabla^{\cl}_{\ol\sigma_1}\mathfrak{g}} - \ipr{\cp^{\perp}\nabla^{\cl}_{\sigma_1}\mathfrak{f},\cp^{\perp}\nabla^{\cl}_{\ol\sigma_2}\mathfrak{g}}. 
\end{equation}
\end{proposition}

\begin{proof}
Let $\mathfrak{g} \in \Gamma(U,\sC^{\infty}(\cl))$. For a smooth map $\ct:\cl\to \cl$ define the map $\nabla^{\cl}_{\sigma_i}\ct:\cl \to \cl$ by 
\begin{equation}\label{Tder}
    \ipr{ \nabla^{\cl}_{\sigma_i}\ct}(\mathfrak{g}) = \nabla^{\cl}_{\sigma_i}(\ct \mathfrak{g}) - \ct\ipr{\nabla^{\cl}_{\sigma_i}\mathfrak{g}}.
\end{equation}
Thus, $\nabla^{\cl}_{\sigma_i}\ct$ is a smooth densely defined map of Hilbert fields. Applying \eqref{Tder} to the smooth map $\cp$ twice, we get
\begin{align}
    \ipr{ \nabla^{\cl}_{\sigma_i} \cp^2} (\mathfrak{g}) = \nabla^\cl_{\sigma_i} (\cp \cp \mathfrak{g}) - \cp \cp\ipr{\nabla^\cl_{\sigma_i} \mathfrak{g}} &= \ipr{\nabla^\cl_{\sigma_i}\cp}(\cp\mathfrak{g}) + \cp\nabla^\cl_{\sigma_i}(\cp \mathfrak{g}) - \cp \cp\ipr{\nabla^\cl_{\sigma_i}\mathfrak{g}} \nonumber \\
    &= \ipr{\nabla^\cl_{\sigma_i}\cp}(\cp \mathfrak{g}) + \cp\ipr{\nabla^\cl_{\sigma_i}\cp}(\mathfrak{g}). \label{P2der}
\end{align}
Since $\cp^2 = \cp$,
\begin{equation}\label{P3Der}
    \cp\ipr{\nabla^\cl_{\sigma_i}\cp}(\mathfrak{g}) = \ipr{\nabla^\cl_{\sigma_i}\cp}(\mathfrak{g}) - \ipr{\nabla^\cl_{\sigma_i}\cp}(\cp \mathfrak{g}) = \ipr{\nabla^\cl_{\sigma_i}\cp}(\cp^{\perp} \mathfrak{g}) = -\cp \nabla^\cl_{\sigma_i}\ipr{\cp^{\perp}\mathfrak{g}},
\end{equation}
where the last equality follows by applying \eqref{Tder} to the map $\cp$ and noting that $\cp\cp^{\perp}\mathfrak{g}=0$. We now calculate 
\begin{align*}
    \nabla^{\ch}_{\sigma_1}\nabla^{\ch}_{\sigma_2}(\mathfrak{f}) &=\cp \nabla^\cl_{\sigma_1}(\cp\nabla^\cl_{\sigma_2}\mathfrak{f}) \\
    &= \cp \left( \ipr{\nabla^\cl_{\sigma_1} \cp}\ipr{\nabla^\cl_{\sigma_2}\mathfrak{f}} + \cp \nabla^\cl_{\sigma_1}\ipr{\nabla^\cl_{\sigma_2}\mathfrak{f}}\right) \\
    &= \cp \ipr{\nabla^\cl_{\sigma_1} \cp}\ipr{\nabla^\cl_{\sigma_2}\mathfrak{f}} + \cp^2 \nabla^\cl_{\sigma_1}\ipr{\nabla^\cl_{\sigma_2}\mathfrak{f}} \\
    &= -\cp\nabla^\cl_{\sigma_1}\ipr{\cp^{\perp}\nabla^\cl_{\sigma_2}\mathfrak{f}} + \cp \nabla^\cl_{\sigma_1}\nabla^\cl_{\sigma_2}\mathfrak{f},
\end{align*}
where we used equation \eqref{Tder} for the smooth map $\cp$ to get the second equality, and \eqref{P3Der} to get the final equality. Using this identity (and the one obtained by interchanging the roles of $\sigma_1$ and $\sigma_2$) we get
\begin{align*}
    \Theta^{\ch}_{\sigma_1\sigma_2}\mathfrak{f} &= \nabla^\ch_{\sigma_1}\nabla^\ch_{\sigma_2}\mathfrak{f} - \nabla^\ch_{\sigma_2}\nabla^\ch_{\sigma_1}\mathfrak{f} - \nabla^\ch_{[\sigma_1,\sigma_2]}\mathfrak{f} \\
    &= -\cp\nabla^\cl_{\sigma_1}\ipr{\cp^{\perp}\nabla^\cl_{\sigma_2}\mathfrak{f}} + \cp \nabla^\cl_{\sigma_1}\nabla^\cl_{\sigma_2}\mathfrak{f} +\cp\nabla^\cl_{\sigma_2}\ipr{\cp^{\perp}\nabla^\cl_{\sigma_1}\mathfrak{f}} - \cp \nabla^\cl_{\sigma_2}\nabla^\cl_{\sigma_1}\mathfrak{f} - \cp\nabla^\cl_{[\sigma_1,\sigma_2]}\mathfrak{f} \\
    &= \cp \left( \nabla^\cl_{\sigma_1}\nabla^\cl_{\sigma_2}\mathfrak{f} - \nabla^\cl_{\sigma_2}\nabla^\cl_{\sigma_1}\mathfrak{f} -\nabla^\cl_{[\sigma_1,\sigma_2]}\mathfrak{f} \right) - \cp \left(\nabla^\cl_{\sigma_1}\cp^{\perp}\nabla^\cl_{\sigma_2}\mathfrak{f} - \nabla^\cl_{\sigma_2}\cp^{\perp}\nabla^\cl_{\sigma_1}\mathfrak{f} \right) \\
    &= \cp \Theta^{\mathcal{L}}_{\sigma_1\sigma_2}\mathfrak{f} - \cp \left(\nabla^{\mathcal{L}}_{\sigma_1}\cp^{\perp}\nabla^{\mathcal{L}}_{\sigma_2}\mathfrak{f} - \nabla^{\mathcal{L}}_{\sigma_2}\cp^{\perp}\nabla^{\mathcal{L}}_{\sigma_1}\mathfrak{f} \right).
\end{align*}
Equation \eqref{progauss2} follows by taking inner products with $\mathfrak{g}$ on both sides of \eqref{propgauss} and noting that 
\begin{align*}
0 = \sigma_i\ipr{\cp^{\perp}\nabla^{\cl}_{\sigma_j}\mathfrak{f},\mathfrak{g}} &= \ipr{\nabla^{\cl}_{\sigma_i}\cp^{\perp}\nabla^{\cl}_{\sigma_j}\mathfrak{f},\mathfrak{g}} + \ipr{\cp^{\perp}\nabla^{\cl}_{\sigma_j}\mathfrak{f},\nabla^{\cl}_{\ol\sigma_j}\mathfrak{g}} \\
&= \ipr{\nabla^{\cl}_{\sigma_i}\cp^{\perp}\nabla^{\cl}_{\sigma_j}\mathfrak{f},\mathfrak{g}} + \ipr{\cp^{\perp}\nabla^{\cl}_{\sigma_j}\mathfrak{f},\cp^{\perp}\nabla^{\cl}_{\ol\sigma_j}\mathfrak{g}}
\end{align*}
\end{proof}

\begin{definition}[Holomorphic Hilbert fields]
Let $B$ be a complex manifold and $\cl \to B$ be a Hilbert field with smooth structure.

\begin{enumerate}[label=(\roman*)]
    \item An \emph{almost complex structure} for $\cl$ is a map 
    \[
    \dbar^{\cl}: \Gamma(B, T_B^{0,1})\times \Gamma(B, \sC^{\infty}(\cl)) \to \Gamma(B, \sC^{\infty}(\cl))
    \]
    satisfying the Leibniz rule 
    \begin{equation}\label{Leibnizdbar} 
    \dbar^{\cl}(f\mathfrak{f})(\sigma) = \dbar f(\sigma)\mathfrak{f}+ f\dbar^{\cl}\mathfrak{f}(\sigma)
    \end{equation}
    for all $f\in \sC^{\infty}(B), \sigma \in \Gamma(B, T^{0,1}_B)$ and $\mathfrak{f} \in \Gamma(B, \sC^{\infty}(\cl))$. The $\dbar$ appearing on the right  hand side of equation \eqref{Leibnizdbar} is the $\dbar$ operator for the complex manifold $B$.
    
    \item We define $(\dbar^{\cl})^2: \Gamma(B, \Lambda^2T^{0,1}_B)\times \Gamma(B, \sC^{\infty}(\cl))\to \Gamma(B, \sC^{\infty}(\cl))$ by
    \[
    (\dbar^\cl)^2\mathfrak{f}(\sigma, \tau) = \sigma(\dbar^{\cl}\mathfrak{f}(\tau)) - \tau(\dbar^{\cl}\mathfrak{f}(\sigma)) - \dbar^{\cl}\mathfrak{f}([\sigma,\tau]).
    \]
    We say that $\dbar^{\cl}$ is \emph{involutive} or \emph{integrable} if $(\dbar^{\cl})^2\mathfrak{f}(\sigma,\tau)=0$ for all $\mathfrak{f} \in \Gamma(U,\sC^{\infty}(\cl))$, all $\sigma, \tau \in \Gamma(U,T^{0,1}_B)$ and all open sets $U\subset B$.
    
    \item A smooth section $\mathfrak{f}$ of the Hilbert field $\cl\to B$ with almost complex structure $\dbar^{\cl}$ is said to be \emph{holomorphic} on an open set $U\subset B$ if $\dbar^{\cl}\mathfrak{f}(\sigma)=0$ for all $\sigma \in \Gamma(U, T^{0,1}_B)$. 
    
    \item An \emph{(almost) holomorphic Hilbert field} is a Hilbert field with smooth structure and an (almost) complex structure. We say that $(\cl, \dbar^{\cl})$ is a \emph{holomorphic Hilbert field} if $\dbar^\cl$ is integrable.
    
    \item Let $(\cl,\dbar^{\cl})$ be an (almost) holomorphic Hilbert field. We call $\cl$ a \emph{Berndtsson-Lempert-Sz\H{o}ke field} (abbreviated as BLS field) if $\cl$ is also a smooth Hilbert field such that $\nabla^{\cl(0,1)}=\dbar^{\cl}$, i.e.,
    \[
    \nabla^{\cl}_{\sigma}\mathfrak{f} = \dbar^{\cl}\mathfrak{f}(\sigma)
    \]for all $\mathfrak{f}\in \Gamma(U, \sC^{\infty}(\cl))$, $\Gamma(U, T^{0,1}_B)$ and open sets $U\subset B$. We call $\nabla^{\cl}$ a \emph{BLS Chern connection}.
    
    \item A BLS field is said to be \emph{integrable} if its almost complex structure is integrable.

\end{enumerate}
\end{definition}
\begin{remark}
    The definitions of a BLS field and a smooth Hilbert field in \cite{V2021} are slightly more general than what we have given. It can be shown that the two sets of definitions are equivalent. 
\end{remark}

\begin{definition}[BLS subfields]\label{blssub}
Let $(\cl, \dbar^{\cl})\to B$ be an almost holomorphic Hilbert field.

\begin{enumerate}[label=(\roman*)]
    \item An (almost) holomorphic Hilbert field $(\ch, \dbar^{\ch})\to B$ is said to be an \emph{(almost) holomorphic Hilbert subfield} of $\cl$ if 
    \begin{enumerate}[label=(\alph*)]
        \item $\ch_b$ is a closed subspace of $\cl_b$ for every $b\in B$ and $(\cdot, \cdot)_{\cl}|_{\ch} = (\cdot, \cdot)_{\ch}$
        
        \item $\dbar^{\cl}|_{\ch} = \dbar^{\ch}$ i.e., $\dbar^{\cl}\mathfrak{f}(\sigma) = \dbar^{\ch}\mathfrak{f}(\sigma)$ for all $\mathfrak{f}\in \Gamma(U, \sC^{\infty}(\ch))$, all $\sigma \in \Gamma(U, T^{0,1}_B)$ and every open set $U\subset B$. In particular, this means that whenever $\mathfrak{f}\in \sC^{\infty}(\ch)_b$ we have $\dbar^{\cl}\mathfrak{f}(\sigma)\in \sC^{\infty}(\ch)_b$.
    \end{enumerate}
    
    \item An (almost) holomorphic Hilbert subfield $\ch \subset \cl$ is said to be a \emph{BLS subfield} if $\ch$ is a regular subfield of $\cl$, i.e. $\cp: \cl \to \ch$ given by orthogonal projection on each fiber is a smooth map. 
\end{enumerate}
\end{definition}
The next proposition shows that if $\ch$ is an integrable BLS subfield of a BLS field $\cl$ then the curvature of the induced BLS Chern connection for $\ch$ is an operator valued $(1,1)$-form even when $\cl$ is not integrable.

\begin{proposition}\label{11formcrit}
Let $B$ be a complex manifold. Let $\cl \to B$ be a BLS field and let $\ch \to B$ be an integrable BLS subfield of $\cl$. Then, the curvature $\Theta^{\ch}$ of the Hilbert field $\ch$ is an operator-valued $(1,1)$-form. Moreove, we haver
\begin{align}
    \ipr{\Theta^{\ch}_{\sigma\ol{\tau}}\mathfrak{f,g}} &= \ipr{\Theta^{\cl}_{\sigma \ol\tau}\mathfrak{f,g}} - \ipr{\cp^{\perp} \nabla^{\cl}_{\sigma}\mathfrak{f}, \cp^{\perp}\nabla^{\cl}_{\tau}\mathfrak{g}} \label{gaussformula}
\end{align}
for all $\mathfrak{f,g} \in \Gamma(U, \sC^{\infty}(\ch)), \sigma, \tau \in \Gamma(U, T^{1,0}_B)$ and all open sets $U\subset B$.
\end{proposition}

\begin{proof}
We want to show that $\Theta^\ch_{\sigma\tau}\mathfrak{f} = \Theta^\ch_{\ol{\sigma}\ol{\tau}}\mathfrak{f} = 0$. Since $\ch \subset \cl$ is a BLS subfield it follows from Definition \ref{blssub} (i) that 
\begin{equation}\label{11formcrit1}
    \cp^{\perp}\nabla^\cl_{\ol{\sigma}}\mathfrak{g} = \cp^{\perp}\dbar^{\cl}\mathfrak{g}(\ol{\sigma}) = 0.
\end{equation} 
Since $\ch$ is an integrable holomorphic Hilbert field,
\begin{equation}\label{11formcrit2}
    \Theta^\cl_{\ol{\sigma}\ol{\tau}}\mathfrak{g} = (\dbar^{\cl})^2\mathfrak{g}(\ol{\sigma}, \ol{\tau}) = (\dbar^{\ch})^2\mathfrak{g}(\ol{\sigma}, \ol{\tau}) =0.
\end{equation}
To see that $\cp\Theta^{\cl}$ is an operator-valued $(1,1)$-form use \eqref{compatibleconnection} repeatedly to get
\[
0 = (\sigma\tau - \tau \sigma -[\sigma, \tau])\ipr{\mathfrak{f,g}} = \ipr{\Theta^\cl_{\sigma\tau}\mathfrak{f},\mathfrak{g}} + \ipr{\mathfrak{f}, \Theta^\cl_{\ol{\sigma}\ol{\tau}}\mathfrak{g}}. 
\]
Now $\Theta^\cl_{\ol{\sigma}\ol{\tau}}\mathfrak{g}=0$ by condition \eqref{11formcrit2}, so we get that $\ipr{\Theta^\cl_{\sigma\tau}\mathfrak{f},\mathfrak{g}}_b=0$ which combined with the density of germs of sections $\mathfrak{g}(b)$ in $\ch_b$ implies that $\cp \Theta^\cl_{\sigma(b)\tau(b)}\mathfrak{f}(b)=0$ for all $b\in U$.

We claim that $\cp\nabla^\cl_{\xi}\cp^{\perp}\nabla^\cl_{\eta}\mathfrak{f}=0$ whenever $\xi$ and $\eta$ are vector fields of same type. We only need to show this when both $\xi, \eta$ are vector fields of type $(1,0)$, since the other case follows directly from \eqref{11formcrit1}. Using \eqref{compatibleconnection} again, we get
\begin{align*}
    0 = \sigma \ipr{\cp^{\perp}\nabla^\cl_{\tau}\mathfrak{f},\mathfrak{g}} &= \ipr{ \nabla^{\cl}_{\sigma}\cp^{\perp}\nabla^{\cl}_{\tau}\mathfrak{f,g}} + \ipr{ \cp^{\perp}\nabla^{\cl}_{\tau}\mathfrak{f}, \nabla^{\cl}_{\ol\sigma}\mathfrak{g}} \\
    &= \ipr{ \nabla^{\cl}_{\sigma}\cp^{\perp}\nabla^{\cl}_{\tau}\mathfrak{f,g}} + \ipr{ \cp^{\perp}\nabla^{\cl}_{\tau}\mathfrak{f}, \cp^{\perp} \nabla^{\cl}_{\ol\sigma}\mathfrak{g}} \\
    &= \ipr{ \nabla^{\cl}_{\sigma}\cp^{\perp}\nabla^{\cl}_{\tau}\mathfrak{f,g}},
\end{align*}
where the last equality follows from \eqref{11formcrit1}. By the density of the germs $\mathfrak{g}(b)$ in $\ch_b$ we see that the the section $\cp \nabla^{\cl}_{\sigma}\cp^{\perp}\nabla^{\cl}_{\tau}\mathfrak{f}$ vanishes at the point $b$ for every $b\in U$. The fact that $\Theta^{\ch}$ is an operator-valued $(1,1)$-form now follows from equation \eqref{pregauss}. Equation \eqref{gaussformula} now follows from \eqref{progauss2}.
\end{proof}
Observe that in the formula \eqref{gaussformula} there are no derivatives of the map $\cp$. Thus this equation can be used to define the curvature of a not-necessarily-smooth Hilbert subfield $\ch$ of a smooth Hilbert field $\cl$. We give a precise definition of such a Hilbert field $\ch$ in the next subsection.

\subsection{iBLS fields}\label{iBLS}
In order to define an iBLS field, we need to introduce more definitions. 

\begin{definition}
Let $\cl\to B$ be a Hilbert field with smooth structure and let $\ch \to B$ be a Hilbert subfield of $\cl$.

\begin{enumerate}[label=(\roman*)]
    \item A \emph{substalk bundle} $\mathcal{S}$ of $\sC^{\infty}(\cl)$ is a choice of $\C$-vector space $\mathcal{S}_b\subset \sC^{\infty}(\cl)_b$ for each $b\in B$.
    
    \item For each $b\in B$, let $\wt{\ch}_{\cl,b}$ denote the germs of smooth sections of $\cl$ at $b$ that take values in $\ch_b$. More precisely, let  
    \[
    \wt{\ch}_{\cl,b} := \{ \mathfrak{f} \in \sC^{\infty}(\cl)_b : \mathfrak{f}(b)\in \ch_b\} 
    \]
    be the $\C$-vector subspace of the stalk $\sC^{\infty}(\cl)_b$. Let $\wt{\ch}_{\cl}\to B$ be the family whose fiber over $b$ is $\wt{\ch}_{\cl,b}$. The family $\wt{\ch}_{\cl}\to B$ is called the \emph{maximal substalk bundle} of $\ch$.
    
    \item If $\mathcal{S}$ is a substalk bundle so that $\mathcal{S}_b\subset \wt{\ch}_{\cl,b}$ for every $b\in B$ then $\mathcal{S}$ is said to be a \emph{tuning} of $\ch$ (in $\cl$). Alternatively, we say that $(\ch, \mathcal{S})$ is a \emph{tuned} Hilbert field.
    
    \item A tuning $\Sigma$ of $\ch \to B$ is said to formalize $\ch$ if for every $b\in B$ 
    \begin{enumerate}[label=(\alph*)]
        \item $\Sigma_b(b) \subset \wt{\ch}_{\cl,b}(b)$ is a dense subspace and
        
        \item If $\mathfrak{f}\in \Sigma_b$ and $\mathfrak{f}(b)=0$ then $(\nabla^{\cl}_\sigma \mathfrak{f})(b) \in \mathfrak{\ch}_b$ for every $\sigma \in \sC^{\infty}(T_B\otimes \C)_b$.
    \end{enumerate}
    We say that $(\ch, \Sigma)$ is a \emph{formal tuning} or that $\ch$ is a \emph{formal subfield} (of $\cl$).
    
    \item Suppose that $(\cl,\dbar^{\cl}) \to B$ is an almost holomorphic Hilbert field, let $\ch \subset \cl$ be a Hilbert subfield and let $\mathcal{S}$ be a tuning of $\ch$. We say that $(\ch, \mathcal{S})$ is an \emph{infinitesimally almost holomorphic Hilbert subfield} of $\cl$ if for all $b\in B$ we have 
    \[
    \dbar^{\cl}\mathfrak{f}(\sigma)\in \wt{\ch}_{\cl,b} \text{ for all } \sigma \in \sC^{\infty}(T^{0,1}_B)_b, \text{ and for all } \mathfrak{f}\in \mathcal{S}_b.
    \]
    We say that $(\ch, \mathcal{S})$ is an \emph{infinitesimally holomorphic subfield} of $\cl$ if for all $b\in B$ we have
    \[
    \dbar^{\cl}\dbar^{\cl}\mathfrak{f}(\sigma, \tau)(b) =0, \text{ for all } \sigma, \tau \in \sC^{\infty}(T^{0,1}_B)_b \text{ and } \mathfrak{f} \in \mathcal{S}_b.
    \]
    
\end{enumerate}
\end{definition}
\begin{remark}
We emphasize that neither $\wt{\ch}_{\cl}$ nor $\Sigma$ is a sheaf, and we give them no additional structure. We equip $\ch$ with a tuning $\Sigma$ to be able to define \emph{curvature} of $\ch$ when $\ch$ may not have a smooth structure. Condition $(b)$ in the definition of a formal tuning $\Sigma$ is our replacement for the smoothness of the Hilbert field morphism $\cp$, as it is automatically satisfied when the morphism $\cp$ is smooth. Indeed, suppose that $\ch$ is a regular subfield and let $\mathfrak{f}\in \sC^{\infty}(\ch)_b$ be such that $\mathfrak{f}(b)=0$. Then for any $\mathfrak{g} \in \sC^{\infty}(\cl)_b$ and $\sigma \in \sC^{\infty}(T_B\otimes \C)_b$, we have
\[
\ipr{\cp^{\perp}\nabla^{\cl}_\sigma \mathfrak{f,g}} = \ipr{\nabla^{\cl}_\sigma \mathfrak{f},\cp^{\perp}\mathfrak{g}} = \sigma \ipr{\mathfrak{f},\cp^{\perp}\mathfrak{g}} - \ipr{\mathfrak{f}, \nabla^{\cl}_{\ol\sigma}\cp^{\perp}\mathfrak{g}} = - \ipr{\mathfrak{f}, \nabla^{\cl}_{\ol\sigma}\cp^{\perp}\mathfrak{g}}.
\]
Evaluating both sides of the equation above at $b$ shows that $(\cp^{\perp}\nabla^{\cl}_\sigma \mathfrak{f})(b) \in \cl_b$ is orthogonal to a dense subspace, so must be $0$. 
\end{remark}

We can define the curvature of a formal subfield $\ch$ using the Gauss-Griffiths formulae \eqref{progauss2} or \eqref{gaussformula} provided some additional conditions are satisfied. As a first step toward this, we define a second fundamental map of a formal subfield below.

\begin{definition}[Second fundamental map of a formal subfield]
Let $\ch \subset \cl$ be a formal subfield of a smooth Hilbert field $\cl$ and let $(\ch, \Sigma)$ be a formal tuning. The second fundamental map of $\ch$ is a collection of linear operators 
\[
N_b : \sC^{\infty}(T_{B}\otimes \C)_b \times \ch_b \to \ch_b^{\perp} \text{ defined by } N_b(\sigma)(\mathfrak{f}(b)) = (\cp^{\perp}\nabla_{\sigma}\mathfrak{f})(b)
\]
for all $\sigma \in \sC^{\infty}(T_{B}\otimes \C)_b$ and $\mathfrak{f}\in \Sigma_b$. Note that for all $\sigma \in \sC^{\infty}(T_B\otimes \C)_b$ we have $\text{Dom}(N_b(\sigma))= \Sigma_b(b)$ where 
\[
\Sigma_b(b) := \{ \mathfrak{f}(b) \in \ch_b : \mathfrak{f} \in \Sigma_b\}. 
\]
\end{definition}
The second fundamental map is well-defined because $(\ch,\Sigma)$ is formal. If $\mathfrak{f}_1, \mathfrak{f}_2 \in \Sigma_b$ are such that $\mathfrak{f}_1(b)-\mathfrak{f}_2(b)=0$ then 
\[
N_b(\sigma)(\mathfrak{f}_1(b)-\mathfrak{f}_2(b)) = \left(\cp^{\perp}\nabla_{\sigma}^{\cl}(\mathfrak{f}_1-\mathfrak{f}_2)\right)(b) = 0.
\]

\begin{definition}[A sesquilinear form]\label{sesQ}
Let $\ch \subset \cl$ be a formal subfield of a smooth Hilbert field $\cl$ and let $(\ch, \Sigma)$ be a formal tuning.
For $\sigma, \tau \in \sC^{\infty}(T_B)_b$ define a sesquilinear form $Q_{\sigma\tau}(\cdot,\cdot)_b$ on $\ch_b$ with $\text{Dom}(Q_{\sigma\tau}(\cdot, \cdot)_b) = \Sigma_b(b)$ by 
   \[
    Q_{\sigma\tau}(\mathfrak{f}(b), \mathfrak{g}(b))_b = \ipr{\Theta^{\cl}_{\sigma\tau}\mathfrak{f,g}}_{\ch_b} - \ipr{N_b(\sigma)\mathfrak{f},N_b(\ol{\tau})\mathfrak{g}}_{\ch_b}, 
    \]
    for $\mathfrak{f, g} \in \Sigma_b(b).$ 
\end{definition}

\begin{definition}[iBLS field]\label{iBLSdef}
Let $\ch \to B$ be a Hilbert field. An \emph{infinitesimally BLS field} (or iBLS field) is a triple $(\ch, \cl, \Sigma)$ where $\cl \to B$ is a BLS field called the \emph{ambient BLS field} (for $\ch$) and $\Sigma \subset \wt{\ch}_{\cl}$ is a tuning such that
\begin{enumerate}[label=(\alph*)]
     \item For every $b\in B$ the subspace $\wt{\ch}_{\cl,b}(b):= \{\mathfrak{f}(b): \mathfrak{f} \in \wt{\ch}_{\cl,b}\}$ is dense in $\ch_b$.
    
    \item $(\ch,\Sigma)$ is a formal tuning.
    
    \item $(\ch, \Sigma)$ is an infinitesimally almost holomorphic Hilbert subfield of $\cl$.
    
    \item For all $b\in B$, for all $\sigma, \tau \in \sC^{\infty}(T_B\otimes \C)_b$ and all $\mathfrak{f}\in \Sigma_b$ the anti-linear functional $Q_{\sigma\tau}(\mathfrak{f}(b),\cdot)_b$ is continuous on its domain $\Sigma_b(b)$, where $Q_{\sigma\tau}(\cdot,\cdot)_b$ is as in definition \ref{sesQ}. 
\end{enumerate}
\end{definition}
\begin{remark}
    Note that by \ref{iBLSdef} (c) the second fundamental map of $\ch$ is a collection of linear operators $N_b: \sC^{\infty}(T^{1,0}_B)_b\times \ch_b \to \ch_b^{\perp}$ given by
    \[
    N_b(\sigma)(\mathfrak{f}(b)) = (\cp^{\perp}\nabla^{\cl(1,0)}_\sigma\mathfrak{f})(b). 
    \]
    For $\sigma, \tau \in \sC^{\infty}(T_B\otimes \C)_b$ the sesquilinear form $Q_{\sigma\tau}(\cdot, \cdot)_b$ is densely defined as conditions (a) and (b) imply that both $N_b(\sigma)$ and $N_b(\tau)$ are densely defined.
\end{remark}


\begin{definition}[Curvature of an iBLS field]
Let $(\ch,\cl, \Sigma)$ be an iBLS field. For $\sigma, \tau \in \sC^{\infty}(T_B\otimes \C)_b$ the curvature $\Theta^{\ch}$ of $\ch$ relative to $\cl$ is defined as
\begin{equation}\label{iBLScurv}
\ipr{ \Theta^{\ch}_{\sigma(b)\tau(b)}\mathfrak{f}(b),\mathfrak{g}(b)}_{\ch_b} := \ipr{\Theta^{\cl}_{\sigma\tau}\mathfrak{f,g}}_{\ch_b} - \ipr{N_b(\sigma)\mathfrak{f},N_b(\ol{\tau})\mathfrak{g}}_{\ch_b}. 
\end{equation}
Thus, the curvature $\Theta^{\ch}_{\sigma(b)\tau(b)}$ is a densely defined operator with $\text{Dom}(\Theta^{\ch}_{\sigma(b)\tau(b)}) = \Sigma_b(b)$.
\end{definition}
\begin{remark}
This method of defining an operator is familiar from the theory of unbounded operators. For each $\mathfrak{f}(b) \in \Sigma_b(b)$ the densely defined anti-linear functional $ Q_{\sigma\tau}(\mathfrak{f}(b),\cdot)_b$ is continuous. Then the Riesz representation theorem gives an element $\Theta^{\ch}_{\sigma(b)\tau(b)}\mathfrak{f}(b) \in \ch_b$ such that 
\[
\ipr{\Theta^{\ch}_{\sigma(b)\tau(b)}\mathfrak{f}(b),\mathfrak{g}(b)}_{\ch_b} = Q_{\sigma\tau}(\mathfrak{f}(b), \mathfrak{g}(b))_b.
\]
\end{remark}

\begin{proposition}\label{iBLS11formcrit}
 Let $(\ch, \cl, \Sigma)$ be an iBLS field. Suppose that $(\ch, \Sigma)$ is an infinitesimally holomorphic Hilbert subfield. Then $\Theta^{\ch}$ has $(1,1)$-coefficients, i.e., for all $b\in B$ and all $\mathfrak{f}\in \Sigma_b$, the quantity $(\Theta^{\ch}_{\sigma\tau}\mathfrak{f})(b)=0$ whenever $\sigma, \tau \in \sC^{\infty}(T_B\otimes \C)_b$ are vector fields of the same type.   
\end{proposition}

\begin{proof}
Since $(\ch, \cl,\Sigma)$ is an iBLS field, $N_b(\eta)=0$ for all $b\in B$ whenever $\eta \in \sC^{\infty}(T^{0,1}_B)$. Thus, the second term on the right hand side of \eqref{iBLScurv} vanishes whenever $\sigma, \tau$ are vector fields of same type. Since $(\ch, \Sigma)$ is an infinitesimally holomorpihc Hilbert subfield, for all $b\in B$, for all $\sigma, \tau \in \sC^{\infty}(T^{1,0}_B)_b$ and for all $\mathfrak{f} \in \Sigma_b$ we have $(\Theta^{\cl}_{\ol{\sigma}\ol{\tau}}\mathfrak{f})(b)=0$. As in the proof of Proposition \ref{11formcrit}, we have
\[
0 = \ipr{\Theta^{\cl}_{\sigma \tau}\mathfrak{f,g} } + \ipr{\mathfrak{f}, \Theta^{\cl}_{\ol{\sigma}\ol{\tau}}\mathfrak{g}}.
\]
Note that $\Sigma_b(b)$ is dense in $\ch_b$: $\Sigma_b(b)$ is dense in $\wt{\ch}_{\cl,b}(b)$ because $(\ch, \Sigma)$ is formal and $\wt{\ch}_{\cl,b}(b)$ is dense in $\ch_b$ because $(\ch, \Sigma,\cl)$ is an iBLS field. Thus, if $(\Theta^{\cl}_{\ol{\sigma}\ol{\tau}}\mathfrak{g})(b)=0$ for all $\mathfrak{g} \in \Sigma_b$ it follows that $(\Theta^{\cl}_{\sigma\tau}\mathfrak{f})(b) = 0$. 
The proposition now follows from equation \eqref{iBLScurv}.
\end{proof}

We end this subsection with an analogue of Lemma \ref{LSlemma} for iBLS fields. 

 
\begin{lemma}\label{iBLSlemma}
Let $(\ch,\Sigma, \cl)$ be an iBLS field. Let $\sigma, \tau \in \sC^{\infty}(T_B\otimes \C)_b$ and let $\mathfrak{f} \in \Sigma_b$. Then  
\begin{enumerate}[label=(\arabic*)]
    \item $(\Theta^{\ch}_{\sigma\ol\tau}\mathfrak{f})(b)$ depends only on $\sigma(b), \tau(b)$ and $\mathfrak{f}(b)$, hence induces a densely defined operator on $\ch_b$, denoted $\Theta^{\ch}_{\sigma(b)\ol\tau(b)}$.
    
    \item The adjoint of $\Theta^{\ch}_{\sigma(b)\tau(b)}$ is an extension of $\Theta^{\ch}_{\ol{\tau}(b)\ol{\sigma}(b)}$. In particular, the adjoint is densely defined, and so $\Theta^{\ch}_{\sigma(b)\tau(b)}$ is closable.
\end{enumerate}
\end{lemma}

\begin{proof}
In view of \eqref{iBLScurv}, the first statement follows from statement (1) of Lemma \ref{LSlemma} applied to the smooth Hilbert field $\cl$ and the fact that $(\ch,\Sigma)$ is a formal tuning. Using metric compatibility \eqref{compatibleconnection} of the connection $\nabla^{\cl}$, for $\mathfrak{f,g} \in \Sigma_b$ we get
\begin{align*}
    0 = (\sigma \tau - \tau\sigma - [\sigma,\tau])\ipr{\mathfrak{f,g}} &= \ipr{\Theta^{\cl}_{\sigma\tau}\mathfrak{f,g}} - \ipr{\mathfrak{f}, \Theta^{\cl}_{\ol{\tau}\ol{\sigma}}\mathfrak{g}} \\
    &= \ipr{\Theta^{\cl}_{\sigma\tau}\mathfrak{f,g}} -\ipr{\cp^{\perp}\nabla^{\cl}_{\sigma^{1,0}}\mathfrak{f},\cp^{\perp}\nabla^{\cl}_{\ol{\tau}^{1,0}}\mathfrak{g}}+\ipr{\cp^{\perp}\nabla^{\cl}_{\sigma^{1,0}}\mathfrak{f},\cp^{\perp}\nabla^{\cl}_{\ol{\tau}^{1,0}}\mathfrak{g}}- \ipr{\mathfrak{f}, \Theta^{\cl}_{\ol{\tau}\ol{\sigma}}\mathfrak{g}} \\
    &= \ipr{\Theta^{\ch}_{\sigma\tau}\mathfrak{f,g}} - \ipr{\mathfrak{f},\Theta^{\ch}_{\ol{\tau}\ol\sigma}\mathfrak{g}}.
\end{align*}
The second statement now follows by evaluating the above equation at $b$ and obvserving that $\Sigma_b(b)$ is dense in $\ch_b$.
\end{proof}

\section{The geometric setup}
In order to equip a family of Bergman spaces with the structure of an iBLS field, we need to describe the geometry of the holomorphic submersion that gives rise to this family. We do so in this section. 

 \subsection{Notation} \label{notation} We fix the following notation for the rest of the paper. Let $\wt{\pi}: \wt{X}\to B$ be a holomorphic submersion of K\"ahler manifolds, where $\text{dim}_{\C} \wt{X} = n+m$ and $\text{dim}_{\C} B= m$. Let $(E,h)\to \wt{X}$ be a holomorphic line bundle with a smooth metric $h$. Let $X\subset \wt{X}$ be a smoothly bounded domain such that for each $b\in B$ the fiber $X_b := \wt{\pi}^{-1}(b)\cap X$ is a bounded pseudoconvex domain with smooth boundary in the $n$-dimensional K\"ahler manifold $\wt{\pi}^{-1}(b)$. Let $\pi: X\to B$ be the submersion $\wt{\pi}|_{X}$. We call $\pi:X\to B$ a \emph{family of pseudoconvex domains}.

 
 Assume that $ X$ has a defining function $\rho$ and denote by $\rho_b$ the restriction of $\rho$ to $\wt{\pi}^{-1}(b)$, so that $\rho_b$ is a defining function for the pseudoconvex domain $X_b$. We will denote the K\"ahler form on $\wt{X}$ by $\omega$ and let $\omega_b = \iota_{X_b}^*\omega$, so $dV_{\omega_b}$ denotes the volume form on $X_b$. To simplify some of our fomrulas we assume that $\abs{d\rho_b}_{\omega_b} = 1$ on $\partial X_b$ and we denote the induced volume form on $\partial X_b$ by $dS_b$, i.e., it is a form that satisfies $d\rho_b\wedge dS_b = dV_{\omega_b}$ on $\partial X_b$.

 
 All hermitian holomorphic line bundles in this article are equipped with the corresponding Chern connection, so we interchangeably use $\dbar = \nabla^{0,1}$ for the $(0,1)$-part of the connection. Also, we use the notation $\nabla$ for the Chern connection of the line bundle $E\to X$, as well as that of the line bundle $E|_{X_b}\to X_b$ and hope this does not cause confusion. 
 
 We denote by $\{\cdot,\cdot \}$ the inner product induced on $E$-valued $(p,q)$-forms by the metrics $\omega$ and $h$. We also use another pairing of $E$-valued forms on $X$, denoted by $\hwedge{\cdot}{\cdot}$. Suppose $u$ is an $E$-valued $(p,q)$-form and $v$ is an $E$-valued $(r,s)$-form on $X$. Let $e$ be a frame for $E$ so that $u$ and $v$ are locally written as $u = \vp \otimes e$, and $v = \psi\otimes e$. Then the pairing $\hwedge{\cdot}{\cdot}$ is given by 
 \[
 \hwedge{u}{v} = \vp\wedge \ol\psi h(e,e).
 \]

\subsection{Horizontal distributions}

Let $\pi:X\to B$ be a holomorphic submersion as in Section \ref{notation}. A \emph{horizontal distribution} for the submersion $\pi$ is a subbundle $\theta$ of the tangent bundle $T_{\ol{X}}$ so that 
\begin{enumerate}
    \item $d\pi|_{\theta}$ maps each fiber $\theta_x$ isomorphically onto $T_{B,\pi(x)}$ (i.e., $d\pi: \theta \to \pi^*T_B$ is an isomorphism of vector bundles),
    
    \item $\theta$ is tangent to the boundary of $X$, i.e., $\theta|_{\partial X} \subset T_{\partial X}$.
\end{enumerate}  
Since $\pi$ is a holomorphic map, $d\pi$ respects the splitting of $\theta \otimes \C$ and $T_B\otimes \C$ into subbundles of type $(1,0)$ and $(0,1)$, i.e., $d\pi(\theta^{1,0}_x)= T^{1,0}_{B,\pi(x)}$ for all $x\in X$. Also $\theta^{1,0}|_{\partial X} \subset T^{1,0}_{\partial X}$ since $\theta|_{\partial X} \subset T_{\partial X}$, so $\theta^{1,0}$ is tangent to the boundary of $X$. Let $\sigma$ be a $(1,0)$-vector field on $B$. A \emph{horizontal lift} of $\sigma$ with respect to the distribution $\theta^{1,0}$ is a section $\xi_{\sigma}^{\theta}$ of $\theta^{1,0} \to \ol{X}$ such that $d\pi(\xi^{\theta}_{\sigma}) =\sigma$.

\subsubsection{Twisted Lie derivatives}
Let $X$ be a K\"ahler manifold as above. For a complex vector field $\eta$ and a complex form $\vp$ on $X$ we define the Lie derivative of $\vp$ along the vector field $\eta$ as
\begin{equation}\label{cplxlie}
L_{\eta} \vp := L_{\text{Re }\eta}\vp + \sqi L_{\text{Im }\eta}\vp,
\end{equation}
where the Lie derivatives appearing on the right hand side of the \eqref{cplxlie} are the usual Lie derivatives. It follows from \eqref{cplxlie} that Cartan's formula
\[
L_{\eta}\vp = d(\eta\lrcorner \vp) + \eta\lrcorner d\vp
\]
holds for complex forms. We extend this formula to define the twisted Lie derivative (which we still denote $L_\eta$) of an $E$-valued $(p,q)$-form $u$ along the complex vector field $\eta$ as 
\[
L_{\eta} u  := \nabla (\eta\lrcorner u) + \eta\lrcorner \nabla u.
\]
We also define the twisted $(1,0)$ Lie derivative $L^{1,0}_{\eta}$ and the twisted $(0,1)$ Lie derivative $L^{0,1}_{\eta}$ of an $E$-valued form $u$ along the complex vector field $\eta$ as
\begin{align*}
    L^{1,0}_{\eta}u &:= \nabla^{1,0} (\eta\lrcorner u) + \eta\lrcorner \nabla^{1,0} u, \\
    L^{0,1}_{\eta}u &:= \nabla^{0,1} (\eta\lrcorner u) + \eta\lrcorner \nabla^{0,1} u.
\end{align*}
Note that $L^{1,0}_{\eta}u$ (resp. $L^{0,1}_{\eta}u$) has same bidegree as $u$ iff $\eta$ is a $(1,0)$-vector field (resp. $(0,1)$-vector field).

\subsubsection{Sections of the twisted relative canonical bundle}
Recall that we have a holomorphic submersion $\pi:X\to B$ of K\"ahler manifolds. Since $\pi$ is a submersion, the map $\pi^*$ maps $T^{1,0*}_{B,\pi(x)}$ injectively into $T^{1,0*}_{X,x}$. Thus, we may identify the image $\pi^*T^{1,0*}_B$ with a subbundle of $T^{1,0*}_X$. This gives rise to a short exact sequence of vector bundles
\begin{equation}\label{relcan}
0\to \pi^*T^{1,0*}_B \to T^{1,0*}_X \to T^{1,0*}_{X/B}\to 0,
\end{equation}
where the first map is the inclusion of $\pi^*T^{1,0*}_B$ as a subbundle. The top exterior power of the quotient vector bundle $ T^{1,0*}_{X/B}$ is called the \emph{relative canonical bundle}, denoted $K_{X/B}$.

Let $(t^1,\cdots, t^m)$ be local coordinates on an open set $U \subset B$, so that $\pi^*dt^1, \cdots, \pi^*dt^m$ is a frame for $\pi^*T^{1,0*}_B$ on the open set $\pi^{-1}(U)$. From the short exact sequence \eqref{relcan} we see that sections of $T^{1,0*}_{X/B}$ are equivalence classes of $(1,0)$-forms on $X$ with respect to the relation $\sim$ given by $\alpha_1 \sim \alpha_2$ if 
\[
\alpha_1 - \alpha_2 = f_1\pi^*dt^1 + \cdots + f_m\pi^*dt^m, \quad \text{for some functions }f_j \in \sC^{\infty}(\pi^{-1}(U)),
\]
or equivalently, if $\iota_{X_b}^*(\alpha_1 - \alpha_2) = 0$ for all $b\in U$. Since $\text{dim}_{\C}(X) - \text{dim}_{\C}(B)=n$, sections of $K_{X/B}$ are equivalence classes of $(n,0)$-forms where $\vp_1 \sim \vp_2$ if
\[
\vp_1 - \vp_2 = \psi_1\wedge\pi^*dt^1 + \cdots \psi_m\wedge \pi^*dt^m, \quad \text{for some }\psi_j \in \Gamma(\pi^{-1}(U),\Lambda^{n-1,0}_X),
\]
or equivalently, if $\iota_{X_b}^*(\vp_1 - \vp_2) = 0$ for all $b \in B$.

Now fix a point $b\in B$. Recall that the conormal bundle $N_{X/X_b}^*$ to $X_b$ consists of covectors in $T^{1,0*}_X\big|_{X_b}$ that annihilate vectors tangent to $X_b$. Thus there is a natural identification of vector bundles $\pi^*T^{1,0*}_B\big|_{X_b}$ and $N^*_{X/X_b}$ as a subbundle of $T^{1,0*}_X\big|_{X_b}$ consisting of $(1,0)$-covectors that annihilate vectors tangent to $X_b$. Restricting the sequence \eqref{relcan} to the submanifold $X_b$ and taking determinants gives 
\[
K_{X/B}\big|_{X_b}\otimes \text{det}(N^*_{X/X_b})  \cong K_X\big|_{X_b}.
\]
Tensoring both sides of the last equation with determinant $\text{det}(N_{X/X_b})$ of the normal bundle to $X_b$ and noting that $\text{det}(N^*_{X/X_b})\otimes \text{det}(N_{X/X_b})$ is the trivial bundle, we get 
\[
K_{X/B}\big|_{X_b} \cong K_{X/B}\big|_{X_b}\otimes \text{det}(N^*_{X/X_b})\otimes \text{det}(N_{X/X_b})  \cong K_X\big|_{X_b}\otimes \text{det}(N_{X/X_b}) \cong K_{X_b},
\]
where the last isomorphism follows from the adjunction formula (see \cite[Proposition 2.2.17]{huyb}).

The same argument works with $E$-valued $(n,0)$-forms and sections of the twisted relative canonical bundle $E\otimes K_{X/B}$, which we state as a lemma.

\begin{lemma}\label{twistedrelcan}
Let $E\to X$ be a holomorphic hermitian line bundle and let $\pi:X\to B$ be a surjective holomorphic submersion as in Section \ref{notation}. Then for each $b\in B$ there is an isomorphism  
\[
\iota_b: (K_{X/B}\otimes E)|_{X_b} \to K_{X_b}\otimes E|_{X_b}
\]
such that for an open set $U$ containing $b$ and a twisted relative canonical section $f$ over $\pi^{-1}(U)$ we have $\iota_b(f)= \iota_{X_b}^*u$, where $u$ is an $E$-valued $(n,0)$-form representing $f$.
\end{lemma}

\subsection{Calculations involving Lie derivatives}

We obtain some identities involving Lie derivatives that are needed to equip the field of Bergman spaces with an iBLS structure in Section \ref{concHilb}. The following lemma says that the twisted Lie derivatives (when taken with respect to a metric compatible connection) obey the same Leibniz rule as the usual Lie derivatives.

\begin{lemma} 

Let $Z$ be a complex manifold (with or without boundary) and $E\to Z$ be a holomorphic line bundle with hermitian metric $h$. Let $u \in \Gamma(Z, E\otimes \Lambda^{p,q}_Z)$ and $v\in \Gamma(Z, E\otimes \Lambda^{r,s}_Z)$ be $E$-valued forms. Then for a complex vector field $\xi$ on $Z$ we have

\begin{align}
    L_{\xi} \hwedge{u}{v} &= \hwedge{L_{\xi}u}{v} + \hwedge{u}{L_{\overline{\xi}}v}, \label{lieder} 
\end{align}

\end{lemma}

\begin{proof}
The proof is by direct calculation as below.
\begin{align*}
    L_{\xi} \hwedge{u}{v} &= \xi \lrcorner d \hwedge{u}{v} + d \left( \xi \lrcorner \hwedge{u}{v} \right) \\
                                &= \xi \lrcorner \hwedge{\nabla u}{v} + (-1)^{p+q} \xi \lrcorner \hwedge{u}{\nabla v}  + d \hwedge{\xi \lrcorner u}{v} + (-1)^{p+q} d \hwedge{u}{\overline{\xi} \lrcorner v} \\
                                &= \hwedge{\xi \lrcorner \nabla u}{v} + (-1)^{p+q+1}\hwedge{\nabla u}{\overline{\xi} \lrcorner v} + (-1)^{p+q} \hwedge{\xi \lrcorner u}{\nabla v} + \hwedge{u}{\ol\xi \lrcorner \nabla v} \\
                                & + \hwedge{\nabla (\xi \lrcorner u)}{v} + (-1)^{p+q-1} \hwedge{\xi \lrcorner u}{\nabla v} + (-1)^{p+q}\hwedge{\nabla u}{\overline{\xi} \lrcorner v} + \hwedge{u}{\nabla(\overline{\xi} \lrcorner v)}\\
                                &= \hwedge{L_{\xi}u}{v} + \hwedge{u}{L_{\overline{\xi}}v}.
\end{align*}
\end{proof}



    
    

\begin{proposition}\label{dbarPL}
Let $\iota_Y: Y \hookrightarrow Z$ be an $n$-dimensional submanifold of a complex manifold $Z$ and $(E,h)\to Z$ a hermitian holomorphic line bundle. Let $u$ be an $E$-valued $(n,0)$-form and let $\xi$ be a $(1,0)$ vector field on $Z$. Then
\begin{equation}\label{2ndfund}
    \dbar \iota_Y^* \left(L^{1,0}_{\xi}u\right) = \iota_Y^* \left( \dbar\xi \lrcorner \nabla^{1,0}u - (\xi\lrcorner \Theta^E)\wedge u  - \nabla^{1,0}(\dbar\xi \lrcorner u) + L^{1,0}_{\xi}\dbar u \right).
\end{equation}
\end{proposition}
It is to be understood that the $\dbar$ appearing on the left hand side of \eqref{2ndfund} is the $\dbar$-operator acting on $E|_Y$-valued forms on $Y$, while the $\dbar$ appearing on the right hand side is the $\dbar$-operator acting on $E$-valued forms on $Z$. We will continue to abuse the notation for $\dbar$ in this manner throughout this paper. 

\begin{proof}
The proof is a straightforward calculation:
\begin{align*}
    \dbar \iota_Y^* \left( L^{1,0}_{\xi}u\right) &= \iota_Y^* \left( \dbar L^{1,0}_{\xi} u\right) \\
                                                 &= \iota_Y^* \left( \dbar \nabla^{1,0}(\xi \lrcorner u) + \dbar (\xi \lrcorner \nabla^{1,0}u) \right) \\
                                                 &= \iota_Y^* \left( \dbar \nabla^{1,0}(\xi \lrcorner u) + \dbar\xi \lrcorner \nabla^{1,0}u - \xi \lrcorner \dbar\nabla^{1,0} u \right) \\
                                                 &= \iota_Y^* \left( \Theta^E(\xi \lrcorner u) - \nabla^{1,0}\dbar(\xi \lrcorner u) + \dbar\xi \lrcorner \nabla^{1,0}u - \xi \lrcorner (\Theta^E\wedge u) + \xi\lrcorner \nabla^{1,0}\dbar u \right) \\ 
                                                 &= \iota_Y^* \left( - (\xi\lrcorner \Theta^E)\wedge u -\nabla^{1,0} (\dbar \xi \lrcorner u - \xi\lrcorner \dbar u) + \dbar\xi \lrcorner \nabla^{1,0}u + \xi \lrcorner \nabla^{1,0}\dbar u\right) \\
                                                 &= \iota_Y^* \left( \dbar\xi \lrcorner \nabla^{1,0}u - (\xi\lrcorner \Theta^E)\wedge u -\nabla^{1,0} (\dbar \xi \lrcorner u) + \nabla^{1,0}(\xi \lrcorner \dbar u) + \xi \lrcorner \nabla^{1,0}\dbar u \right) \\
                                                 &= \iota_Y^* \left( \dbar\xi \lrcorner \nabla^{1,0}u - (\xi\lrcorner \Theta^E)\wedge u -\nabla^{1,0} (\dbar \xi \lrcorner u) + L^{1,0}_{\xi}\dbar u \right).
\end{align*}        
\end{proof}

\subsubsection{Applications to families of domains}

\begin{proposition}\label{metcomp}
Let $\pi:X \to B$ be a family of pseudoconvex domains. Assume that $X$ is a domain in a larger complex manifold $\widetilde{X}$ and $\pi$ is the restriction of a holomorphic submersion $\widetilde{X}\to B$. Let $(E,h) \to \widetilde{X}$ be a holomorphic line bundle with a smooth metric $h$ and let $u, v$ be $E$-valued $(n,0)$-forms that are smooth up to the boundary of $X$. Let $\tau$ be a $(1,0)$-vector field on $B$ and let $\xi_{\tau}$ be a horizontal lift of $\tau$ to $X$. Then we obtain the following identities
\begin{align}
    \tau \int_{X_b} \hwedge{u}{v} &= \int_{X_b} \hwedge{L^{1,0}_{\xi_{\tau}}u}{v} + \int_{X_b} \hwedge{u}{L^{0,1}_{\overline{\xi}_{\tau}}v} \label{metcomp10} \\
    \ol{\tau}  \int_{X_b} \hwedge{u}{v} &= \int_{X_b} \hwedge{L^{0,1}_{\ol{\xi}_{\tau}}u}{v} + \int_{X_b} \hwedge{u}{L^{1,0}_{\xi_{\tau}}v} \label{metcomp01}
\end{align}
\end{proposition}

\begin{proof}
Write $\tau = \tau^1 + \sqi \tau^2$ and $\xi_{\tau} = \xi_{\tau}^1 + \sqi \xi_{\tau}^2$ for real vector fields $\tau^i$ and $\xi_{\tau}^i$. Since $\pi$ is holomorphic, we have $d\pi(\xi_{\tau}^i) = \tau^i$. Moreover, since $\xi_{\tau}$ is tangent to the boundary of $X$, so are $\xi^i_{\tau}$. Denote by $\Psi^i_t$ the flow of the vector field $\xi_{\tau}^i$. Since the form $\hwedge{u}{v}$ is smooth up to the boundary of $X$, we may differentiate under the integral sign to get
\begin{align*}
    \tau^i \int_{X_b} \hwedge{u}{v} &= \frac{\partial}{\partial t}\Big|_{t=0} \int_{\Psi^i_t(X_b)} \hwedge{u}{v} \\
                                    &= \frac{\partial}{\partial t}\Big|_{t=0} \int_{X_b} (\Psi^i_t)^*\hwedge{u}{v} = \int_{X_b} L_{\xi_{\tau}^i} \hwedge{u}{v}.
\end{align*}
Thus, using the fact that $L_{\xi_{\tau}} = L_{\xi_{\tau}^1} +\sqi L_{\xi_{\tau}^2}$ and the identity \eqref{lieder} we get
\begin{equation*}
    \tau \int_{X_b} \hwedge{u}{v} = \int_{X_b} L_{\xi_{\tau}} \hwedge{u}{v} = \int_{X_b} \hwedge{L_{\xi_{\tau}}u}{v} + \int_{X_b} \hwedge{u}{L_{\overline{\xi}_{\tau}}v}.
\end{equation*}
Note that $L^{1,0}_{\overline{\xi}_{\tau}}v = 0$ since $\overline{\xi}_{\tau}$ is a vector field of type $(0,1)$. Thus the integral of the form $\hwedge{u}{L^{1,0}_{\overline{\xi}_{\tau}}v}$ vanishes. Since the form $\hwedge{L^{0,1}_{\xi_{\tau}}u}{v}$ has bidegree $(n-1,n+1)$, its integral also vanishes. This gives us \eqref{metcomp10}.

Using identity \eqref{lieder} again, we see that
\begin{equation*}
    \ol{\tau} \int_{X_b} \hwedge{u}{v} = \int_{X_b} L_{\ol{\xi}_{\tau}} \hwedge{u}{v} = \int_{X_b} \hwedge{L_{\ol{\xi}_{\tau}}u}{v} + \int_{X_b} \hwedge{u}{L_{\xi_{\tau}}v}.
\end{equation*}
As before, $L^{1,0}_{\ol{\xi}_{\tau}}u=0$ since $\xi_{\tau}$ is a $(1,0)$-vector field and the integral of $\hwedge{u}{L^{0,1}_{\xi_{\tau}}v}$ is zero for reasons of bidegree. Thus, we get \eqref{metcomp01}.
\end{proof}

\begin{corollary}
Let $(E,h)\to X\to B$, $\tau, u$ and $v$ be as in Proposition \ref{metcomp}. Let $\xi^1_{\tau}, \xi^2_{\tau}$ be horizontal lifts of $\tau $ to $X$. Let $w_\tau$ be the vertical vector field $w_\tau = \xi^1_\tau - \xi^2_\tau$. Then 
\begin{align}
    0 &= \int_{X_b} \hwedge{L^{1,0}_{w_\tau}u}{v} + \hwedge{u}{L^{0,1}_{\ol{w}_\tau}v}, \label{zerolie10} \\
    0 &= \int_{X_b} \hwedge{L^{0,1}_{\ol{w}_\tau}u}{v} + \hwedge{u}{L^{1,0}_{w_\tau}v}. \label{zerolie01}
\end{align}
\end{corollary}

\begin{proof}
Using the identity \eqref{metcomp10} twice we get
\begin{align*}
    \tau \int_{X_b} \hwedge{u}{v} &= \int_{X_b} \hwedge{L^{1,0}_{\xi^i_{\tau}}u}{v} + \int_{X_b} \hwedge{u}{L^{0,1}_{\overline{\xi}^i_{\tau}}v}
\end{align*}
for $i=1,2$. Subtracting the two equations gives us \eqref{zerolie10}. Interchanging $u$ and $v$ and taking complex conjugates in \eqref{zerolie10} gives us \eqref{zerolie01}.
\end{proof}

\section{Concrete Hilbert Fields}\label{concHilb}
In this section we equip a family of Bergman spaces with the structure of an iBLS field. Recall that an iBLS field is a triple $(\ch, \cl, \Sigma)$ where $\cl$ is the ambient BLS field, and $(\ch, \Sigma)$ is a formal tuning for $\ch$. First, we introduce the ambient BLS field $\cl^{\theta}$.

\subsection{The Hilbert field of square integrable sections}\label{concL}
Let $\pi: X\to B$ a family of pseudoconvex domains and $(E,h)\to X$ be a hermitian holomorphic line bundle as in Section \ref{notation}. Let $\cl\to B$ be the Hilbert field whose fiber $\cl_b$ over a point $b$ of $B$ is given by
\[
\cl_b = \ol{\left\{f\in \Gamma(X_b, \sC^{\infty}(K_{X_b}\otimes E|_{X_b})) : (f,f)_b = \sqi^{n^2} \int_{X_b} \hwedge{f}{f} < \infty \right\}},
\]
the Hilbert space completion of the smooth twisted relative canonical sections with respect to the norm $(\cdot, \cdot)_b$. We equip $\cl \to B$ with the structure of a BLS field as below.

\begin{enumerate}
    \item \emph{The sheaf $\sC^{\infty}(\cl).$} In view of Lemma \ref{twistedrelcan} we define the sections of the sheaf $\sC^{\infty}(\cl)$ over an open set $U$ to be 
    \[
    \Gamma(U,\sC^{\infty}(\cl)) := \Gamma\ipr{\ol{\pi^{-1}(U)}, \sC^{\infty}(K_{X/B}\otimes E)},
    \]
    where the sections of $K_{X/B}\otimes E$ over $\pi^{-1}(U)$ on the right are smooth up to the boundary of $X$. Given a section $f\in \Gamma\ipr{\ol{\pi^{-1}(U)}, \sC^{\infty}(K_{X/B}\otimes E)}$, we write $\mathfrak{f} = \mathfrak{i}(f)$ to say that the smooth section $\mathfrak{f}$ of $\cl$ is \emph{induced} by $f$. Conversely given a smooth section $\mathfrak{f} \in \Gamma(U,\sC^{\infty}(\cl))$, we write $f = \mathfrak{a(f)}$ to say that the twisted relative canonical section $f$ is \emph{associated} to $\mathfrak{f}$. The $\sC^{\infty}(U)$-module structure on $\Gamma(U, \sC^{\infty}(\cl)$ is the obvious one - for $r\in \sC^{\infty}(U)$ and $\mathfrak{f}\in \Gamma(U, \sC^{\infty}(\cl)$ 
    \[
    r\mathfrak{f} := \mathfrak{i}((\pi^*r) \mathfrak{a(f)}).
    \]

    \item \emph{The operator $\dbar^{\cl^{\theta}}$.} Fix a horizontal distribution $\theta \subset T_X$. Define the operator $\dbar^{\cl^{\theta}}$ by
    \[
    \dbar^{\cl^{\theta}}\mathfrak{f}(\ol{\sigma}) = \mathfrak{i}(L^{0,1}_{\ol{\xi}^{\theta}_{\sigma}}\mathfrak{a(f)})
    \]
    for $\mathfrak{f} \in \Gamma(U,\sC^{\infty}(\cl)$ and $\sigma \in \Gamma(U, T^{1,0}_B)$.
    From now on, we denote the Hilbert field $\cl$ by $\cl^{\theta}$ to emphasize the choice of the horizontal distribution $\theta \subset T_X$ used to define the operator $\dbar^{\cl^{\theta}}$.
        
    \item \emph{A connection on $\cl^{\theta}$.} 
    Let $\mathfrak{f} \in \Gamma(U,\sC^{\infty}(\cl^{\theta}))$ and $\sigma \in \Gamma(U, T^{1,0}_B)$. We define the $(1,0)$-part of the connection $\nabla^{\cl^\theta (1,0)}$ by 
    \[
    \nabla^{\cl^\theta (1,0)}_{\sigma}\mathfrak{f} = \mathfrak{i}\ipr{L^{1,0}_{\xi^\theta_\sigma}\mathfrak{a(f)}}.
    \]
    If a complex vector field $\tau$ on $U$ is written $\tau = \tau^{1,0}+\tau^{0,1}$ as a sum of its $(1,0)$ and $(0,1)$ components then we define
    \[
    \nabla^{\cl^\theta}_{\tau}\mathfrak{f} := \nabla^{\cl^\theta (1,0)}_{\tau^{1,0}}\mathfrak{f} + \dbar^{\cl^\theta}\mathfrak{f}(\tau^{0,1}).
    \]

\end{enumerate}

\begin{proposition}
With the smooth structure given by the sheaf $\sC^{\infty}(\cl^{\theta})$ and the connection $\nabla^{\cl^\theta}$ as above the almost holomorphic Hilbert field $(\cl^{\theta}, \dbar^{\cl^\theta})$ is a BLS field.
\end{proposition}

\begin{proof}
The density of the image of the stalk $\sC^{\infty}(\cl^\theta)_b$ in $\cl^{\theta}_b$ under the evaluation-at-$b$ map follows from the density of compactly supported smooth sections of $K_{X_b}\otimes E|_{X_b}\to X_b$ in $\cl_b$. 

We work in a small enough open set $U\subset B$ so that $(t^1, \cdots, t^m)$ are local coordinates for $B$ over $U$. Let $\mathfrak{f,g}\in \Gamma(U, \sC^{\infty}(\cl^{\theta}))$ and $\sigma \in \Gamma(U, T^{1,0}_B)$. We check that the almost complex structure $\dbar^{\cl^\theta}$ is well-defined, i.e. independent of the choice of a twisted $(n,0)$-form representing $\mathfrak{a(f)}$. Toward this end, suppose that $u_1$ and $u_2$ are two $(n,0)$-forms that represent $\mathfrak{a(f)}$. Then $u_1-u_2= \gamma_j \wedge\pi^*dt^j$ for some $E$-valued $(n-1,0)$-forms $\gamma_j$. If we write $\sigma = \sigma^j\frac{\partial}{\partial t^j}$ in these coordinates then 
\begin{align*}
    L^{0,1}_{\ol{\xi}^\theta_\sigma}(u_1-u_2) &= L^{0,1}_{\ol{\xi}^\theta_\sigma}\gamma_j \wedge \pi^*dt^j.
\end{align*}
This shows that $L^{0,1}_{\ol{\xi}^\theta_\sigma}(u_1-u_2)$ lies in the ideal generated by the $(1,0)$-forms $\{\pi^*dt^j\}$, so $L^{0,1}_{\ol{\xi}^\theta_\sigma} u_1$ and $L^{0,1}_{\ol{\xi}^\theta_\sigma} u_2$ represent the same twisted relative canonical section.
To check that $\dbar^{\cl^\theta}\mathfrak{f}(\ol{\sigma})$ is tensorial in the argument $\ol{\sigma}$, note that if $r\in \sC^{\infty}(U)$ then $(\pi^*r)\ol{\xi}^{\theta}_{\sigma}$ is the horizontal lift of the $(0,1)$-vector field $r\ol{\sigma}$ with respect to the distribution $\theta$. If $u$ is a twisted $(n,0)$-form representing $\mathfrak{a(f)}$ then
\begin{align*}
    \dbar^{\cl^\theta}\mathfrak{f}(\ol{\sigma}) = \mathfrak{i}\left(L^{0,1}_{\ol{\xi}^\theta_{\ol{r}\sigma}} u\right) = \mathfrak{i}\left( (\pi^*r)(\ol{\xi}^\theta_{\sigma}\lrcorner \dbar u)\right) = \mathfrak{i}\left( (\pi^*r)\mathfrak{a}(\dbar^{\cl^\theta}\mathfrak{f}(\ol\sigma))\right) = r\dbar^{\cl^\theta}\mathfrak{f}(\ol\sigma).
\end{align*}
To check the Leibniz rule, for $r \in \sC^{\infty}(U)$ we have
\begin{align*}
    \dbar^{\cl^\theta}(r\mathfrak{f})(\ol\sigma) = \mathfrak{i}\left( L^{0,1}_{\ol{\xi}^{\theta}_\sigma}\mathfrak{a(f)} \right) = \mathfrak{i}\left( \ol{\xi}^{\theta}_\sigma\lrcorner \dbar((\pi^*r)u) \right) &= \mathfrak{i}\left( \pi^*(\ol{\sigma} r)u + (\pi^*r) \ol{\xi}^{\theta}_\sigma\lrcorner\dbar u \right) = \ol{\sigma}r \mathfrak{f} + r\dbar^{\cl^\theta}\mathfrak{f}(\ol\sigma).
\end{align*}
The smoothness of $b\mapsto (\mathfrak{f},\mathfrak{g})_b$ follows from the fact that $\mathfrak{a(f)}$, $\mathfrak{a(g)}$ and the metric $h$ for $E$ are smooth up to the boundary of $X$. Thus, the metric for $\cl^\theta$ is smooth. 
We verify that the $(1,0)$-part of the connection $\nabla^{\cl^\theta}$ is well-defined. Indeed, let $u_1$ and $u_2$ be two $(n,0)$-forms that represent $\mathfrak{a(f)}$ so that $u_1-u_2= \gamma_j \wedge\pi^*dt^j$. Then
\begin{align*}
    L^{1,0}_{\xi^\theta_\sigma}(u_1-u_2) &= L^{1,0}_{\xi^\theta_\sigma}\gamma_j \wedge \pi^*dt^j + \gamma_j\wedge L^{1,0}_{\xi^\theta_\sigma}(\pi^*dt^j) =  L^{1,0}_{\xi^\theta_\sigma}\gamma_j \wedge \pi^*dt^j + \gamma_j\wedge \pi^*\partial \sigma^j.
\end{align*}
This shows that $L^{1,0}_{\xi^\theta_\sigma}(u_1-u_2)$ lies in the ideal generated by the $(1,0)$-forms $\{\pi^*dt^j\}$ and so $L^{1,0}_{\xi^\theta_\sigma} u_1$ and $L^{1,0}_{\xi^\theta_\sigma} u_2$ represent the same twisted relative canonical section. It can be verified that $\nabla^{\cl^\theta(1,0)}$ satisfies the Leibniz rule in a way similar to the verification of the Leibniz rule for $\dbar^{\cl^\theta}$. To check that the quantity $\nabla^{\cl^\theta(1,0)}_{\sigma}\mathfrak{f}$ is tensorial in the argument $\sigma$ it suffices to show that $L^{1,0}_{\xi^\theta_{r\sigma}} u = \mathfrak{a}(\nabla^{\cl^\theta(1,0)}_{r\sigma}\mathfrak{f})$ and $(\pi^*r)L^{1,0}_{\xi^\theta_{\sigma}} u = \mathfrak{a}(r\nabla^{\cl^\theta(1,0)}_\sigma\mathfrak{f})$ represent the same twisted canonical section. Since $\xi^{\theta}_{r\sigma} = (\pi^*r)\xi^{\theta}_\sigma$ we have
\begin{align*}
    L^{1,0}_{\xi^\theta_{r\sigma}}u &= (\pi^*r) (\xi^\theta_\sigma \lrcorner \nabla^{1,0}u) + \nabla^{1,0}( (\pi^*r) \xi^\theta_\sigma\lrcorner u) \\
    &= (\pi^*r) \left(\xi^\theta_\sigma \lrcorner \nabla^{1,0}u + \nabla^{1,0} (\xi^\theta_\sigma\lrcorner u)\right) + \pi^*\partial r \wedge (\xi^\theta_\sigma\lrcorner u) \\
    &= (\pi^*r) L^{1,0}_{\xi^\theta_\sigma}u + (\pi^*\partial r )\wedge (\xi^\theta_\sigma\lrcorner u).
\end{align*}
Since $\iota_{X_b}^*(\pi^*\partial r) = 0$, this shows that $L^{1,0}_{\xi^\theta_{r\sigma}}u$ and $(\pi^*r) L^{1,0}_{\xi^\theta_\sigma}u$ represent the same twisted relative canonical section.
Finally metric compatibility \eqref{compatibleconnection} of the connection $\nabla^{\cl^\theta} = \dbar^{\cl^\theta}+\nabla^{\cl^\theta(1,0)}$ follows from the Proposition \ref{metcomp}, so $\nabla^{\cl^\theta}$ is the BLS-Chern connection for $\cl^\theta$. This shows that $\cl^{\theta}$ is a BLS-field.
\end{proof}

\subsection{Field of Bergman spaces}\label{fieldofberg}
Let $\ch\to B$ be the Hilbert subfield of $\cl^\theta$ such that the fiber $\ch_b\subset \cl^\theta_b$ is the closed subspace containing holomorphic twisted canonical sections, i.e.,
\[
\ch_b = \left\{ f\in \Gamma\ipr{X_b, \mathcal{O}(K_{X_b}\otimes E|_{X_b})} : \sqi^{n^2} \int_{X_b} \hwedge{f}{f} < +\infty\right\}.
\]
Define a sheaf $\sC^{\infty}(\ch)$ as follows. Given an open set $U$ of $B$, let 
\[
\Gamma(U, \sC^{\infty}(\ch)) := \{ \mathfrak{f}\in \Gamma(U,\sC^{\infty}(\cl^\theta)) : \mathfrak{f}(b) \in \ch_b \quad  \text{for all } b\in U\}.
\]
Thus $\Gamma(U, \sC^{\infty}(\ch))$ consists of sections $\mathfrak{f}$ such that the restriction $\iota_b(\mathfrak{a(f)})$ of $\mathfrak{a(f)}$ to $X_b$ is a holomorphic section of $K_{X_b}\otimes E|_{X_b} \to X_b$ that is smooth up to the boundary of $X_b$. For a general family of pseudoconvex domains, the image of the stalk $\sC^{\infty}(\ch)_b$ under the evaluation map need not be dense in $\ch_b$. Even when the image of the stalk $\sC^{\infty}(\ch)_b$ is dense in $\ch_b$, $\ch$ may not be a BLS field as it may not be possible to equip $\ch$ with a BLS-Chern connection. Indeed, requiring the existence of a BLS-Chern connection imposes strong restrictions on the fibers $X_b$ (see \cite{W2017} for instance, where a Chern connection is defined). For this reason we equip $\ch$ with the structure of an iBLS field and exploit the ambient BLS field $\cl^{\theta}$ (which has a BLS-Chern connection) to define the curvature of $\ch$.

\subsubsection{Concrete iBLS fields}\label{iBLSconc}
Let $\cl^\theta$ be the BLS field introduced in Section \ref{concL}. Then the maximal substalk bundle $\wt{\ch}_{\cl^\theta}$ of $\ch$ relative to $\cl^\theta$ is the collection of stalks $\{ \wt{\ch}_{\cl^\theta,b}\}_{b\in B}$ where
\[
 \wt{\ch}_{\cl^\theta,b} := \{\mathfrak{f}\in \sC^{\infty}(\cl^\theta)_b :\iota_b \mathfrak{a(f)} \in \ch_b \}. 
\]
Recall that $\iota_b: (K_{X/B}\otimes E)|_{X_b}\to K_{X_b}\otimes E|_{X_b}$ is the restriction isomorphism in Lemma \ref{twistedrelcan}. For $b\in B$, we denote by $\wt{\ch}_{\cl^\theta, b}(b)$ the image of $\wt{\ch}_{\cl^\theta,b}$ under the evaluation map, i.e.,
\[
 \wt{\ch}_{\cl^\theta,b}(b) := \{\iota_b \mathfrak{a(f)} \in \ch_b:\mathfrak{f}\in \wt{\ch}_{\cl^\theta,b} \}. 
\]
Let $\Sigma^\theta \subset \wt{\ch}_{\cl^\theta,b}$ be the subtalk bundle whose fiber $\Sigma^{\theta}_b$ over a point $b\in B$ is 
\begin{equation}\label{concsigma}
\Sigma^\theta_b := \{\mathfrak{f} \in \wt{\ch}_{\cl^\theta, b} : \iota_{X_b}^*L^{0,1}_{\ol{\xi}^\theta_\tau}\mathfrak{a(f)} \in \ch_b \text{ and } \iota_{X_b}^* L^{1,0}_{\xi^\theta_\tau}\dbar \mathfrak{a(f)}=0 \text{ for all } \tau \in \sC^{\infty}(T^{1,0}_B)_b \}.
\end{equation}
\begin{remark}
Note that the first condition in \eqref{concsigma} means that $\dbar^{\cl^\theta}\mathfrak{f}(\ol\tau) \in \wt{\ch}_{\cl^{\theta},b}$ so that $(\ch, \Sigma^{\theta})$ is an infinitesimally holomorphic subfield of $\cl^{\theta}$. It can be checked that the second condition in \eqref{concsigma} is independent of the choice of representative $u$ of $\mathfrak{a(f)}$. To this end let $t^1,\dots, t^m$ be coordinates for $B$ near $b$ and let $\tau\in \sC^{\infty}(T^{1,0}_b)_b$ be written as $\tau^j\frac{\partial}{\partial t^j}$ in these coordinates. It suffices to show that if $\alpha = \gamma_j\wedge dt^j$, then $\iota_{X_b}^*L^{1,0}_{\xi^\theta_\tau}\dbar \alpha = 0$. Indeed, 
\[
\iota_{X_b}^*L^{1,0}_{\xi^\theta_\tau} \dbar \alpha  = \iota_{X_b}^*L^{1,0}_{\xi^\theta_\tau} (\dbar\gamma_j \wedge \pi^*dt^j) = \iota_{X_b}^*\left(L^{1,0}_{\xi^\theta_\tau} (\dbar\gamma_j) \wedge \pi^*dt^j + \dbar\gamma_j \wedge \pi^*\partial \tau_j\right) =0. 
\]
Thus, the second condition in \eqref{concsigma} means that $\mathfrak{a(f)}$ is represented by a twisted $(n,0)$-form $u$ so that $\iota_{X_b}^*L^{1,0}_{\xi^\theta_\tau}\dbar u=0$ for all $\tau \in \sC^{\infty}(T^{1,0}_B)_b$. Since the quantities $\dbar^{\cl^{\theta}}\mathfrak{f}(\ol\tau)$ and $\iota_{X_b}^*L^{1,0}_{\xi^\theta_\tau}\dbar \mathfrak{a(f)}$ are tensorial in the argument $\tau$, the conditions in \eqref{concsigma} only need to be checked for a finite number of vector fields that span $T^{1,0}_{B,b}$.
\end{remark}

\begin{proposition}\label{HiBLSprop}
Let $\ch \to B$ be the field of Bergman spaces, let $\cl^{\theta}$ be the BLS field as above. Let $\Sigma^\theta$ be the tuning defined by \eqref{concsigma} and suppose that $\Sigma^\theta_{b}(b)$ is dense in $\wt{\ch}_{\cl^\theta,b}(b)$ for all $b\in B$. Then $(\ch, \Sigma^{\theta})$ is a formal tuning. Moreover, for all $b\in B$ and for all $\sigma \in \sC^{\infty}(T^{1,0}_B)_b$ the domain of the second fundamental map $N^{\theta}_b(\sigma)$ is $\Sigma_b(b)$ and 
\[
N_b^{\theta}(\sigma)\mathfrak{f}(b) = \mathfrak{i}\left(P_b^{\perp}L^{1,0}_{\xi^\theta_{\sigma}}\mathfrak{a(f)} \right).
\]
\end{proposition}
\begin{proof}
Since $(\ch, \Sigma^\theta)$ is an infinitesimally almost holomorphic subfield, we have $(\cp^{\perp}\dbar^{\cl^\theta}\mathfrak{f}(\ol\sigma))(b)=0$ for all $\sigma \in \sC^{\infty}(T^{1,0}_B)_b$ and all $\mathfrak{f}\in \Sigma^{\theta}_b$. Consequently to check formality of $\Sigma^\theta$ it remains to verify that $(\cp^{\perp}\nabla^{\cl^\theta(1,0)}_\sigma\mathfrak{f})(b)=0$ for all $\sigma \in \sC^{\infty}(T^{1,0}_B)_b$ whenever $\mathfrak{f}(b)=0$. Since $\nabla^{\cl^\theta(1,0)}_\sigma\mathfrak{f}$ is tensorial in the argument $\sigma$, we may assume $\sigma$ is a holomorphic vector field. After unravelling the definitions, the vanishing of $\cp^{\perp} \nabla^{\cl^\theta(1,0)}_\sigma \mathfrak{f}$ at $b$ is equivalent to $P_b^{\perp}(\iota_{X_b}^*L^{1,0}_{\xi^\theta_\sigma}u)=0$ where $u$ is a twisted $(n,0)$-form representing $\mathfrak{a(f)}$ and $P_b$ denotes the Bergman projection. Thus we want to show that $\iota_{X_b}^*L^{1,0}_{\xi^\theta_\sigma}u$ is a holomorphic section of $K_{X_b}\otimes E|_{X_b} \to X_b$. By Proposition \ref{dbarPL},
\begin{align}\label{HiBLS}
  \dbar \iota_{X_b}^* L^{1,0}_{\xi^\theta_\sigma}u &=  \iota_{X_b}^* \left( \dbar\xi^\theta_\sigma \lrcorner \nabla^{1,0}u - (\xi^\theta_\sigma\lrcorner \Theta^E)\wedge u  - \nabla^{1,0}(\dbar\xi^\theta_\sigma \lrcorner u) + L^{1,0}_{\xi^\theta_\sigma}\dbar u \right).
\end{align}
Since $\sigma$ is a holomorphic $(1,0)$-vector field, $\dbar\xi^\theta_\sigma$ is a vertical vector field valued $(0,1)$-form. Therefore 
\[
\iota_{X_b}^* (\dbar \xi^\theta_\sigma \lrcorner \nabla^{1,0}u) = \dbar (\xi^{\theta}_\sigma|_{X_b})\lrcorner \iota_{X_b}^*\nabla^{1,0}u=0,
\]
since $\nabla^{1,0}u$ is an $(n+1,0)$-form. The second term $\iota_{X_b}^* ((\xi^\theta_{\sigma}\lrcorner \Theta^E)\wedge u)=0$ because $\iota_{X_b}^*u=0$. For the third term, since $\dbar\xi^\theta_\sigma$ is vertical 
\begin{align*}
\iota_{X_b}^*(\nabla(\dbar\xi^\theta_\sigma \lrcorner u)) &= \nabla^{1,0}\iota_{X_b}^*(\dbar \xi^\theta_\sigma\lrcorner u) = \nabla^{1,0} \left(\dbar(\xi^\theta_\sigma|_{X_b}) \lrcorner \iota_{X_b}^*u\right) = 0,
\end{align*}
since $\iota_{X_b}^*u =0$.
Finally, we see from condition \eqref{concsigma} that the last term in equation \eqref{HiBLS} is zero. 
\end{proof}


\begin{lemma}
Let $\ch \to B$ be the field of Bergman spaces, let $\cl^\theta$ be the ambient BLS field and let $\Sigma^\theta$ be the tuning defined by \eqref{concsigma}. Then $\Theta^\ch$ is a (possibly nonsmooth) opeartor valued $(1,1)$-form.
\end{lemma}

\begin{proof}
By Proposition \ref{iBLS11formcrit} it suffices to show that $(\ch, \Sigma^{\theta})$ is an infinitesimally holomorphic Hilbert subfield of $\cl^{\theta}$, which means that for all $b\in B$ and $\sigma, \tau \in \sC^{\infty}(T^{1,0}_B)_b$ we have $(\Theta^{\cl^\theta}_{\ol\sigma\ol\tau}\mathfrak{f})(b)=0$ for all $\mathfrak{f} \in \Sigma^{\theta}_b$. Let $u$ be an $(n,0)$-form representing $\mathfrak{a(f)}$. Note that 
\begin{align*}
    0 = \ol{\xi}^\theta_{\tau}\lrcorner \ol{\xi}^{\theta}_{\sigma} \lrcorner \dbar^2 u &= L^{0,1}_{\ol{\xi}^\theta_{\sigma}}L^{0,1}_{\ol{\xi}^\theta_{\tau}}u - L^{0,1}_{\ol{\xi}^\theta_{\tau}}L^{0,1}_{\ol{\xi}^\theta_{\sigma}}u- L^{0,1}_{\left[\ol{\xi}^\theta_{\sigma},\ol{\xi}^\theta_{\tau}\right]}u
\end{align*}
Therefore
\begin{align*}
    \Theta^{\cl^\theta}_{\ol\sigma(b)\ol\tau(b)}\mathfrak{f}(b) &= \iota_{X_b}^*\left( L^{0,1}_{\ol{\xi}^\theta_{\sigma}}L^{0,1}_{\ol{\xi}^\theta_{\tau}}u - L^{0,1}_{\ol{\xi}^\theta_{\tau}}L^{0,1}_{\ol{\xi}^\theta_{\sigma}}u- L^{0,1}_{\ol{\xi}^{\theta}_{[\sigma,\tau]}}u \right) \\
    &= \iota_{X_b}^* \left(L^{0,1}_{\left[\ol{\xi}^\theta_{\sigma},\ol{\xi}^\theta_{\tau}\right]}u - L^{0,1}_{\ol{\xi}^{\theta}_{[\sigma,\tau]}}u \right) \\
    &= \iota_{X_b}^*  \left( \left[\ol{\xi}^\theta_{\sigma},\ol{\xi}^\theta_{\tau}\right] - \ol{\xi}^{\theta}_{[\sigma,\tau]} \right)\lrcorner \dbar u\\
    &= \left( \left[\ol{\xi}^\theta_{\sigma},\ol{\xi}^\theta_{\tau}\right] - \ol{\xi}^{\theta}_{[\sigma,\tau]} \right)\Big|_{X_b}\lrcorner \iota_{X_b}^*\dbar u =0.
\end{align*}
Here the penultimate equality follows from the fact that $\left( \left[\ol{\xi}^\theta_{\sigma},\ol{\xi}^\theta_{\tau}\right] - \ol{\xi}^{\theta}_{[\sigma,\tau]} \right)$ is a vertical vector field, and the last equality  holds because $\iota_{X_b}^*u$ is holomorphic along $X_b$. 
\end{proof}

Thus, we have the following definition.

\begin{definition}[Curvature of a field of Bergman spaces]
Let $\ch \to B$ be the field of Bergman spaces. Suppose that $(\ch, \cl^{\theta},\Sigma^{\theta})$ is an iBLS field. Then for $b\in B$ and $\sigma, \tau \in \sC^{\infty}(T^{1,0}_B)_b$ the curvature $\Theta^{\ch/\cl^\theta}_{\sigma(b)\ol\tau(b)}$ is a densely defined operator with $\text{Dom}(\Theta^{\ch/\cl^\theta}_{\sigma(b)\ol\tau(b)}) = \Sigma^\theta_b(b)\subset \ch_b$ given by
\begin{equation}\label{conccurv}
\ipr{\Theta^{\ch/\cl^\theta}_{\sigma(b)\ol\tau(b)}\mathfrak{f}(b),\mathfrak{g}(b)}_b = \ipr{\Theta^{\cl^\theta}_{\sigma\ol\tau}\mathfrak{f},\mathfrak{g}}_b - \ipr{\cp^{\perp}\nabla^{\cl^\theta}_{\sigma}\mathfrak{f},\cp^{\perp}\nabla^{\cl^\theta}_{\tau}\mathfrak{g}}_b.
\end{equation}
\end{definition}
At this point, the curvature of $\ch$ seems to depend on the choice of the horizontal distribution $\theta$ used to equip $\cl^{\theta}$ with a connection. Indeed the curvature $\Theta^{\cl^\theta}|_{\ch}$ does depend on $\theta$, but the quantity $\Theta^{\ch/\cl^\theta}$ on the right hand side of \eqref{conccurv} is independent of $\theta$ (see Theorem \ref{liftinv}). The next proposition shows that the domain of $\Theta^{\ch/\cl^\theta}$ is independent of the choice of the horizontal distribution.

\begin{proposition}\label{liftinvintro}
Let $(E,h)\to X\to B$ be as in Section \ref{notation}. Let $\ch \to B$ be the field of Bergman spaces and let $\cl^{\theta_1}\to B$ (resp.  $\cl^{\theta_2}\to B$) be an ambient BLS field whose Chern connection $\nabla^{\cl^{\theta_1}}$ (resp. $\nabla^{\cl^{\theta_2}}$) is defined with respect to a horizontal distribution $\theta_1$ (resp. $\theta_2$). Let $\Sigma^{\theta_1}$ (resp. $\Sigma^{\theta_2}$) be the tuning of $\ch$ in $\cl^{\theta_1}$ (resp. $\cl^{\theta_2}$) given by \eqref{concsigma}. Then 
\begin{enumerate}
    \item $(\ch, \Sigma^{\theta_1})$ is an infinitesimally holomorphic Hilbert subfield of $\cl^{\theta_1}$ if and only if $(\ch, \Sigma^{\theta_2})$ is an infinitesimally holomorphic Hilbert subfield of $\cl^{\theta_2}$.

    \item The tunings $\Sigma^{\theta_1}$ and $\Sigma^{\theta_2}$ are the same. Thus $(\ch, \Sigma^{\theta_1})$ is a formal tuning if and only if $(\ch, \Sigma^{\theta_2})$ is. 

\end{enumerate}
\end{proposition}

\begin{proof}
Since the smooth structure for $\cl^{\theta_1}$ and $\cl^{\theta_2}$ are the same, the maximal substalk bundles $\wt{\ch}_{\cl^{\theta_1}}$ and $\wt{\ch}_{\cl^{\theta_2}}$ are equal. For $b\in B$ let $\eta \in \sC^{\infty}(T^{1,0}_B)_b$, let $\xi^{\theta_1}_{\eta}$ and $\xi^{\theta_2}_\eta$ be horizontal lifts of $\eta$ with respect to the distributions $\theta_1$ and $\theta_2$. Then $w_{\eta} = \xi^{\theta_1}_{\eta}-\xi^{\theta_2}_{\eta}$ is a vertical vector field. Let $\mathfrak{f}\in \wt{\ch}_{\cl^{\theta_1},b} =  \wt{\ch}_{\cl^{\theta_2},b}$ and let $u$ be an $E$-valued $(n,0)$-form representing $\mathfrak{a(f)}$. Then
\begin{align*}
\iota_{X_b}^*(L^{0,1}_{\ol{\xi}^{\theta_1}_{\eta}}u - L^{0,1}_{\ol{\xi}^{\theta_2}_{\eta}}u) = \iota_{X_b}^*(L^{0,1}_{\ol{w}_{\eta}}u) = \iota_{X_b}^*(\ol{w}_{\eta}\lrcorner \dbar u) = \ol{w}_{\eta}\big|_{X_b} \lrcorner \iota_{X_b}^*\dbar u =0,
\end{align*}
where the penultimate equality depends on the fact that $w_{\eta}$ is vertical and the last equality follows because $\iota_b(\mathfrak{a(f)})=\iota_{X_b}^*u$ is in $\ch_b$.
Thus 
\begin{equation}\label{sigmainv1}
    \iota_{X_b}^*L^{0,1}_{\ol{\xi}^{\theta_1}_{\eta}}u\in \ch_b \text{ if and only if } \iota_{X_b}^*L^{0,1}_{\ol{\xi}^{\theta_2}_{\eta}}u\in \ch_b.
\end{equation} This proves the first statement in the proposition.  For the second condition in \eqref{concsigma}, since $w_\eta$ is vertical
\begin{align*}
\iota_{X_b}^* (L^{1,0}_{\xi^{\theta_1}_\eta}\dbar u - L^{1,0}_{\xi^{\theta_2}_\eta}\dbar u) = \iota_{X_b}^* L^{1,0}_{w_\eta}\dbar u &= \iota_{X_b}^*(w_\eta\lrcorner \nabla^{1,0}\dbar u + \nabla^{1,0}(w_\eta \lrcorner \dbar u)) \\
&= w_\eta|_{X_b}\lrcorner \nabla^{1,0} (\iota_{X_b}^*\dbar u) + \nabla^{1,0}(w_\eta|_{X_b} \lrcorner \iota_{X_b}^* \dbar u) = 0,
\end{align*}
where the last equality follows because $\iota_{X_b}^*u \in \ch_b$.
Thus 
\begin{equation}\label{sigmainv2}
\iota_{X_b}^* L^{1,0}_{\xi^{\theta_1}_\eta}\dbar u=0 \text{ if and only if } \iota_{X_b}^* L^{1,0}_{\xi^{\theta_2}_\eta}\dbar u=0. 
\end{equation}
It follows from \eqref{sigmainv1} and \eqref{sigmainv2} that $\Sigma_b^{\theta_1} = \Sigma_b^{\theta_2}$. Now the rest of statement (2) follows from Proposition \ref{HiBLSprop} and the fact that $\wt{\ch}_{\cl^{\theta_1},b}= \wt{\ch}_{\cl^{\theta_2},b}$.
\end{proof}

\begin{theorem}\label{liftinv}
Let $(E,h)\to X\to B$ be as in Section \ref{notation}. Let $\ch \to B$ be the field of Bergman spaces. Let $\cl^{\theta_1}, \cl^{\theta_2} \to B$ and $\Sigma^{\theta_1},\Sigma^{\theta_2}\to B$ be as in Proposition \ref{liftinvintro}. Then for $\mathfrak{f,g} \in \Sigma^{\theta_1}_b= \Sigma^{\theta_2}_b$ and $\sigma, \tau \in \sC^{\infty}(T^{1,0}_B)_b$ we have 
\begin{equation}\label{Qequal}
    \ipr{ \Theta^{\mathcal{L}^{\theta_1}}_{\sigma \ol \tau}\mathfrak{f,g}}_b - \ipr{ \cp^{\perp} \nabla^{\mathcal{L}^{\theta_1}}_{\sigma}\mathfrak{f}, \cp^{\perp}\nabla^{\mathcal{L}^{\theta_1}}_{\tau}\mathfrak{g}}_b = \ipr{ \Theta^{\mathcal{L}^{\theta_2}}_{\sigma \ol \tau}\mathfrak{f,g}}_b - \ipr{ \cp^{\perp}\nabla^{\mathcal{L}^{\theta_2}}_{\sigma}\mathfrak{f}, \cp^{\perp}\nabla^{\mathcal{L}^{\theta_2}}_{\tau}\mathfrak{g}}_b.
\end{equation}
Thus, if $(\ch, \cl^{\theta_1},\Sigma^{\theta_1})$ is an iBLS field then $(\ch, \cl^{\theta_2}, \Sigma^{\theta_2})$ is also an iBLS field and $\Theta^{\ch/\cl^{\theta_1}}_{\sigma(b)\ol\tau(b)} = \Theta^{\ch/\cl^{\theta_2}}_{\sigma(b)\ol\tau(b)}$.
\end{theorem}
\begin{proof}
By Proposition \ref{liftinvintro} we have $\Sigma^{\theta_1} = \Sigma^{\theta_2}$ so denote both tunings by $\Sigma$. If $(\ch, \cl^{\theta_1},\Sigma)$ is an iBLS field, by definition \eqref{iBLSdef} we have
\begin{enumerate}[label=(\alph*)]
     \item For every $b\in B$ the subspace $\wt{\ch}_{\cl^{\theta_1},b}(b)$ is dense in $\ch_b$.
    
    \item $(\ch, \Sigma)$ is a formal tuning.
    
    \item $(\ch, \Sigma)$ is an infinitesimally holomorphic Hilbert subfield of $\cl^{\theta_1}$. 
    
    \item For all $b\in B$, all $\sigma, \tau \in \sC^{\infty}(T_B)_b$ and all $\mathfrak{f} \in \Sigma_b$ the anti-linear functional $Q^{\theta_1}_{\sigma\ol\tau}(\mathfrak{f}(b),\cdot)_b$ is continuous where 
    \[
    Q^{\theta_1}_{\sigma\ol\tau}(\mathfrak{f}(b),\mathfrak{g}(b))_b = \ipr{ \Theta^{\mathcal{L}^{\theta_1}}_{\sigma \ol \tau}\mathfrak{f,g}}_b - \ipr{ \cp^{\perp} \nabla^{\mathcal{L}^{\theta_1}}_{\sigma}\mathfrak{f}, \cp^{\perp}\nabla^{\mathcal{L}^{\theta_1}}_{\tau}\mathfrak{g}}_b.
    \]
\end{enumerate}
By Proposition \ref{liftinvintro} conditions $(a)$-$(c)$ also hold for $(\ch,\cl^{\theta_2},\Sigma)$. It follows from \eqref{Qequal} that for $\sigma, \tau \in \sC^{\infty}(T^{1,0}_B)_b$ and $\mathfrak{f}\in \Sigma_b$ the functional $Q_{\sigma\ol\tau}^{\theta_1}(\mathfrak{f}(b), \cdot)_b = Q_{\sigma\ol\tau}^{\theta_2}(\mathfrak{f}(b), \cdot)_b$, so $Q_{\sigma\ol\tau}^{\theta_2}(\mathfrak{f}(b), \cdot)_b$ is continuous as well. Thus, $(\ch, \cl^{\theta_2},\Sigma)$ is also an iBLS field and by the Riesz representation theorem we must have $\Theta^{\ch/\cl^{\theta_1}}_{\sigma(b)\ol\tau(b)}\mathfrak{f}(b) = \Theta^{\ch/\cl^{\theta_2}}_{\sigma(b)\ol\tau(b)}\mathfrak{f}(b)$. It remains to prove \eqref{Qequal}. 

To do so, let $u$ and $v$ be twisted $(n,0)$-forms that represent $\mathfrak{a(f)}$ and $\mathfrak{a(g)}$ respectively. 
For the $(1,0)$ vector field $\sigma$ (resp. $\tau$) on $B$, let $\xi^i_{\sigma}$ (resp. $\xi^i_{\tau}$) denote the horizontal lift of $\sigma$ (resp. $\tau$) with respect to the distribution $\theta_i^{1,0}$, where $i=1,2$. Moreover, denote by $w_{\sigma}$ and $w_{\tau}$ the vertical vector fields $w_{\sigma}= \xi^1_{\sigma} - \xi^2_{\sigma}$ and $w_{\tau} = \xi^1_{\tau} - \xi^2_{\tau}$. With this notation we first calculate the difference between the curvatures of the Hilbert fields $\cl^{\theta_1}$ and $\cl^{\theta_2}$. We have 
\begin{align*}
    \ipr{ \Theta^{\cl^{\theta_1}}_{\sigma\ol\tau}\mathfrak{f,g}}_b  &= \ipr{ L^{1,0}_{\xi^1_{\sigma}}L^{0,1}_{\ol{\xi}^1_{\tau}} u - L^{0,1}_{\ol{\xi}^1_{\tau}}L^{1,0}_{\xi^1_{\sigma}}u,v}_b \\
    &=\ipr{L^{1,0}_{\xi^1_{\sigma}} L^{0,1}_{\ol{\xi}^1_{\tau}}u,v }_b - \ipr{ L^{1,0}_{\xi^2_{\sigma}}L^{0,1}_{\ol{\xi}^1_{\tau}}u,v }_b + \ipr{ L^{1,0}_{\xi^2_{\sigma}}L^{0,1}_{\ol{\xi}^1_{\tau}}u,v }_b -\ipr{L^{0,1}_{\ol{\xi}^1_{\tau}}L^{1,0}_{\xi^1_{\sigma}}u,v}_b \\
    &= \ipr{L^{1,0}_{w_{\sigma}}L^{0,1}_{\ol{\xi}^1_{\tau}}u,v}_b + \ipr{ L^{1,0}_{\xi^2_{\sigma}}L^{0,1}_{\ol{\xi}^1_{\tau}}u,v }_b  -\ipr{L^{0,1}_{\ol{\xi}^1_{\tau}}L^{1,0}_{\xi^1_{\sigma}}u,v}_b, 
\end{align*}
and 
\begin{align*}
    \ipr{ \Theta^{\cl^{\theta_2}}_{\sigma\ol\tau}\mathfrak{f,g}}_b &=  \ipr{ L^{1,0}_{\xi^2_{\sigma}}L^{0,1}_{\ol{\xi}^2_{\tau}} u - L^{0,1}_{\ol{\xi}^2_{\tau}}L^{1,0}_{\xi^2_{\sigma}}u,v}_b \\
    &=  \ipr{ L^{1,0}_{\xi^2_{\sigma}}L^{0,1}_{\ol{\xi}^2_{\tau}} u, v}_b - \ipr{L^{0,1}_{\ol{\xi}^1_{\tau}}L^{1,0}_{\xi^2_{\sigma}}u,v}_b + \ipr{L^{0,1}_{\ol{\xi}^1_{\tau}}L^{1,0}_{\xi^2_{\sigma}}u,v}_b - \ipr{L^{0,1}_{\ol{\xi}^2_{\tau}}L^{1,0}_{\xi^2_{\sigma}}u,v}_b \\
    &= \ipr{ L^{1,0}_{\xi^2_{\sigma}}L^{0,1}_{\ol{\xi}^2_{\tau}} u, v}_b - \ipr{L^{0,1}_{\ol{\xi}^1_{\tau}}L^{1,0}_{\xi^2_{\sigma}}u,v}_b + \ipr{ L^{0,1}_{\ol{w}_{\tau}}L^{1,0}_{\xi^2_{\sigma}}u,v}_b.
\end{align*}
Subtracting the two quantities, we get
\begin{align}\label{curvdiff}
    \ipr{ \Theta^{\cl^{\theta_1}}_{\sigma\ol{\tau}}\mathfrak{f,g}}_b - \ipr{ \Theta^{\cl^{\theta_2}}_{\sigma \ol{\tau}} \mathfrak{f,g}}_b &= \ipr{ L^{1,0}_{w_{\sigma}}L^{0,1}_{\ol{\xi}^1_{\tau}} u,v}_b + \ipr{ L^{1,0}_{\xi^2_{\sigma}} L^{0,1}_{\ol{w}_{\tau}}u,v}_b -\ipr{ L^{0,1}_{\ol{\xi}^1_{\tau}}L^{1,0}_{w_{\sigma}}u,v }_b - \ipr{ L^{0,1}_{\ol{w}_{\tau}}L^{1,0}_{\xi^2_{\sigma}}u,v}_b.  
\end{align}
Let us deal with terms on the right hand side of the equation \eqref{curvdiff} individually. Using identity \eqref{zerolie10} we see that
\[
\ipr{ L^{1,0}_{w_{\sigma}}L^{0,1}_{\ol{\xi}^1_{\tau}} u,v}_b = - \ipr{ L^{0,1}_{\ol\xi^1_\tau}u, L^{0,1}_{\ol{w}_\sigma} v }_b = 0,
\]
where the second equality holds because $\iota_{X_b}^*(L^{0,1}_{\ol{w}_{\sigma}}v) =0$. To see this, first observe that for a vertical vector field $\eta$ on $X$ and an $E$-valued $(p,q)$-form $\alpha$ we have $\iota_{X_b}^*(\eta \lrcorner \alpha) = \eta|_{X_b}\lrcorner \iota^*_{X_b}\alpha$. Since $w_\sigma$ is vertical
\[
\iota_{X_b}^*(L^{0,1}_{\ol{w}_{\sigma}}v) = \iota_{X_b}^* (\ol{w}_{\sigma} \lrcorner \dbar v) = \ol{w}_{\sigma}|_{X_b} \lrcorner (\iota_{X_b}^* \dbar v) = 0,
\]
as $v$ is holomorphic along $X_b$.

We claim that the second term on the right hand side of \eqref{curvdiff} also vanishes. Indeed, since both $w_{\tau}$ and $[\xi^2_{\sigma},w_{\tau}]$ are vertical vector fields,
\begin{align*}
\iota_{X_b}^*(L^{1,0}_{\xi^2_{\sigma}}L^{0,1}_{\ol{w}_{\tau}} u) = \iota_{X_b}^*(L^{1,0}_{\xi^2_{\sigma}} (\ol{w}_{\tau}\lrcorner \dbar u)) &= \iota_{X_b}^* \left( \ol{w}_{\tau}\lrcorner (L^{1,0}_{\xi^2_{\sigma}} \dbar u) + [\xi^2_{\sigma},\ol{w}_{\tau}]^{0,1}\lrcorner \dbar u \right) \\
&= \ol{w}_{\tau}|_{X_b}\lrcorner \iota_{X_b}^*(L^{1,0}_{\xi^2_{\sigma}}\dbar u) + [\xi^2_{\sigma},\ol{w}_{\tau}]^{0,1}|_{X_b} \lrcorner \iota_{X_b}^* (\dbar u) = 0.
\end{align*}
The last equality follows from the fact that $\iota_{X_b}^*(L^{1,0}_{\xi^2_{\sigma}}\dbar u) =0$ and $\iota_{X_b}^*(\dbar u) =0$ for an $E$-valued $(n,0)$-form $u$ representing the section $\mathfrak{a(f)}$ associated to $\mathfrak{f}\in \Sigma_b(U)$.

For the third term, the identity \eqref{zerolie10} yields 
\[
\ipr{ L^{1,0}_{w_\sigma}u,v}_b = - \ipr{u, L^{0,1}_{\ol{w}_\sigma}v}_b.
\]
Using \eqref{metcomp01} to differentiate both sides of this equation, we get
\begin{align*}
    \ipr{L^{0,1}_{\ol{\xi}^1_{\tau}} L^{1,0}_{w_\sigma}u,v}_b + \ipr{ L^{1,0}_{w_\sigma}u, L^{1,0}_{\xi^1_{\tau}}v }_b &= \ol\tau \ipr{L^{1,0}_{w_\sigma}u,v }_b \\
    &= - \ol\tau \ipr{u, L^{0,1}_{\ol{w}_\sigma}v}_b \\
    &= - \ipr{ L^{0,1}_{\ol\xi^1_{\tau}}u, L^{0,1}_{\ol{w}_\sigma}v}_b - \ipr{ u, L^{1,0}_{\xi^1_{\tau}} L^{0,1}_{\ol{w}_\sigma}v}_b =0.
\end{align*}
The same reasoning we used to explain the vanishing of the first two terms on the right hand side of \eqref{curvdiff} gives $\iota_{X_b}^*(L^{0,1}_{\ol{w}_\sigma}v)=0$ and $\iota_{X_b}^*(L^{1,0}_{\xi^1_{\tau}}L^{0,1}_{\ol{w}_\sigma}v)=0$.

By using the identity \eqref{zerolie01} the fourth term can be written as
\[
- \ipr{ L^{0,1}_{\ol{w}_\tau} L^{1,0}_{\xi^2_\sigma}u,v}_b =  \ipr{ L^{1,0}_{\xi^2_\sigma}u, L^{1,0}_{w_\tau}v }_b.
\]
After all these simplifications, equation \eqref{curvdiff} reduces to 
\[
\ipr{ \Theta^{\cl^{\theta_1}}_{\sigma \ol{\tau}}\mathfrak{f,g}}_b - \ipr{ \Theta^{\cl^{\theta_2}}_{\sigma \ol\tau}\mathfrak{f,g} }_b = \ipr{ L^{1,0}_{w_\sigma}u,L^{1,0}_{\xi^1_{\tau}}v }_b + \ipr{ L^{1,0}_{\xi^2_{\sigma}}u,L^{1,0}_{w_\tau}v}_b.
\]
We simplify \eqref{curvdiff} one final time. To do so, observe that because $w_\sigma$ is a vertical vector field, we have $\iota_{X_b}^*(w_\sigma \lrcorner \nabla^{1,0}u) = w_{\sigma}|_{X_b}\lrcorner \iota_{X_b}^* (\nabla^{1,0}u) = 0$, as $\nabla^{1,0}u$ is an $(n+1,0)$-form. Thus we have
\begin{align*}
    \ipr{ L^{1,0}_{w_\sigma}u, L^{1,0}_{\xi^1_{\tau}}v }_b &= \sqi^{n^2} \int_{X_b} \hwedge{\nabla^{1,0}(w_\sigma \lrcorner u)}{L^{1,0}_{\xi^1_{\tau}}v} \\
    &= \sqi^{n^2} \int_{X_b} \partial \hwedge{(w_\sigma \lrcorner u)}{L^{1,0}_{\xi^1_{\tau}}v} + (-1)^n \hwedge{(w_\sigma\lrcorner u)}{\dbar L^{1,0}_{\xi^1_{\tau}}v} \\
    &= \sqi^{n^2} \int_{X_b} \partial \hwedge{(w_\sigma \lrcorner u)}{L^{1,0}_{\xi^1_{\tau}}v} + (-1)^n \hwedge{(w_\sigma\lrcorner u)}{\dbar P^{\perp}_b L^{1,0}_{\xi^1_{\tau}}v} \\
    &= \sqi^{n^2} \int_{X_b} \partial \hwedge{(w_\sigma \lrcorner u)}{L^{1,0}_{\xi^1_{\tau}}v} - \partial \hwedge{(w_\sigma \lrcorner u)}{P_b^{\perp}L^{1,0}_{\xi^1_{\tau}}v} + \hwedge{\nabla^{1,0}(w_\sigma\lrcorner u)}{ P^{\perp}_b L^{1,0}_{\xi^1_{\tau}}v} \\
    &=\sqi^{n^2} \int_{X_b}d \left(w_\sigma \lrcorner \hwedge{u}{P_bL^{1,0}_{\xi^1_{\tau}}v} \right) +\ipr{L^{1,0}_{w_\sigma}u, P_b^{\perp} L^{1,0}_{\xi^1_{\tau}}v }_b \\
    &= \int_{X_b} d\left(w_\sigma \lrcorner \left\{u, P_bL^{1,0}_{\xi^1_\tau}v \right\}dV_{\omega_b} \right) + \ipr{P_b^{\perp} L^{1,0}_{w_\sigma}u, P_b^{\perp} L^{1,0}_{\xi^1_{\tau}}v}_b \\
    &= \int_{\partial X_b} \partial \rho_b(w_\sigma)\left\{ u, P_bL^{1,0}_{\xi^1_\tau}v\right\}dS +\ipr{P_b^{\perp} L^{1,0}_{w_\sigma}u, P_b^{\perp} L^{1,0}_{\xi^1_{\tau}}v}_b \\
    &= \ipr{P_b^{\perp} L^{1,0}_{w_\sigma}u, P_b^{\perp} L^{1,0}_{\xi^1_{\tau}}v}_b,
\end{align*}
where $\{ \cdot, \cdot\}$ denotes the pointwise inner product on $E|_{X_b}$-valued forms. We also assume that $\rho$ is a defining function for $\partial X$ so that $\rho_b := \rho|_{X_b}$ is defining function for $\partial X_b$ satisfying $\abs{\partial \rho_b}=1$ on $\partial X_b$. Using the fact that $w_\sigma$ is a vertical vector field we get
\begin{align*}
\iota_{\partial X_b}^* \left(\partial \rho_b(w_\sigma) \right) = \iota_{\partial X_b}^* \left(w_\sigma|_{X_b} \lrcorner \partial \rho_b \right) &= \iota_{\partial X_b}^* \left(w_\sigma|_{X_b} \lrcorner \iota_{X_b}^*\partial \rho \right) \\ 
&= \iota_{\partial X_b}^* \iota_{X_b}^*(w_\sigma \lrcorner \partial \rho) \\
&= \iota_{\partial X_b}^* \iota_{X_b}^* \left(\partial\rho(\xi^1_\sigma) - \partial \rho(\xi^2_\sigma)\right) \\
&=  \iota_{X_b}^* \iota_{\partial X}^* \left(\partial\rho(\xi^1_\sigma) - \partial \rho(\xi^2_\sigma)\right) =0, 
\end{align*}
since $\xi^1_\sigma$ and $\xi^2_\sigma$ are tangent to $\partial X$. The last equality follows from this. In a similar manner, we can show that $\ipr{ L^{1,0}_{\xi^2_\sigma}u, L^{1,0}_{w_\tau}v}_b = \ipr{ P^{\perp}_b L^{1,0}_{\xi^2_\sigma}u, P^{\perp}_b L^{1,0}_{w_\tau}v}_b$.
After this final simplification, \eqref{curvdiff} becomes
\begin{equation}\label{curvdiff2}
    \ipr{ \Theta^{\cl^{\theta_1}}_{\sigma \ol{\tau}}\mathfrak{f,g}}_b - \ipr{ \Theta^{\cl^{\theta_2}}_{\sigma \ol\tau}\mathfrak{f,g} }_b = \ipr{ P_b^{\perp}L^{1,0}_{w_\sigma}u, P_b^{\perp} L^{1,0}_{\xi^1_{\tau}}v }_b + \ipr{ P_b^{\perp} L^{1,0}_{\xi^2_{\sigma}}u, P_b^{\perp} L^{1,0}_{w_\tau}v}_b.
\end{equation}

We now calculate the difference between the square of the second fundamental maps of $\ch$ with respect to the Hilbert fields $\cl^{\theta_1}$ and $\cl^{\theta_2}$ as
\begin{align}
    &\quad \ipr{ \cp^{\perp} \nabla^{\cl^{\theta_1}}_{\sigma} \mathfrak{f}, \cp^{\perp} \nabla^{\cl^{\theta_1}}_{\tau} \mathfrak{g} }_b - \ipr{ \cp^{\perp} \nabla^{\cl^{\theta_2}}_{\sigma} \mathfrak{f}, \cp^{\perp} \nabla^{\cl^{\theta_2}}_{\tau} \mathfrak{g} }_b \nonumber \\
    &= \ipr{P_b^{\perp} L^{1,0}_{\xi^1_\sigma} u, P_b^{\perp} L^{1,0}_{\xi^1_\tau} v  }_b -  \ipr{P_b^{\perp} L^{1,0}_{\xi^2_\sigma} u, P_b^{\perp} L^{1,0}_{\xi^2_\tau} v  }_b \nonumber \\
    &= \ipr{P_b^{\perp} L^{1,0}_{\xi^1_\sigma} u, P_b^{\perp} L^{1,0}_{\xi^1_\tau} v  }_b - \ipr{P_b^{\perp} L^{1,0}_{\xi^2_\sigma} u, P_b^{\perp} L^{1,0}_{\xi^1_\tau} v  }_b + \ipr{P_b^{\perp} L^{1,0}_{\xi^2_\sigma} u, P_b^{\perp} L^{1,0}_{\xi^1_\tau} v  }_b- \ipr{P_b^{\perp} L^{1,0}_{\xi^2_\sigma} u, P_b^{\perp} L^{1,0}_{\xi^2_\tau} v  }_b \nonumber \\
    &= \ipr{ P_b^{\perp}L^{1,0}_{w_\sigma}u, P_b^{\perp} L^{1,0}_{\xi^1_{\tau}}v }_b + \ipr{ P_b^{\perp} L^{1,0}_{\xi^2_{\sigma}}u, P_b^{\perp} L^{1,0}_{w_\tau}v}_b. \label{curvdiff3}
\end{align}
The identity \eqref{Qequal} now follows from equations \eqref{curvdiff2} and \eqref{curvdiff3}.
\end{proof}


\subsection{Final remarks on the iBLS-structure}\label{finalremarks}
Since the definition of an iBLS field is somewhat technical, it seems difficult to determine the precise necessary and sufficient geometric conditions on the fibration $E\to X\to B$ that guarantee that $\ch\to B$ is an iBLS field. We can however give examples of fibrations when $(\ch,\Sigma)$ is an iBLS field. In all the examples below it is possible to equip $\ch \to B$ with a smooth structure $\sC^{\infty}(\ch)$ so that the image of the stalk $\sC^{\infty}(\ch)_b$ under the evaluation map is dense in $\ch_b$ for all $b\in B$. The Hilbert fields in examples $(3)$-$(5)$ might fail to be BLS fields as they cannot always be equipped with a BLS-Chern connection.

All these fibrations also satisfy the condition that for all $b\in B$ the fiber $X_b$ is a domain on which the Neumann operator $\mathcal{N}_b^{(0,2)}$ acting on $(0,2)$-forms is globally regular. The latter condition is needed in Theorem \ref{Bgeneral}, so the theorem holds for these classes of examples. First we need the following estimate on the norm of the second fundamental map.

\begin{theorem}\label{2ndfunprop}
Let $(E,h)\to B$ be as in subsection \ref{notation}. Assume that the curvature of line bundle $E$ is semipositive. Let $\tau_1, \cdots, \tau_m$ be holomorphic $(1,0)$-vector fields on an open set $U$ and $\xi_{\tau_1},\cdots, \xi_{\tau_m}$ be their horizontal lifts to $X$. Then for $\mathfrak{f}_1,\cdots,\mathfrak{f}_m\in \Sigma_b$ we have the estimate
\begin{equation}\label{2ndfunestimate}
    \norm{P_b^{\perp}\iota_b\left(\sum_{j=1}^m L^{1,0}_{\xi_{\tau_j}} \mathfrak{a}(\mathfrak{f}_j)\right)}_b^2 \leq \sum_{j,k=1}^m \left[\ipr{\Theta^E_{\xi_{\tau_j}\ol{\xi}_{\tau_k}}\mathfrak{a}(\mathfrak{f}_j),\mathfrak{a}(\mathfrak{f}_k)}_b + \ipr{ \dbar\xi_{\tau_j}\lrcorner \mathfrak{a}(\mathfrak{f}_j),\dbar\xi_{\tau_k}\lrcorner \mathfrak{a}(\mathfrak{f}_k)}_b\right].
\end{equation}
\end{theorem}

\begin{proof}
Let $u_j$ be an $E$-valued $(n,0)$-form representing the twisted relative canonical section $\mathfrak{a}(\mathfrak{f}_j)$. From \eqref{2ndfund} we have
\begin{align*}
    \dbar\iota_{X_b}^*\ipr{L^{1,0}_{\tau_j}u_j} &= \iota_{X_b}^* \left( \dbar \xi_{\tau_j}\lrcorner \nabla^{1,0}u_j - (\xi_{\tau_j}\lrcorner \Theta^E)\wedge u_j - \nabla^{1,0}(\dbar \xi_{\tau_j}\lrcorner u_j) + L^{1,0}_{\xi_{\tau_j}}\dbar u_j\right) \\
    &= \iota_{X_b}^* \left( -(\xi_{\tau_j}\lrcorner \Theta^E)\wedge u_j - \nabla^{1,0}(\dbar \xi_{\tau_j}\lrcorner u_j) \right).
\end{align*}
Indeed, $\iota_{X_b}^*\ipr{\dbar \xi_{\tau_j}\lrcorner \nabla^{1,0}u_j}=0$ and $\iota_{X_b}^*\ipr{L^{1,0}_{\xi_{\tau_j}}\dbar u_j}=0$ as in the proof of Proposition \ref{HiBLSprop}. Thus $P_b^{\perp}\iota_b\ipr{\sum_{j=1}^m L^{1,0}_{\xi_{\tau_j}}\mathfrak{a}(\mathfrak{f}_j)}$ is the solution of the smallest $L^2$-norm to the equation 
\[
\dbar f = \sum_{j=1}^m\iota_{X_b}^* \left( -(\xi_{\tau_j}\lrcorner \Theta^E)\wedge u_j - \nabla^{1,0}(\dbar \xi_{\tau_j}\lrcorner u_j)\right)
\]
on the pseudoconvex domain $X_b$. To simplify notation, let 
\[
\kappa = \sum_{j=1}^m \iota_{X_b}^*\left(\dbar \xi_{\tau_j}\lrcorner u_j\right), \quad \quad \mu = \sum_{j=1}^m \iota_{X_b}^* \left((\xi_{\tau_j}\lrcorner \Theta^E)\wedge u_j\right),
\]
and denote the quantity on the right hand side of the inequality \eqref{2ndfunestimate} by $C$. By the usual functional analysis associated to $L^2$ existence theory for the $\dbar$-equation on complete K\"ahler manifolds, in order to get the estimate \eqref{2ndfunestimate}, it suffices to show that for all compactly supported $E$-valued $(n,1)$-forms $\alpha$ on $X_b$ we have
\[
\abs{\ipr{\mu+ \nabla^{1,0}\kappa,\alpha}_b}^2 \leq C\ipr{\norm{\dbar\alpha}_b^2 + \norm{\dbarstar\alpha}_b^2} = C\ipr{\Box \alpha,\alpha}_b. 
\]
Since $\Theta^E$ is a nonnegative $(1,1)$-form on $X$, we can locally write $\Theta^E = \sum_{k=1}^{n+m}\gamma^{k}\wedge \ol{\gamma}^k$ for some $(1,0)$-forms $\gamma^1,\dots, \gamma^{n+m}$ on $X$. Let $\gamma^k_b$ denote the form $\gamma^k_b = \iota_{X_b}^*\gamma^k$, so that $\sum_{j=1}^{n+m} \gamma_b^k\wedge \ol{\gamma}^k_b := \Theta^{E_b}$ is the curvature of the line bundle $E|_{X_b}$. Finally, let $\vee$ denote the adjoint of the wedge product, i.e., for a form $\beta$ and twisted forms $v$ and $w$ on $X_b$ of appropriate bidegrees, we have
\[
\{ \beta\wedge v,w\} = \{ v,\beta \vee w\}.
\]
Now, let $\alpha$ be a compactly supported $E|_{X_b}$-valued $(n,1)$-form on $X_b$. Then
\begin{align}
     \left\{ \sum_{j=1}^m (\xi_{\tau_j}\lrcorner \Theta^E)\wedge u_j,\alpha\right\} &= \sum_{j=1}^m\sum_{k=1}^{n+m}\left\{ \gamma^k(\xi_{\tau_j})\ol{\gamma}^k_b\wedge u_j,\alpha\right\} \nonumber \\
    &= \sum_{k=1}^{n+m}\left\{ \sum_{j=1}^m\gamma^k(\xi_{\tau_j}) u_j, \ol{\gamma}_b^k\vee \alpha\right\} \nonumber \\
    &\leq \left[\sum_{k=1}^{n+m} \left\{ \sum_{j=1}^m \gamma^k(\xi_{\tau_j})u_j, \sum_{i=1}^m \gamma^k(\xi_{\tau_i})u_i\right\}\right]^{1/2}\left[\sum_{k=1}^{n+m} \left\{ \ol{\gamma}^k_b\vee \alpha, \ol{\gamma}^k_b\vee \alpha \right\}\right]^{1/2} \nonumber \\
    &= \left[ \sum_{k=1}^{n+m}\sum_{i,j=1}^m \left\{\gamma^k(\xi_{\tau_j})\ol{\gamma^k(\xi_{\tau_i})}u_j,u_i \right\}\right]^{1/2} \left\{\sum_{k=1}^{n+m}\ol{\gamma}^k_b\wedge(\ol{\gamma}^k_b\vee \alpha),\alpha \right\}^{1/2} \nonumber\\
    &= \left[\sum_{i,j=1}^m \{ \Theta^E_{\xi_{\tau_j}\ol{\xi}_{\tau_i}}u_j,u_i\}\right]^{1/2}\left\{ \left[\Theta^{E_b},\Lambda_{\omega_b}\right]\alpha, \alpha\right\}^{1/2}. \label{2ndfunest1}
\end{align}
Here $\Lambda_{\omega_b}$ denotes the adjoint of the Lefschetz operator $L_{\omega_b}$ on $X_b$ with respect to the inner product $\{\cdot,\cdot\}$. To see the final equality, note that $\Lambda_{\omega_b}\Theta^{E_b}\wedge \alpha =0$ for reasons of bidegree. A local calculation then shows that
\[
\sum_{k} \ol{\gamma}^k_b\wedge (\ol{\gamma}^k_b\vee\alpha) = \Theta^{E_b}\wedge \Lambda_{\omega_b}\alpha = \left[\Theta^{E_b},\Lambda_{\omega_b}\right]\alpha.
\]
Integrating both sides of \eqref{2ndfunest1} on $X_b$, we get
\begin{equation}\label{2ndfunest2}
    \abs{\ipr{ \mu,\alpha}_b}  \leq \left[ \sum_{i,j=1}^m \ipr{\Theta^E_{\xi_{\tau_j}\ol{\xi}_{\tau_i}}u_j,u_i }_b \right]^{1/2} \ipr{\left[\Theta^{E_b},\Lambda_{\omega_b}\right]\alpha,\alpha}_b^{1/2}. 
\end{equation}
On the other hand, we have 
\begin{align}
    \abs{\ipr{\nabla^{1,0}\kappa, \alpha}_b} = \abs{\ipr{\kappa,\nabla^{1,0*}\alpha}_b} &\leq \ipr{\kappa, \kappa}_b^{1/2}\ipr{\nabla^{1,0*}\alpha,\nabla^{1,0*}\alpha}_b^{1/2} = \ipr{\kappa, \kappa}_b^{1/2}\ipr{\Box^{1,0}\alpha,\alpha}_b^{1/2}, \label{2ndfunest3}
\end{align}
where the last equality follows from the fact that $\nabla^{1,0}\alpha=0$ for degree reasons.
Using equations \eqref{2ndfunest2} and \eqref{2ndfunest3}, we have
\begin{align*}
    \abs{\ipr{\mu+\nabla^{1,0}\kappa,\alpha}_b}^2 &\leq \abs{\ipr{ \mu,\alpha}_b}^2 + \abs{\ipr{\nabla^{1,0}\kappa, \alpha}_b}^2 + 2\abs{\ipr{ \mu,\alpha}_b} \abs{\ipr{\nabla^{1,0}\kappa, \alpha}_b} \\
    &\leq \left[ \sum_{i,j=1}^m \ipr{\Theta^E_{\xi_{\tau_j}\ol{\xi}_{\tau_i}}u_j,u_i }_b \right] \ipr{\left[\Theta^{E_b},\Lambda_{\omega_b}\right]\alpha,\alpha}_b + \ipr{\kappa, \kappa}_b\ipr{\Box^{1,0}\alpha,\alpha}_b \\
    &+ 2\left[ \sum_{i,j=1}^m \ipr{\Theta^E_{\xi_{\tau_j}\ol{\xi}_{\tau_i}}u_j,u_i }_b \right]^{1/2} \ipr{\left[\Theta^{E_b},\Lambda_{\omega_b}\right]\alpha,\alpha}_b^{1/2} \ipr{\kappa, \kappa}_b^{1/2}\ipr{\Box^{1,0}\alpha,\alpha}_b^{1/2} \\
    &\leq \left[ \sum_{i,j=1}^m \ipr{\Theta^E_{\xi_{\tau_j}\ol{\xi}_{\tau_i}}u_j,u_i }_b \right] \ipr{\left[\Theta^{E_b},\Lambda_{\omega_b}\right]\alpha,\alpha}_b + \ipr{\kappa, \kappa}_b\ipr{\Box^{1,0}\alpha,\alpha}_b \\
    &+ \left[ \sum_{i,j=1}^m \ipr{\Theta^E_{\xi_{\tau_j}\ol{\xi}_{\tau_i}}u_j,u_i }_b \right]\ipr{\Box^{1,0}\alpha,\alpha}_b + \ipr{\left[\Theta^{E_b},\Lambda_{\omega_b}\right]\alpha,\alpha}_b \ipr{\kappa, \kappa}_b\\
    &= C\ipr{\ipr{\Box^{1,0}\alpha,\alpha}_b+ \ipr{\left[\Theta^{E_b},\Lambda_{\omega_b}\right]\alpha,\alpha}_b}  \\
    &= C\ipr{\Box\alpha, \alpha}_b,
\end{align*}
where we used the formal identity $\Box^{1,0}+\left[\Theta^{E_b},\Lambda_{\omega_b}\right] =\Box$ to arrive at the last equality. This completes the proof.
\end{proof}

\begin{proposition}\label{SteiniBLS}
Suppose that the ambient manifold $\wt{X}$ is Stein and that $\Sigma_b^{\theta}(b)$ is dense in $\wt{\ch}_{\cl^{\theta},b}(b)$ for all $b\in B$. Then $(\ch, \Sigma^{\theta},\cl^{\theta})$ is an iBLS field.
\end{proposition}

\begin{proof}
Since we assume that $\Sigma^\theta_b(b)$ is dense in $\wt{\ch}_{\cl^\theta,b}$, by Proposition \eqref{liftinvintro} it follows that $\Sigma^\theta$ formalizes $\ch$. The only other condition that needs to be checked is the continuity of the anti-linear functional $Q^{\theta}_{\sigma\ol\tau}(\mathfrak{f}(b),\cdot)_b$ for all $b\in B$, all $\sigma, \tau \in \sC^{\infty}(T^{1,0}_B)_b$ and all $\mathfrak{f} \in \Sigma_b$. In order to do this, it suffices to check that the second fundamental map is bounded, i.e., for all $\sigma \in \sC^{\infty}(T^{1,0}_B)_b$ and all $\mathfrak{g}\in \Sigma_b$ we have $\norm{\cp^{\perp}\nabla^{\cl^\theta(1,0)}_\sigma \mathfrak{g}}_b \leq C\norm{\mathfrak{g}}_b$. We use the estimate \eqref{2ndfunestimate} to accomplish this. 

Let $r$ be a smooth plurisubharmonic exhaustion function for $\wt{X}$ which can be assumed to be positive. Since the metric $(E,h)\to \wt{X}$ is smooth, we may assume that for sufficiently large positive number $A$, the twisted metric $he^{-Ar}$ for $E\to V$ has nonnegative curvature, where $V$ is a compactly contained open neighborhood of the compact set $\ol{\pi^{-1}(U)}$, where $U$ is a small enough open set containing $b$. Since $r$ is an exhaustion function, we may choose $M$ large enough so that $\ol{V} \subset r^{-1}(0,\infty)$. Then by Theorem \ref{2ndfunprop} we have
\begin{align*}
    \norm{\cp^{\perp}\nabla^{\cl^\theta(1,0)}_\sigma \mathfrak{g}}_b^2 &= \int_{X_b} \left\{ P_b^{\perp} L^{1,0}_{\sigma}u, P_b^{\perp}L^{1,0}_{\xi_\sigma} u\right\}dV \\
    &\leq e^{AM} \int_{X_b} \left\{ P_b^{\perp} L^{1,0}_{\sigma}u, P_b^{\perp}L^{1,0}_{\xi_\sigma} u\right\}e^{-Ar} dV \\
    &\leq e^{AM} \left[ \int_{X_b}\{(\Theta^E_{\xi_\sigma, \ol\xi_\sigma}+A\partial\dbar r(\xi_\sigma, \ol\xi_\sigma))u,u\}e^{-Ar}dV + \int_{X_b} \{\dbar \xi_\sigma\lrcorner u, \dbar\xi_\sigma \lrcorner u \}e^{-Ar}dV\right]\\
    &\leq e^{AM}C_0 \norm{u}_b^2  = e^{AM}C_0 \norm{\mathfrak{g}}_b^2.
\end{align*}
Here we used the fact that $e^{-Ar}<1$ and that $\dbar\xi_{\sigma}$ and $\Theta^E_{\xi_\sigma, \ol\xi_\sigma}+A\partial\dbar r(\xi_\sigma, \ol\xi_\sigma)$ are smooth on $\ol{X}_b$. This completes the proof. 
\end{proof}


\subsubsection{Examples of iBLS fields}\label{examples}
In all the examples below we show that $\Sigma_b(b)$ is dense in $\ch_b$ for all $b\in B$. By Proposition \ref{SteiniBLS}, the field of Bergman spaces associated to these families of domains is an iBLS field. 
\begin{enumerate}
   \item \emph{Trivial fibration.} Let $B\subset \C^m$ be the unit ball, let $\wt{X} = B\times \C^n$ and let $E\to \wt{X}$ be the trivial line bundle with a nontrivial metric. Let $\Omega \subset \C^n$ be a smoothly bounded pseudoconvex domain. Then we can take $X = B\times \Omega \subset \wt{X}$ and $\pi :X\to B$ to be the projection onto the first factor. Define a smooth structure $\sC^{\infty}(\cl)$ for the Hilbert field of square integrable canonical sections by
   \[
   \Gamma(U,\sC^{\infty}(\cl)) = \{ [fd\vec{z}\,] : f\in \sC^{\infty}(U\times \Omega) \text{ such that } f(b,\cdot)d\vec{z} \in \cl_b \text{ for all } b\in U \},
   \]
   where $d\vec{z} := dz^1\wedge \cdots \wedge dz^n$ and $[ fd\vec{z}\,]$ denotes the relative canonical section represented by $fd\vec{z}$. Let $\sC^\infty(\ch)$ be the smooth structure for $\ch$ given by
   \[
   \Gamma(U,\sC^\infty(\ch)) = \{ [fd\vec{z}\,]: f\in \sC^{\infty}(U\times \Omega) \text{ such that } f(b,\cdot)d\vec{z}\in \ch_b \text{ for all }b\in U\}.
   \]
    Note that all fibers $\ch_b$ can be canonically identified with $\ch_0$ as the metric for $E$ is smooth on $\ol{X}$. Thus, given any element $fd\vec{z}\in \ch_b$ there is a relative canonical section $[Fd\vec{z}\,]$ on $X$ such that $fd\vec{z} = \iota_{X_b}^* Fd\vec{z}$. Here $Fd\vec{z}$ is the extension of $fd\vec{z}$ as a `constant section' 
    \[
    F(b',\cdot) d\vec{z} := fd\vec{z} \quad \text{ for all } b'\in B.
    \]
    Since the $(n,1)$-form $\dbar(Fd\vec{z})\equiv 0$ for such extensions, we see that $[Fd\vec{z}\,]$ satisfies the conditions in \eqref{concsigma}. Thus for all $b\in B$, by taking $\Sigma_b = \sC^{\infty}(\ch)_b$ we see that the image of $\Sigma_b$  under the evaluation at map is all of $\ch_b$. As discussed in the introduction, the compactness of the Neumann operator of $\Omega$ at the level of $(0,1)$-forms is not necessary for Theorem \eqref{Bgeneral} to hold.

    \item Let $E\to \wt{X}$ a hermitian holormophic line bundle over an $n+m$-dimensional complex manifold $\wt{X}$ and let $\wt{\pi}:\wt{X}\to B$ be a holomorphic submersion, where $B\subset \C^m$ is the unit ball. Take $X\subset \wt{X}$ to be the domain $X = \{ \rho <0\}$, where $\rho$ is a smooth real valued function on $\wt{X}$ such that for all $b\in B$ (i) $\rho|_{\wt{\pi}^{-1}(b)}$ is a strictly plurisubharmonic function in a neighborhood of the closure of $X_b := X\cap \wt{\pi}^{-1}(b)$ and (ii) $d\rho(x)\neq 0$ for all $x \in \partial X_b$. Take $\pi:X\to B$ be the map $\wt{\pi}|_X$. Then the work of X. Wang (see \cite{W2017}) shows that $\ch$ can be equipped with the structure of a BLS-field.
    
    \item Let $\wt{X} = B\times \C^n$ where $B\subset \C^m$ is a domain, let $E\to \wt{X}$ be the trivial line bundle with a nontrivial metric $e^{-\phi}$ and let $\wt{\pi}:\wt{X}\to B$ be the projection onto the first factor. Then take $X\subset \wt{X}$ to be the domain $X = \{ \rho <0\}$, where $\rho$ is a smooth real valued function on $\wt{X}$ such that for all $b\in B$ (i) $\rho|_{\wt{\pi}^{-1}(b)}$ is a plurisubharmonic function and (ii) $d\rho(x)\neq 0$ for all $x \in \partial X_b$. Take $\pi:X\to B$ be the map $\wt{\pi}|_X$.

Let $\Sigma\to B$ be the substalk bundle as follows. The stalk $\Sigma_b$ consists of germs of $\sC^{\infty}(\cl)$ that are associated to the relative canonical section $[\wt{p}_2^*f|_{\pi^{-1}(U)}]$ where $f\in \Gamma(\C^n, \mathcal{O}(\Lambda^{n,0}_{\C^n}))$, $U$ is an open set containing $b$ and $\wt{p}_2: \wt{X}\to \C^n$ is the projection onto the second factor. Since $\dbar [\wt{p}_2^*f] \equiv 0$, it follows that such sections satisfy \eqref{concsigma}. Thus, to show that $\Sigma_b(b)$ is dense in $\ch_b$ for all $b \in B$ it suffices to show that $(n,0)$-forms $fd\vec{z}|_{X_b}$ with $f$ entire are dense in $\ch_b$. This is done by an argument similar to \cite[Proposition 4.1]{BL2016}, which we sketch now. For each $b \in B$ the domain $X_b$ has a plurisubharmonic defining function, so by a result of Boas-Straube (see \cite{boasstraube}) the Neumann operator $\mathcal{N}^{(p,q)}_b$ is globally regular for all $1\leq q \leq n$. This in turn implies that forms smooth up to the boundary are dense in $\ch_b$ for all $b\in B$. Indeed, given $f \in \ch_b$ let $\{\chi_j\}$ be a sequence of cutoff functions with compact support in $X_b$ so that $\chi_j f\to f$ in $\cl_b$. By the global regularity of $\mathcal{N}_b^{(n,1)}$, the forms $P_b(\chi_j f)$ are smooth up to the boundary and 
\[
\norm{P_b(\chi_j f) - f}_{\ch_b} = \norm{ P_b(\chi_j f- f)}_{\ch_b} \leq \norm{\chi_j f - f}_{\cl_b},
\]
so $P(\chi_j f)\to f$ in $\ch_b$. Thus to show that $\Sigma_b(b)$ is dense in $\ch_b$, it suffices to show that for all forms $f\in \ch_b$ smooth up to the boundary, there is a sequence $\{f_j\} \subset \Gamma(\C^n, \mathcal{O}(\Lambda^{n,0}_{\C^n}))$ such that $f_j|_{X_b}\to f$ in $\ch_b$. Extend $f$ to a smooth form $F$ with compact support in $\C^n$. Let $\rho_b := \rho|_{\wt{\pi}^{-1}(b)}$ and let $u_j$ be the solution to $\dbar u_j =  \dbar F$ satisfying the $L^2$-estimate
\[
\int_{\C^n} \abs{u_j}^2e^{-j \max(\rho_b,0)+\abs{z}^2}dV \leq C_1 \int_{\C^n} \abs{\dbar F}^2 e^{-j\max(\rho_b,0)+\abs{z}^2}dV.
\]
Then $f_j :=F-u_j$ are holomorphic on $\C^n$ and $f_j|_{X_b}\to f$ in $\ch_b$ because $\norm{u_j}_{\ch_b} \to 0$. Indeed
\begin{align*}
\int_{X_b} \abs{u_j}^2 e^{-\phi}dV & \leq C_2 \int_{X_b} \abs{u_j}^2 e^{-j\max(\rho_b,0)+\abs{z}^2}dV \\
				               &	\leq C_2 \int_{\C^n} \abs{u_j}^2e^{-j \max(\rho_b,0)+\abs{z}^2}dV  \\
					      &\leq C_1C_2 \int_{\C^n} \abs{\dbar F}^2 e^{-j\max(\rho_b,0)+\abs{z}^2}dV \\
					       &= C_1 C_2 \int_{\{\rho_b>0\}} \abs{\dbar F}^2 e^{-j\max(\rho_b,0)+\abs{z}^2}dV
\end{align*}
and the last quantity goes to zero as $j\to \infty$.

\item Let $\wt{X} = B\times \C^n$ where $B\subset \C^m$ is a domain, let $E\to \wt{X}$ be the trivial line bundle with a nontrivial metric and let $\wt{\pi}:\wt{X}\to B$ be the projection onto the first factor. Take $X\subset \wt{X}$ to a smoothly bounded domain such that the domains $X_b := X\cap \wt{\pi}^{-1}(b) \subset \C^n$ admit good Stein neighborhood bases in the sense of \cite{straube2001}. Take $\pi:X\to B$ be the map $\wt{\pi}|_X$. By the results of \cite{straube2001}, the Neumann operator $\mathcal{N}^{(p,q)}_b$ is compact, and hence globally regular for all $0\leq p\leq n$, $1\leq q\leq n$. Taking $\Sigma \to B$ to be the substalk bundle in example (3) and following the same argument shows that elements of $\Sigma_b$ satisfy \eqref{concsigma} for all $b\in B$. For each $b\in B$, by \cite[Theorem 1.5]{SL-T2020} the holomorphic forms in $\C^n$ are dense in $\ch_b$, consequently $\Sigma_b(b)$ is dense in $\ch_b$. 

\item Let $\wt{X}$ be a Stein manifold and $E\to \wt{X}$ be a line bundle. Let $X\subset \wt{X}$ be a domain so that for every $b\in B$ the fibration $X\to B$ satisfies (i) $X_b$ is such that the Neumann operator $\mathcal{N}_b^{(0,1)}$ acting on $(0,1)$-forms is compact (ii) $X_b$ admits a neighborhood of Stein domains in $\wt{X}_b$. Let $\sC^{\infty}(\cl)$ be the smooth structure as in Section \ref{concL}. Let $\Sigma\to B$ be the following substalk bundle. The stalk $\Sigma_b$ contains (equivalence classes) of sections $F\in \Gamma(\pi^{-1}(U), \mathcal{O}(K_{X/B}\otimes E))$ where $U$ is an open set containing $b$ such that $\iota_{X_{b'}}^*F \text{ is smooth up to }\partial X_{b'} \text{ for all }b' \in U$. Since $F$ is a holomorphic twisted relative canonical section on $\pi^{-1}(U)$, it follows that $F$ satisfies the conditions in \eqref{concsigma}. The density of $\Sigma_b(b)$ in $\ch_b$ is a consequence of a result of Shaw-Laurent-Thi\'ebaut \cite[Theorem 1.5]{SL-T2020}. According to their theorem, global sections $F$ of $K_{\wt{X}_b}\otimes E|_{\wt{X}_b}\to \wt{X}_b$ are dense in $\ch_b$. Since $\wt{X}$ is Stein, such sections extend to holomorphic sections $\wt{F}$ of $\Lambda^{n,0}_{\wt{X}}\otimes E\to \wt{X}$. Thus, $\left[\wt{F}|_{\pi^{-1}(U)}\right]$ is a representative of a germ of $\Sigma_b$.

\end{enumerate}


\section{Curvature formula}

\subsection{Commutators of Lie derivatives.} Let $\xi, \eta$ be $(1,0)$ vector fields on a complex manifold $Z$, and let $E\to Z$ be a hermitian holomorphic line bundle. In this section we obtain the following formula for the commutator $\left[L^{1,0}_{\xi}, L^{0,1}_{\overline{\eta}} \right]$ acting on an $E$-valued $(n,0)$-form $u$.
\begin{equation}\label{liedercomm}
    \left[L^{1,0}_{\xi}, L^{0,1}_{\overline{\eta}} \right]u = \Theta^E_{\xi\overline{\eta}} u + L_{[\xi, \overline{\eta}]}u - \left[L^{0,1}_{\xi}, L^{1,0}_{\overline{\eta}} \right]u - \left[L^{0,1}_{\xi}, L^{0,1}_{\overline{\eta}} \right]u.
\end{equation}

Note that all the operators appearing in \eqref{liedercomm} are local. To work locally we introduce a local holomorphic frame $e$ for the line bundle $E$ and write $u = \vp \otimes e$, where $\vp$ is an $(n,0)$-form. If $\tau$ is a complex vector field on $X$ then we have
\begin{align*}
    L_{\tau}(\vp \otimes e) &= \tau \lrcorner \nabla(\vp \otimes e) + \nabla((\tau \lrcorner\vp)\otimes e) \\
                                   &= \tau\lrcorner (d \vp\otimes e +(-1)^n \vp\wedge \nabla e) + d(\tau \lrcorner \vp)\otimes e + (-1)^{n-1} (\tau \lrcorner \vp)\wedge \nabla e \\
                                  &= (\tau\lrcorner d \vp)\otimes e + (-1)^n (\tau \lrcorner \vp)\wedge \nabla e + \vp \otimes (\tau \lrcorner \nabla e) + d (\tau \lrcorner \vp)\otimes e + (-1)^{n-1} (\tau \lrcorner \vp)\wedge \nabla e  \\
                                  &= L_{\tau}\vp \otimes e + \vp \otimes (\tau\lrcorner \nabla e).
\end{align*}
Here we have used the same notation for the twisted Lie derivative acting on $E$-valued forms (appearing on the left hand side) and the usual Lie derivative (appearing on the right hand side). We will continue to employ this abusive notation throughout this subsection.
Similarly, we have
\begin{align} 
     L^{1,0}_{\tau} (\vp \otimes e) &= L^{1,0}_{\tau}\vp \otimes e + \vp \otimes (\tau \lrcorner \nabla^{1,0}e) \label{localie10} \\
     L^{0,1}_{\tau} (\vp \otimes e) &= L^{0,1}_{\tau} \vp \otimes e \label{localie01}.
\end{align}
Formula \eqref{localie01} follows from the fact that $\nabla$ is the Chern connection for $E$ so $\nabla^{0,1}=\dbar$ and $\tau \lrcorner \dbar e =0$ because $e$ is a holomorphic section of the line bundle $E$. 

We use formulas \eqref{localie01} and \eqref{localie10} repeatedly in the calculation below. First we obtain,
\begin{align*}
    L^{1,0}_{\xi} L^{0,1}_{\overline{\eta}}(\vp \otimes e) &= L^{1,0}_{\xi} (L^{0,1}_{\overline{\eta}}\vp \otimes e)\\
    &= L^{1,0}_{\xi}L^{0,1}_{\overline{\eta}}\vp \otimes e + L^{0,1}_{\overline{\eta}}\vp \otimes (\xi\lrcorner \nabla^{1,0} e)\\
    &= L^{1,0}_{\xi}L^{0,1}_{\overline{\eta}}\vp \otimes e + L^{0,1}_{\overline{\eta}}\vp \otimes \nabla^{1,0}_{\xi} e.
\end{align*}
On the other hand, 
\begin{align*}
    L^{0,1}_{\overline{\eta}}L^{1,0}_{\xi} (\vp \otimes e) &= L^{0,1}_{\overline{\eta}}(L^{1,0}_{\xi}\vp \otimes e + \vp\otimes (\xi \lrcorner \nabla^{1,0}e)) \\
    &= L^{0,1}_{\overline{\eta}}(L^{1,0}_{\xi}\vp \otimes e + \vp\otimes \nabla^{1,0}_{\xi}e) \\
    &= L^{0,1}_{\overline{\eta}}L^{1,0}_{\xi}\vp \otimes e + L^{0,1}_{\overline{\eta}}\vp \otimes \nabla^{1,0}_{\xi}e + \vp\otimes \nabla^{0,1}_{\overline{\eta}}\nabla^{1,0}_{\xi}e. \\
\end{align*}
Subtracting the two quantities, we get
\begin{align*}
    \left[ L^{1,0}_{\xi}, L^{0,1}_{\overline{\eta}}\right] (\vp\otimes e) &= \left[ L^{1,0}_{\xi}, L^{0,1}_{\overline{\eta}}\right]\vp \otimes e - \vp \otimes \nabla^{0,1}_{\overline{\eta}}\nabla^{1,0}_{\xi}e \\
    &= \left[ L^{1,0}_{\xi}, L^{0,1}_{\overline{\eta}}\right]\vp \otimes e + \vp \otimes \Theta^E_{\xi\overline{\eta}} e + \vp\otimes \nabla_{[\xi, \overline{\eta}]} e \\
    &= \left[ L^{1,0}_{\xi}, L^{0,1}_{\overline{\eta}}\right]\vp \otimes e + \vp \otimes \Theta^E_{\xi\overline{\eta}} e + \vp\otimes ([\xi, \overline{\eta}]\lrcorner \nabla e) \\
    &= \left[ L^{1,0}_{\xi}, L^{0,1}_{\overline{\eta}}\right]\vp \otimes e + \vp \otimes \Theta^E_{\xi\overline{\eta}} e + L_{[\xi, \overline{\eta}]}(\vp\otimes e) - (L_{[\xi, \overline{\eta}]}\vp)\otimes e \\
    &= \Theta^E_{\xi\overline{\eta}}(\vp\otimes e) + L_{[\xi, \overline{\eta}]}(\vp\otimes e) + \left( [L^{1,0}_{\xi},L^{0,1}_{\overline{\eta}}]\vp - L_{[\xi,\overline{\eta}]}\vp\right)\otimes e
\end{align*}
By the Jacobi identity, for a differential form $\psi$ on $X$ we have $[L_{\xi},L_{\ol{\eta}}]\psi = L_{[\xi,\ol{\eta}]}\psi$. Thus we get
\begin{align*}
    [L^{1,0}_{\xi},L^{0,1}_{\overline{\eta}}]\vp - L_{[\xi,\overline{\eta}]}\vp &= [L^{1,0}_{\xi},L^{0,1}_{\overline{\eta}}]\vp - [L_{\xi},L_{\ol{\eta}}]\vp \\
    &=[L^{1,0}_{\xi},L^{0,1}_{\overline{\eta}}]\vp - [L^{1,0}_{\xi}+L^{0,1}_{\xi},L^{1,0}_{\ol{\eta}}+L^{0,1}_{\ol{\eta}}]\vp \\
    &= -[L^{1,0}_{\xi},L^{1,0}_{\ol{\eta}}]\vp -[L^{0,1}_{\xi},L^{1,0}_{\ol{\eta}}]\vp - [L^{0,1}_{\xi},L^{0,1}_{\ol{\eta}}]\vp \\
    &= -[L^{0,1}_{\xi},L^{1,0}_{\ol{\eta}}]\vp - [L^{0,1}_{\xi},L^{0,1}_{\ol{\eta}}]\vp,
\end{align*}
where the last equality follows from the fact that $L^{1,0}_{\ol{\eta}}\vp=L^{1,0}_{\ol{\eta}}L^{1,0}_{\xi}\vp =0$ for reasons of type. For the same reason, we see that $\ol\eta \lrcorner \nabla^{1,0}e = 0$. Using this fact along with equations \eqref{localie10} and \eqref{localie01} we see that
\begin{align*}
    [L^{0,1}_{\xi},L^{1,0}_{\ol\eta}](\vp\otimes e) + [L^{0,1}_{\xi},L^{0,1}_{\ol\eta}](\vp\otimes e) &= \left([L^{0,1}_{\xi},L^{1,0}_{\ol\eta}]\vp \right)\otimes e + \left([L^{0,1}_{\xi},L^{0,1}_{\ol\eta}] \vp \right)\otimes e.
\end{align*}
Putting all this together, we obtain
\begin{align*}
    \left[ L^{1,0}_{\xi}, L^{0,1}_{\overline{\eta}}\right] (\vp\otimes e)
    &= \Theta^E_{\xi\overline{\eta}}(\vp\otimes e) + L_{[\xi, \overline{\eta}]}(\vp\otimes e) + \left( [L^{1,0}_{\xi},L^{0,1}_{\overline{\eta}}]\vp - L_{[\xi,\overline{\eta}]}\vp\right)\otimes e \\
    &=\Theta^E_{\xi\overline{\eta}}(\vp\otimes e) + L_{[\xi, \overline{\eta}]}(\vp\otimes e) - [L^{0,1}_{\xi},L^{1,0}_{\ol\eta}](\vp\otimes e) - [L^{0,1}_{\xi},L^{0,1}_{\ol\eta}](\vp\otimes e).
\end{align*}
This is equation \eqref{liedercomm}.

\begin{theorem}\label{smoothcurv}
Let $(E,h)\to X\to B$ be as in Section \ref{notation}. Fix a horizontal distribution $\theta^{1,0} \subset T^{1,0}_X$. Let $\sigma, \tau$ be holomorphic $(1,0)$-vector fields on an open set $U$ containing the point $b$ and $\xi_{\sigma}, \xi_{\tau}$ be their horizontal lifts with respect to the distribution $\theta^{1,0}$. Then for $\mathfrak{f}, \mathfrak{g} \in \Sigma_b$ we have 
\begin{equation*}
    \ipr{\Theta^{\mathcal{L}^{\theta}}_{\sigma\ol{\tau}}\mathfrak{f},\mathfrak{g}}_b = \ipr{\Theta^E_{\xi_{\sigma}\ol{\xi}_{\tau}} \mathfrak{a(f)}, \mathfrak{a(g)}}_b - \sqi^{n^2} \int_{X_b} \hwedge{\dbar \xi_{\sigma}\lrcorner \mathfrak{a(f)}}{\dbar \xi_{\tau}\lrcorner \mathfrak{a(g)}} + \int_{\partial X_b} \partial \dbar\rho(\xi_{\sigma},\ol{\xi}_{\tau})\{\mathfrak{a(f),a(g)}\}dS,
\end{equation*}
where $\rho$ is a defining function for $X$ such that $\rho_b := \rho|_{X_b}$ is a defining function for $\partial X_b$ normalized to satisfy $\abs{d\rho_b}=1$ on $\partial X_b$, and $dS$ is the volume form induced on $\partial X_b$ by the Kahler form $\omega_b$ on $X_b$, i.e., $d\rho_b \wedge dS = dV_{\omega_b}|_{\partial X_b}$.
\end{theorem}

\begin{proof}
Let $u$ and $v$ be $E$-valued $(n,0)$-forms on $\pi^{-1}(U)$ that represent the twisted relative canonical sections $\mathfrak{a(f)}$ and $\mathfrak{a(g)}$ associated to $\mathfrak{f}$ and $\mathfrak{g}$ respectively. Since $\sigma, \tau$ are holomorphic vector fields on $U$, we have $[\sigma, \ol\tau]=0$. Thus, the curvature $\Theta^{\cl^\theta}$ is given by 
\begin{align*}
    \Theta^{\cl^\theta}_{\sigma\ol\tau} \mathfrak{f} &= \nabla^{\cl^\theta(1,0)}_{\sigma}\nabla^{\cl^\theta(0,1)}_{\ol\tau}\mathfrak{f} - \nabla^{\cl^\theta(0,1)}_{\ol\tau}\nabla^{\cl^\theta(1,0)}_{\sigma}\mathfrak{f} \\
    &= \nabla^{\cl^\theta(1,0)}_{\sigma}\ipr{\mathfrak{i}\ipr{L^{0,1}_{\ol\xi_\tau}\mathfrak{a(f)}}} - \nabla^{\cl^\theta(0,1)}_{\ol\tau}\ipr{\mathfrak{i}\ipr{L^{1,0}_{\xi_\sigma}\mathfrak{a(f)}}} \\
    &= \mathfrak{i}\ipr{L^{1,0}_{\xi_\sigma}L^{0,1}_{\ol\xi_\tau}\mathfrak{a(f)}} - \mathfrak{i}\ipr{L^{0,1}_{\ol\xi_\tau}L^{1,0}_{\xi_\sigma}\mathfrak{a(f)}} = \mathfrak{i}\ipr{\left[L^{1,0}_{\xi_\sigma}, L^{0,1}_{\ol{\xi}_\tau} \right]\mathfrak{a(f)}}.
\end{align*}

By using formula \eqref{liedercomm}
\begin{align}
    \ipr{\Theta^{\mathcal{L}^{\theta}}_{\sigma\ol{\tau}}\mathfrak{f},\mathfrak{g}}_b &= \ipr{\left[L^{1,0}_{\xi_\sigma}, L^{0,1}_{\ol{\xi}_\tau} \right]u, v}_b \nonumber \\
    &=\sqi^{n^2} \int_{X_b} \hwedge{\Theta^E_{\xi_\sigma \overline{\xi}_\tau} u}{v} + \sqi^{n^2}\int_{X_b} \hwedge {L_{[\xi_\sigma, \overline{\xi}_\tau]}u}{v}  \nonumber \\
    &- \sqi^{n^2}\int_{X_b} \hwedge{\left[L^{0,1}_{\xi_\sigma}, L^{1,0}_{\overline{\xi}_\tau} \right]u}{v} - \sqi^{n^2} \int_{X_b} \hwedge{\left[L^{0,1}_{\xi_\sigma}, L^{0,1}_{\overline{\xi}_\tau} \right]u}{v}. \label{smoothcurveq}
\end{align}
We deal with each of the terms in the last line individually. We deal with the last term first. Note that 
$\left[L^{0,1}_{\xi_\sigma}, L^{0,1}_{\overline{\xi}_\tau} \right]u$ is an $E$-valued $(n-1,1)$-form, so that the integrand in the last term is an $(n-1,n+1)$-form. Thus, the last term vanishes. 

To simplify the third term, note that $L^{0,1}_{\xi_\sigma}L^{1,0}_{\ol{\xi}_\tau} u =0$ for reasons of type. Thus
\begin{align*}
    - \int_{X_b} \hwedge{\left[ L^{0,1}_{\xi_\sigma}, L^{1,0}_{\ol{\xi}_\tau}\right] u}{v} &= \int_{X_b} \hwedge{ L^{1,0}_{\ol{\xi}_\tau}L^{0,1}_{\xi_\sigma}u}{v} \\
    &= \int_{X_b} \hwedge{ L_{\ol{\xi}_\tau}L^{0,1}_{\xi_\sigma}u}{v} \\
    &= \int_{X_b} L_{\ol{\xi}_{\tau}}\hwedge{L^{0,1}_{\xi_\sigma}u}{v} - \hwedge{L^{0,1}_{\xi_\sigma}u}{L^{0,1}_{\xi_\tau}v} \\
    &= \ol{\tau} \int_{X_b} \hwedge{L^{0,1}_{\xi_\sigma}u}{v} - \int_{X_b} \hwedge{L^{0,1}_{\xi_\sigma}u}{L^{0,1}_{\xi_\tau}v}\\
    &= - \int_{X_b} \hwedge{L^{0,1}_{\xi_\sigma}u}{L^{0,1}_{\xi_\tau}v}.
\end{align*}
To get the second equality we use the fact that $\hwedge{L^{0,1}_{\ol{\xi}_{\sigma}}L^{0,1}_{\xi_\tau}u}{v}$ is an $(n-1,n+1)$-form so its integral vanishes. The third equality follows from equation \eqref{lieder}. The last equality follows from the fact that $\hwedge{L^{0,1}_{\xi_\sigma}u}{v}$ is an $(n-1,n+1)$-form, so its integral over the fiber $X_{\wt{b}}$ vanishes for all $\wt{b}$ near $b$.

To simplify the second term, note that $[\xi_\sigma, \ol\xi_\tau]$ is a vertical vector field, so
\[
\iota_{X_b}^* ([\xi_\sigma, \ol\xi_\tau]\lrcorner \nabla u) = \left([\xi_\sigma, \ol\xi_\tau]\Big|_{X_b} \right) \lrcorner \left(\iota_{X_b}^*\nabla u\right)=0,
\]
where the last equality follows because $\nabla^{1,0}u$ is an $(n+1,0)$-form, and $\iota_{X_b}^*\dbar u=0$ as $u$ is holomorphic along $X_b$. The same reasoning also shows that $\iota_{X_b}^* \nabla v=0$. Thus, we may write 
\begin{align*}
    \sqi^{n^2} \int_{X_b} \hwedge{L_{[\xi_\sigma, \ol\xi_\tau]}u}{v} &= \sqi^{n^2} \int_{X_b} \hwedge{\nabla ([\xi_\sigma, \ol\xi_\tau]\lrcorner u)}{v} \\
    &= \sqi^{n^2} \int_{X_b} d\hwedge{([\xi_\sigma, \ol\xi_\tau]\lrcorner u)}{v} + (-1)^n \hwedge{[\xi_\sigma, \ol\xi_\tau]\lrcorner u}{\nabla v}\\
    &= \sqi^{n^2} \int_{X_b} d\hwedge{([\xi_\sigma, \ol\xi_\tau]^{1,0}\lrcorner u)}{v} \\
    &= \sqi^{n^2}\int_{X_b} d\left( [\xi_\sigma, \ol\xi_\tau]^{1,0}\lrcorner \hwedge{u}{v}\right) \\
    &= \int_{X_b} d\left( [\xi_\sigma, \ol\xi_\tau]^{1,0}\lrcorner \{u,v\}dV_{\omega_b}\right) \\
    &= \int_{\partial X_b} \partial \rho_b([\xi_\sigma, \ol\xi_\tau]^{1,0}) \{u,v\}dS \\
    &= \int_{\partial X_b} \partial \dbar \rho(\xi_\sigma, \ol\xi_\tau)\{u,v\}dS.
\end{align*}
The explanation for the last equality which we have seen before, is as follows. Since $[\xi_\sigma, \ol\xi_\tau]^{1,0}$ is a vertical vector field we have
\begin{align*}
    \iota_{\partial X_b}^* \left(\left([\xi_\sigma, \ol\xi_\tau]^{1,0}\Big|_{X_b}\right) \lrcorner \partial\rho_b \right) = \iota_{\partial X_b}^* \left(\left([\xi_\sigma, \ol\xi_\tau]^{1,0}\Big|_{X_b}\right) \lrcorner \iota_{X_b}^* \partial \rho \right) &= \iota_{\partial X_b}^* \iota_{X_b}^* \left([\xi_\sigma, \ol\xi_\tau]^{1,0} \lrcorner \partial \rho \right) \\
    &= \iota_{X_b}^* \iota_{\partial X}^* \partial \dbar\rho(\xi_\sigma, \ol\xi_{\tau}),
\end{align*}
where the last equality follows because $\xi_\sigma, \xi_\tau$ are $(1,0)$-vector fields tangent to $\partial X$.

Applying these computations to the terms on the right hand side of equation \eqref{smoothcurveq} yields the desired formula.
\end{proof}

\begin{lemma}\label{primlift}
Let $(E,h)\to X\to B$ be as in Theorem \ref{smoothcurv}. Assume that $\mathcal{N}_b^{(0,2)}: L^2_{0,2}(X_b)\to L^2_{0,2}(X_b)$ is globally regular. Then for any holomorphic vector field $\tau$ on $B$ and any $E$-valued $(n,0)$-form $u$ on $X$ there is a horizontal lift $\xi_{\tau}$ such that the form $\iota_{X_b}^* (\dbar \xi_{\tau} \lrcorner u)$ is primitive.
\end{lemma}

\begin{proof}
Fix a preliminary lift $\xi^{\theta}_{\tau}$ of $\tau$ with respect to a distribution $\theta$. Then $\dbar \xi^{\theta}_{\tau}$ is a smooth $(0,1)$-form valued vertical vector field. Then we have
\[
\iota_{X_b}^* (\dbar \xi^{\theta}_{\tau} \lrcorner (\omega \wedge u))= \dbar (\xi^{\theta}_{\tau}\big|_{X_b}) \lrcorner \iota_{X_b}^* (\omega \wedge u) = 0,
\] 
where the last equality follows from the fact that $\omega \wedge u$ is an $E$-valued $(n+1,1)$-form.
Consequently, we get
\[
\iota_{X_b}^* ((\dbar \xi^{\theta}_{\tau} \lrcorner \omega) \wedge u) = - \iota_{X_b}^* (\omega \wedge (\dbar \xi^{\theta}_{\tau} \lrcorner u)) = -\omega_b \wedge \iota_{X_b}^* (\dbar \xi^{\theta}_{\tau} \lrcorner u).
\]
To show that $\iota_{X_t}^* (\dbar \xi_{\tau} \lrcorner u)$ is primitive it suffices to show that $\iota_{X_b}^* (\dbar \xi_{\tau} \lrcorner \omega)=0$. We will add a vertical vector field to $\xi^{\theta}_{\tau}$ that is tangent to $\partial X_b$ to get the desired lift. Note that 
\[
\dbar(\xi^{\theta}_{\tau} \lrcorner \omega) = \dbar \xi^{\theta}_{\tau} \lrcorner \omega - \xi^{\theta}_{\tau} \lrcorner \dbar \omega = \dbar \xi^{\theta}_{\tau} \lrcorner \omega,
\]
so $\iota_{X_b}^* (\dbar \xi^{\theta}_{\tau} \lrcorner \omega) = \iota_{X_b}^* \dbar (\xi^{\theta}_{\tau} \lrcorner \omega) = \dbar \iota_{X_b}^* (\xi^{\theta}_{\tau} \lrcorner \omega)$. Thus, $\iota_{X_b}^*(\dbar \xi^{\theta}_{\tau} \lrcorner \omega)$ is a $\dbar$-exact $(0,2)$-form on $X_b$, so it is orthogonal to the harmonic $(0,2)$-forms. By the global regularity of $\mathcal{N}_b^{(0,2)}$, the $(0,1)$-form $\beta = \dbarstar \mathcal{N}_b^{(0,2)} \iota_{X_b}^* (\dbar \xi^{\theta}_{\tau} \lrcorner \omega)$ is smooth up to the boundary of $X_b$ and satisfies $\dbar \beta = \iota_{X_b}^* (\dbar \xi^{\theta}_{\tau} \lrcorner \omega)$. 


Let $\eta$ be the $(1,0)$-vector field on $X_b$ obtained from $\beta$ by raising the index, i.e., $\beta = \eta \lrcorner \omega_b$. We claim that $\eta$ is tangent to $\partial X_b$. Suppose locally at a point on $\partial X_b$ we write $\beta = \beta_{\overline{j}} d\overline{z}^j$. Then we have
\[
\beta_{\overline{j}} d\overline{z}^j = \eta^i \frac{\partial}{\partial z^i} \lrcorner \left[ (\omega_b)_{l\overline{k}} dz^l \wedge d\overline{z}^k \right] = \eta^l (\omega_b)_{l\overline{k}} d\overline{z}^k.
\]
Thus, we have that $\eta^i = \omega_b^{i\overline{k}} \beta_{\overline{k}}$ and 
\[
\eta \rho = \eta \rho_b = \omega_b^{i\overline{k}} \beta_{\overline{k}} \frac{\partial \rho_b}{\partial z^i} = 0,
\]
where the last equality is the condition that $\beta \in \text{Dom}(\dbarstar)$. Now $\eta$ extends to a vector field $\wt{\eta}$ on $X$ that is tangent to $\partial X$ and $\xi_{\tau} = \xi^{\theta}_{\tau} - \wt{\eta}$ is the desired lift.
\end{proof}

\subsection*{Proof of Theorem \ref{Bgeneral}}
By example (5) in Section \ref{examples}, $(\ch, \Sigma, \cl)$ is an iBLS field. Let $\mathfrak{f}_1,\dots, \mathfrak{f}_m \in \Sigma_b$ and let $\tau_1, \dots, \tau_m$ be holomorphic $(1,0)$-vector fields on $U$. Let $u_j$ be an $E$-valued $(n,0)$-form representing the twisted relative canonical section $\mathfrak{a}(\mathfrak{f}_j)$ for $1\leq j\leq m$. Since the curvature of $\ch$ is independent of the lift, by Lemma \ref{primlift} we can choose a horizontal lift $\xi_{\tau_j}$ of vector field $\tau_j$ on $B$ so that $\iota_{X_b}^*(\dbar \xi_{\tau_j}\lrcorner u_j)$ is primitive on $X_b$ for $j=1,\dots ,m$. Thus, 
\[
- \sum_{j,k=1}^m\sqi^{n^2} \int_{X_b} \hwedge{\dbar \xi_{\tau_j}\lrcorner u_j}{\dbar \xi_{\tau_k}\lrcorner u_k} = \sum_{j,k=1}^m \ipr{ \dbar \xi_{\tau_j}\lrcorner u_j, \dbar \xi_{\tau_k}\lrcorner u_k}_b.
\]
Recall that
\[
\sum_{j,k=1}^m \ipr{\Theta^{\ch}_{\tau_j,\ol{\tau}_k}\mathfrak{f}_j,\mathfrak{f}_k}_b = \sum_{j,k=1}^m \ipr{\Theta^{\cl}_{\tau_j,\ol{\tau}_k}\mathfrak{f}_j,\mathfrak{f}_k} -  \sum_{j,k=1}^m \ipr{\cp^{\perp}\nabla^{\cl}_{\tau_j}\mathfrak{f}_j, \cp^{\perp} \nabla^{\cl}_{\tau_k} \mathfrak{f}_k}_b.
\]
Thus, by Theorem \ref{smoothcurv} and equation \eqref{2ndfunestimate} it follows that 
\begin{align*}
\sum_{j,k=1}^m \ipr{\Theta^\ch_{\tau_j\ol\tau_k}\mathfrak{f}_j, \mathfrak{f}_k}_b &\geq \sum_{j,k=1}^m \int_{\partial X_b} \partial \dbar \rho(\xi_{\tau_j},\ol{\xi}_{\tau_k})\{u_j,u_k\}dS_b \geq 0,
\end{align*}
where the last inequality follows from the fact that $\partial X$ is pseudoconvex. This completes the proof.

\section{Exact formulas for the second fundamental form}

In this section we derive an exact formula for the second fundamental form with the extra assumption that $\wt{X}$ is a Stein manifold. The following proposition makes use of the Stein assumption on $\wt{X}$.

\begin{proposition}\label{Hodge0}
Let $D$ be a bounded pseudoconvex domain in an $n$-dimensional Stein manifold $M$, and let $E\to M$ be a hermitian holomorphic line bundle. The Neumann operator $N^{(n,0)}:\Gamma(D,L^2(\Lambda^{n,0}_D\otimes E)) \to \Gamma(D,L^2(\Lambda^{n,0}_D\otimes E))$ satisfies 
\[
\text{Im}(N^{(n,0)})\subset \text{Dom}(\Box)  \quad \text{and} \quad N^{(n,0)}\Box = \Box N^{(n,0)} = \text{Id} - P,
\]
where $P$ is the Bergman projection. We also have
\begin{align}
u &= \dbarstar\dbar N^{(n,0)}u + Pu \quad \text{for all } u\in \Gamma(D,L^2(\Lambda^{n,0}_D\otimes E)) \label{Hodge01} \\
\dbar N^{(n,0)}u &= N^{(n,1)}\dbar u\quad \text{for all }  u \in \text{Dom}(\dbar) \subset \Gamma(D,L^2(\Lambda^{n,0}_D\otimes E)) \label{Hodge02} \\
\dbarstar N^{(n,1)}u &= N^{(n,0)}\dbarstar u \quad \text{for al } u \in \text{Dom}(\dbarstar)\subset \Gamma(D,L^2(\Lambda^{n,1}_D\otimes E)) \label{Hodge03} \\
N^{(n,0)} &= \dbarstar N^{(n,1)} N^{(n,1)} \dbar.  \label{Hodge04}
\end{align}
\end{proposition}
This proposition is taken from \cite[Theorem 4.4.3]{chenshaw}. Though the theorem is stated for the trivial line bundle over domains in $\C^n$, the proof can be adapted to work in the setting of general line bundles over domains in Stein manifolds.

\begin{proposition}[Berndtsson's choice of representatives]\label{Bchoice}
Let $E, X, \wt{X}, B$ be as in Section \ref{notation}. Further assume that $\wt{X}$ is a Stein manifold. Assume that the Neumann operator $N_b^{n,q}$ acting on $E|_{X_b}$-valued $(n,q)$-forms is globally regular for $1\leq q\leq n$ i.e., it maps the space $\Gamma(\ol{X}_b, \sC^{\infty}(E|_{X_b}\otimes \Lambda^{n,q}_{X_b}))$ to itself. Let $f\in \Gamma(X, \sC^{\infty}(K_{X/B}\otimes E))$ be a twisted relative canonical section so that $\iota_b f \in \ch_b$. Then for any holomorphic $(1,0)$-vector field $\tau$ on $B$ and any horizontal lift $\xi_{\tau}$ of $\tau$ to $X$ there is a representative $u\in \Gamma(\ol{X},\sC^{\infty}(\Lambda^{n,0}_X\otimes E))$ of $f$ so that $\iota_{X_b}^*(\xi_\tau \lrcorner \nabla^{1,0}u)$ is a holomorphic $E|_{X_b}$-valued $(n,0)$-form on $X_b$

\end{proposition}

\begin{proof}
Note that the holomorphicity of $\iota_{X_b}^*(\xi_\tau\lrcorner \nabla^{1,0}u)$ is independent of the lift $\xi_\tau$. To see this, suppose $\xi^1_{\tau}$ and $\xi^2_{\tau}$ are two lifts of the vector field $\tau$. Then $v_\tau = \xi^1_\tau - \xi^2_\tau$ is a vertical vector field on $X$.Then 
\[
\iota_{X_b}^*(v_\tau\lrcorner \nabla^{1,0}u) = v_\tau|_{X_b}\lrcorner \iota_{X_b}^*(\nabla^{1,0}u) = 0,
\]
because $\nabla^{1,0}u$ is an $(n+1,0)$-form, so $\dbar \iota_{X_b}^*(\xi^1_{\tau}\lrcorner \nabla^{1,0}u) = \dbar \iota_{X_b}^*(\xi^2_{\tau}\lrcorner \nabla^{1,0}u)$.

Let $\wt{u}$ be an initial representative of the twisted relative canonical section $f$. Let $(t^1,\dots, t^m)$ be local coordinates centered at $b\in B$. Then we may write 
\[
\nabla^{1,0}\wt{u} = dt^k\wedge \mu_k,
\]
where $\mu_k$ are $E$-valued $(n,0)$-forms on $X$. Any other representative $u$ of $f$ is of the form $u = \wt{u} + dt^k\wedge w_k$ for $E$-valued $(n-1,0)$-forms $w_k$.
We calculate
\[
\nabla^{1,0}u = dt^k \wedge (\mu_k - \nabla^{1,0} w_k).
\]
Consequently,
\[
\iota_{X_b}^*(\xi_\tau \lrcorner \nabla^{1,0}u) = \tau^k(b)\iota^*_{X_b}(\mu_k - \nabla^{1,0}w_k),
\]
where $\tau = \tau^k\frac{\partial}{\partial t^k}$.
Thus, for $\iota_{X_b}^*(\xi_\tau \lrcorner \nabla^{1,0}u)$ to be holomorphic we want $P_b^{\perp}\iota_{X_b}^*(\mu_k -\nabla^{1,0}w_k)=0$ for all $k$. 
Let $W_k$ the $E|_{X_b}$-valued $(n-1,0)$-form 
\[
W_k := \sqi \Lambda_{\omega_b}\dbar N^{(n,0)}_b P_b^{\perp} \iota_{X_b}^*\mu_k.
\]
Then
\begin{align*}
    \nabla^{1,0}W_k &= \sqi \nabla^{1,0}\Lambda_{\omega_b}\dbar N^{(n,0)}_b P_b^{\perp} \iota_{X_b}^*\mu_k \\
                     &= \sqi [\nabla^{1,0},\Lambda_{\omega_b}]\dbar N^{(n,0)}_b P_b^{\perp} \iota_{X_b}^*\mu_k \\
                     &= \dbarstar \dbar N^{(n,0)}_b P_b^{\perp} \iota_{X_b}^*\mu_k \\
                     &= P_b^{\perp} \iota_{X_b}^*\mu_k.
\end{align*}
A few remarks are in order. Since $\iota_{X_b}^*\mu_k$ is a twisted form smooth up to the boundary and $N^{(n,1)}_b$ is globally regular, we see that $P_b^{\perp} \iota_{X_b}^*\mu_k$ is smooth up to the boundary. Thus the twisted form $P_b^{\perp} \iota_{X_b}^*\mu_k$ is in $\text{Dom}(\dbar)$. The identity $N^{(n,0)}_b = \dbarstar N_b^{(n,1)} N_b^{(n,1)} \dbar$ shows that $N^{(n,0)}_b$ is globally regular. Thus the twisted form $\dbar N^{(n,0)}_b P_b^{\perp} \iota_{X_b}^*\mu_k$ is smooth up to the boundary of $X_b$ and is in $\text{Dom }\dbarstar$. Recall that on such forms $\dbarstar = \vartheta$, where $\vartheta$ denotes the formal adjoint of $\dbar$. This combined with the K\"ahler identity $\sqi [\nabla^{1,0},\Lambda_{\omega_b}] = \vartheta$, yields the penultimate equality.  

Thus, $W_k$ is smooth on $\ol{X}_b$ and taking $w_k$ to be any smooth extension of $W_k$ to $\ol{X}$, we see that $u = \wt{u} + dt^k\wedge w_k$ is the desired representative. 
\end{proof}

\begin{remark}
With more work, we could have chosen a representative $u$ so that $\iota_{X_b}^*(\xi_\tau \lrcorner \nabla^{1,0}u)$ is holomorphic and  $\iota_{X_b}^*(\xi_\tau \lrcorner \dbar u)$ is primitive. The choice of a representative with these two properties is attributed to B. Berndtsson (see \cite{B2009}, also \cite{V2021}). Since we do not use the primitivity of the form $\iota_{X_b}^*(\xi_\tau\lrcorner  \dbar u)$, we avoid the extra work.  
\end{remark}

\begin{theorem}\label{2ndfunexact}
Let $(E,h)\to X \to B$ be as in Proposition \ref{Bchoice}. Let $\sigma, \tau \in \sC^{\infty}(T^{1,0}_B)_b$ be germs of holomorphic $(1,0)$-vector fields and let $\xi_{\sigma}, \xi_{\tau}$ be their horizontal lifts. Let $\mathfrak{f,g} \in \Sigma_b$ where $\mathfrak{f}(b) = \iota_{X_b}^*u$ and $\mathfrak{g}(b) = \iota_{X_b}^*v$ for $E$-valued $(n,0)$-forms $u$ and $v$. Then  
\begin{align*}
    \ipr{\cp^{\perp}\nabla^{\cl(1,0)}_{\sigma}\mathfrak{f},\cp^{\perp}\nabla^{\cl(1,0)}_{\tau}\mathfrak{g}}_b &=  -\frac{1}{2}  \ipr{N^{(n,1)}_b\iota_{X_b}^*\dbar\nabla^{1,0}(\xi_\sigma \lrcorner u), \iota_{X_b}^*(\xi_{\tau}\lrcorner \Theta^{E})v+\iota_{X_b}^*\nabla^{1,0}(\dbar\xi_{\tau}\lrcorner v)}_b  \\
    &\quad- \frac{1}{2} \ipr{ \iota_{X_b}^*(\xi_{\sigma}\lrcorner \Theta^{E})u+\iota_{X_b}^*\nabla^{1,0}(\dbar\xi_{\sigma}\lrcorner u), N^{(n,1)}_b\iota_{X_b}^*\dbar\nabla^{1,0}(\xi_\tau \lrcorner v)}_b.
\end{align*}
\end{theorem}

\begin{proof}
By proposition \ref{dbarPL} 
\begin{equation}\label{exact2}
\dbar \iota_{X_b}^*(L^{1,0}_{\xi_\tau}v) = -\iota_{X_b}^*(\xi_\tau\lrcorner \Theta^E)v - \iota_{X_b}^*\nabla^{1,0}(\dbar\xi_{\tau}\lrcorner v).
\end{equation} 
Note that 
\[
P^{\perp}_b\iota_{X_b}^*\nabla^{1,0}(\xi_{\sigma}\lrcorner u) = \dbarstar N^{(n,1)}_b\dbar \iota_{X_b}^*\nabla^{1,0}(\xi_{\sigma}\lrcorner u). 
\]
By proposition \ref{Bchoice} we can take $u$ to be an $E$-valued $(n,0)$-form such that $\iota_{X_b}^*(\xi_{\sigma}\lrcorner \nabla^{1,0} u)$ is a holomorphic section of $K_{X_b}\otimes E|_{X_b}\to X_b$. Then
\begin{align}
 \ipr{\cp^{\perp}\nabla^{\cl(1,0)}_{\sigma}\mathfrak{f},\cp^{\perp}\nabla^{\cl(1,0)}_{\tau}\mathfrak{g}}_b &= \ipr{P^{\perp}_b \iota_{X_b}^*(L^{1,0}_{\xi_\sigma}u), P^{\perp}_b \iota_{X_b}^*(L^{1,0}_{\xi_\tau}v) }_b \nonumber \\
&= \ipr{ P^{\perp}_b \iota_{X_b}^*(\nabla^{1,0}\xi_{\sigma}\lrcorner u), P^{\perp}_b \iota_{X_b}^*(L^{1,0}_{\xi_{\tau}}v)}_b \nonumber \\
&= \ipr{\dbarstar N^{(n,1)}_b\dbar \iota_{X_b}^*\nabla^{1,0}(\xi_{\sigma}\lrcorner u), P^{\perp}_b \iota_{X_b}^*(L^{1,0}_{\xi_\tau} v)}_b \nonumber \\
&= \ipr{ N^{(n,1)}_b\dbar \iota_{X_b}^*\nabla^{1,0}(\xi_{\sigma}\lrcorner u),  \dbar P^{\perp}_b \iota_{X_b}^*(L^{1,0}_{\xi_\tau} v)}_b \nonumber \\
&= \ipr{ N^{(n,1)}_b\dbar \iota_{X_b}^*\nabla^{1,0}(\xi_{\sigma}\lrcorner u), \dbar \iota_{X_b}^*(L^{1,0}_{\xi_\tau} v)}_b \nonumber \\
&= -\ipr{ N^{(n,1)}_b\dbar \iota_{X_b}^*\nabla^{1,0}(\xi_{\sigma}\lrcorner u), \iota_{X_b}^*(\xi_\tau\lrcorner \Theta^E)v + \iota_{X_b}^*\nabla^{1,0}(\dbar\xi_{\tau}\lrcorner v)}_b, \label{exact3}
\end{align}
where the last equality follows from \eqref{exact2}.

Similarly, we can take $v$ to be an $E$-valued $(n,0)$-form such that $\iota_{X_b}^*(\xi_\tau\lrcorner \nabla^{1,0}v)$ is a holomorphic section of $K_{X_b}\otimes E|_{X_b}\to X_b$ and show that
\begin{equation}\label{exact4}
\ipr{\cp^{\perp}\nabla^{\cl(1,0)}_{\sigma}\mathfrak{f},\cp^{\perp}\nabla^{\cl(1,0)}_{\tau}\mathfrak{g}}_b = -\ipr{ \iota_{X_b}^*(\xi_\sigma\lrcorner \Theta^E)u + \iota_{X_b}^*\nabla^{1,0}(\dbar\xi_{\sigma}\lrcorner u), N^{(n,1)}_b\dbar \iota_{X_b}^*\nabla^{1,0}(\xi_{\tau}\lrcorner v)}_b.
\end{equation}
Adding equations \eqref{exact3} and \eqref{exact4} yields the desired formula. 
\end{proof}

\subsection{Another exact formula for the second fundamental form}

\begin{theorem}\label{2ndfunexact2}
Let $(E,h)\to X \to B$ be as in Proposition \ref{Bchoice}. Further assume that for all $b\in B$ the Neumann operator $\mathcal{N}^{(0,2)}_b$ acting on $(0,2)$-forms on $X_b$ is globally regular. Let $\tau \in \sC^{\infty}(T^{1,0}_B)_b$ be a germ of holomorphic $(1,0)$-vector field and let $\xi_{\tau}$ be its primitive horizontal lift. Let $\mathfrak{f} \in \Sigma_b(b)$ where $\mathfrak{f}(b) = \iota_{X_b}^*u$ for an $E$-valued $(n,0)$-form $u$. 
Then 
\begin{equation*}
   \norm{P_b^{\perp}\nabla^{\cl^\theta(1,0)}_\tau\mathfrak{f}}^2_b = \ipr{\alpha, \iota_{X_b}^*(\xi_\tau\lrcorner \Theta^E)u}_b + \ipr{\dbar \Lambda_{\omega_b}\alpha, \iota_{X_b}^*(\dbar\xi_\tau \lrcorner u)}_b + \int_{\partial X_b}\left\{\iota_{X_b}^*(\ol{\xi}_\tau \lrcorner \partial \dbar \rho)\wedge\Lambda_{\omega_b}\alpha ,u   \right\}dS_b, 
\end{equation*}
where $\alpha := \sqi N^{(n,1)}_b \iota_{X_b}^*\dbar \nabla^{1,0} (\xi_\tau \lrcorner u)$ and $N^{(n,1)}_b$ is the Neumann operator acting on $E|_{X_b}$-valued $(n,1)$-forms. 
\end{theorem}

\begin{proof}
By proposition \ref{Bchoice} we can take $u$ to be an $E$-valued $(n,0)$-form such that $\iota_{X_b}^*(\xi_{\sigma}\lrcorner \nabla^{1,0} u)$ is a holomorphic $E$-valued form on $X_b$. Then
\[
 \ipr{\cp^{\perp}\nabla^{\cl(1,0)}_{\tau}\mathfrak{f},\cp^{\perp}\nabla^{\cl(1,0)}_{\tau}\mathfrak{f}}_b = \ipr{P^{\perp}_b \iota_{X_b}^*(L^{1,0}_{\xi_\tau}u), P^{\perp}_b \iota_{X_b}^*(L^{1,0}_{\xi_\tau}u) } = \ipr{ P^{\perp}_b \iota_{X_b}^*(\nabla^{1,0}\xi_{\tau}\lrcorner u), P^{\perp}_b \iota_{X_b}^*(\nabla^{1,0}\xi_{\tau}\lrcorner u)}.
\]
By our choice of $u$ and \eqref{Hodge01}
\begin{align*}
P^{\perp}\iota_{X_b}^*\nabla^{1,0}(\xi_\tau\lrcorner u) &= \dbarstar \dbar N^{(n,0)}P^{\perp} \iota_{X_b}^*\nabla^{1,0} (\xi_\tau \lrcorner u) \\
&= \vartheta \dbar N^{(n,0)}P^{\perp} \iota_{X_b}^*\nabla^{1,0} (\xi_\tau \lrcorner u) \\
&= \sqi [\nabla^{1,0},\Lambda_{\omega_b}]\dbar N^{(n,0)}P^{\perp} \iota_{X_b}^*\nabla^{1,0} (\xi_\tau \lrcorner u) \\
&= \sqi \nabla^{1,0}\Lambda_{\omega_b} N^{(n,1)}\dbar P^{\perp} \iota_{X_b}^*\nabla^{1,0} (\xi_\tau \lrcorner u) \\
&= \sqi \nabla^{1,0}\Lambda_{\omega_b} N^{(n,1)} \iota_{X_b}^*\dbar \nabla^{1,0} (\xi_\tau \lrcorner u).
\end{align*}
The second equality follows since $\dbar N^{(n,0)}P^{\perp} \iota_{X_b}^*\nabla^{1,0} (\xi_\tau \lrcorner u)$ is a smooth form in $\text{Dom}(\dbarstar)$. The third equality follows from the K\"ahler identity $\sqi[\nabla^{1,0},\Lambda_{\omega_b}]= \vartheta$, and the penultimate equality follows from \eqref{Hodge02}. Let $\alpha := \sqi N^{(n,1)} \iota_{X_b}^*\dbar \nabla^{1,0} (\xi_\tau \lrcorner u)$.
Then,
\begin{align}
&\ipr{ P^{\perp}_b \iota_{X_b}^*(\nabla^{1,0}\xi_{\tau}\lrcorner u), P^{\perp}_b \iota_{X_b}^*(\nabla^{1,0}\xi_{\tau}\lrcorner u)} \nonumber \\
&= c_n \int_{X_b} \hwedge{P^{\perp}_b \iota_{X_b}^*(\nabla^{1,0}\xi_{\tau}\lrcorner u)}{P^{\perp}_b \iota_{X_b}^*(\nabla^{1,0}\xi_{\tau}\lrcorner u)} \nonumber \\
&= c_n\int_{X_b} \hwedge{\nabla^{1,0}\Lambda_{\omega_b}\alpha}{P^{\perp}_b \iota_{X_b}^*(\nabla^{1,0}\xi_{\tau}\lrcorner u)} \nonumber \\
&= c_n\int_{X_b}\partial \hwedge{\Lambda_{\omega_b}\alpha}{P^{\perp}_b \iota_{X_b}^*(\nabla^{1,0}\xi_{\tau}\lrcorner u)} + (-1)^n \hwedge{\Lambda_{\omega_b}\alpha}{\dbar P^{\perp}_b \iota_{X_b}^*(\nabla^{1,0}\xi_{\tau}\lrcorner u)}  \nonumber \\
&= c_n\int_{\partial X_b} \hwedge{\Lambda_{\omega_b}\alpha}{P^{\perp}_b \iota_{X_b}^*(\nabla^{1,0}\xi_{\tau}\lrcorner u)}  + (-1)^n c_n \int_{X_b} \hwedge{\Lambda_{\omega_b}\alpha}{\dbar \iota_{X_b}^*(L^{1,0}_{\xi_\tau} u)}. \label{exact1}
\end{align}
On $\partial X_b$ 
\begin{align*}
\hwedge{\Lambda_{\omega_b}\alpha}{P^{\perp}_b \iota_{X_b}^*(\nabla^{1,0}\xi_{\tau}\lrcorner u)} &= \left\{ d\rho_b  \wedge  \Lambda_{\omega_b}\alpha \wedge \ol{P^{\perp}_b \iota_{X_b}^*(\nabla^{1,0}\xi_{\tau}\lrcorner u)}, \frac{\omega_b^n}{n!} \right\} dS_b \\
&= \left\{ \partial \rho_b \wedge \Lambda_{\omega_b}\alpha \wedge \ol{P^{\perp}_b \iota_{X_b}^*(\nabla^{1,0}\xi_{\tau}\lrcorner u)}, \frac{\omega_b^n}{n!} \right\} dS_b =0,
\end{align*}
since $\partial \rho_b \wedge \Lambda_{\omega_b}\alpha =0$ on $\partial X_b$ is the condition that $\alpha \in \text{Dom}(\dbarstar)$. Thus the boundary integral in \eqref{exact1} vanishes.

Substituting \eqref{exact2} in \eqref{exact1} gives
\begin{align}
&\ipr{ P^{\perp}_b \iota_{X_b}^*(\nabla^{1,0}\xi_{\tau}\lrcorner u), P^{\perp}_b \iota_{X_b}^*(\nabla^{1,0}\xi_{\tau}\lrcorner u)} \nonumber \\
&= (-1)^{n-1}c_n\int_{X_b} \hwedge{\Lambda_{\omega_b}\alpha}{\iota_{X_b}^*(\xi_\tau\lrcorner \Theta^E)u} + c_n\int_{X_b} (-1)^{n-1}\hwedge{\Lambda_{\omega_b}\alpha}{\iota_{X_b^*}\nabla^{1,0}(\dbar\xi_\tau \lrcorner u)} \nonumber \\
&= (-1)^{n-1}c_n \int_{X_b} \hwedge{\Lambda_{\omega_b}\alpha}{\iota_{X_b}^*(\xi_\tau\lrcorner \Theta^E)u}  +c_n \int_{X_b} \dbar\hwedge{\Lambda_{\omega_b}\alpha}{\iota_{X_b}^*(\dbar\xi_\tau\lrcorner u)} \nonumber \\
&\quad - c_n\int_{X_b} \hwedge{\dbar \Lambda_{\omega_b}\alpha}{\iota_{X_b}^*(\dbar\xi_\tau \lrcorner u)} \label{secwmorrey} 
\end{align}
Using the identity $L_{\omega_b}\Lambda_{\omega_b}=[L_{\omega_b},\Lambda_{\omega_B}] = \text{Id}$ on twisted $(n,1)$-forms, we write the first term on the right hand side of \eqref{secwmorrey} as
\begin{align*}
    (-1)^{n-1}c_n \int_{X_b} \hwedge{\Lambda_{\omega_b}\alpha}{\iota_{X_b}^*(\xi_\tau\lrcorner \Theta^E)u} &= (-1)^{n-1}c_n \int_{X_b} \hwedge{\Lambda_{\omega_b}\alpha}{L_{\omega_b}\Lambda_{\omega_b}\iota_{X_b}^*(\xi_\tau\lrcorner \Theta^E)u} \\
                &= -\sqi c_{n-1} \int_{X_b} \hwedge{\omega_b \wedge \Lambda_{\omega_b}\alpha}{\Lambda_{\omega_b}\iota_{X_b}^*(\xi_\tau\lrcorner \Theta^E)u} \\
                &= \ipr{\Lambda_{\omega_b}\alpha, \Lambda_{\omega_b}\iota_{X_b}^*(\xi_\tau \lrcorner \Theta^E)u}_{\ch_b} \\
                &= \ipr{L_{\omega_b}\Lambda_{\omega_b}\alpha, \iota_{X_b}^*(\xi_\tau \lrcorner \Theta^E)u}_{\ch_b}  = \ipr{\alpha, \iota_{X_b}^*(\xi_\tau \lrcorner \Theta^E)u}_{\ch_b}. \\
\end{align*}
The second term on the right hand side of \eqref{secwmorrey} reduces to an integral over the boundary of $X_b$
\begin{equation}\label{morreybintegral}
c_n \int_{X_b} \dbar\hwedge{\Lambda_{\omega_b}\alpha}{\iota_{X_b}^*(\dbar\xi_\tau\lrcorner u)}
    = c_n \int_{\partial X_b} \hwedge{\Lambda_{\omega_b}\alpha}{\iota_{X_b}^*(\dbar\xi_\tau\lrcorner u)}. 
\end{equation}
We use the Morrey trick to simplify this boundary integral. Since $\xi_\tau$ is tangent to $\partial X$, we can write \begin{equation}\label{morrey1}
    \xi_\tau \lrcorner \partial \rho = r\rho
\end{equation} near $\partial X$ for some smooth positive function $r$. Applying $\dbar$ on both sides of \eqref{morrey1} gives
\begin{equation}\label{morrey2}
    \dbar\xi_\tau \lrcorner \partial \rho - \xi_\tau \lrcorner \dbar \partial \rho =\rho\dbar r + r\dbar \rho.
\end{equation}
Since $\dbar \xi_\tau$ is vertical and $\partial \rho \wedge u$ is an $(n+1,0)$-form we have
\begin{equation}\label{morrey3}
    0 = \iota_{X_b}^* \left(\dbar\xi_\tau \lrcorner (\partial \rho \wedge u)\right) = \iota_{X_b}^* \left((\dbar \xi_\tau \lrcorner \partial \rho)\wedge u - \partial \rho \wedge (\dbar \xi_\tau \lrcorner u)\right).
\end{equation}
Taking the wedge product of both sides of \eqref{morrey2} with $u$ and using \eqref{morrey3} to restrict to the fiber $X_b$, we get
\begin{equation}\label{morrey4}
    \iota_{X_b}^* \left((\xi_\tau \lrcorner \dbar \partial \rho)\wedge u +\rho\dbar r\wedge u + r\dbar \rho \wedge u - \partial \rho \wedge (\dbar \xi_\tau\lrcorner u)  \right) = 0.
\end{equation}
On $\partial X_b$, integrand on the right hand side of \eqref{morreybintegral} can be simplified as
\begin{align*}
    c_n \hwedge{\Lambda_{\omega_b}\alpha}{\iota_{X_b}^*(\dbar\xi_\tau\lrcorner u)} 
    &= c_n \left\{d\rho_b \wedge \Lambda_{\omega_b} \alpha\wedge \overline{\iota_{X_b}^*(\dbar\xi_\tau \lrcorner u)}, \frac{\omega_b^n}{n!}   \right\}dS_b \\
    &= c_n  \left\{\dbar \rho_b \wedge \Lambda_{\omega_b} \alpha\wedge \overline{\iota_{X_b}^*(\dbar\xi_\tau \lrcorner u)}, \frac{\omega_b^n}{n!}   \right\}dS_b \\
    &= c_n (-1)^{n-1}  \left\{\Lambda_{\omega_b} \alpha\wedge \overline{\partial \rho_b \wedge\iota_{X_b}^*(\dbar\xi_\tau \lrcorner u)}, \frac{\omega_b^n}{n!}   \right\}dS_b \\
    &=c_n (-1)^{n-1}  \left\{\Lambda_{\omega_b} \alpha\wedge \overline{\iota_{X_b}^*(\partial \rho \wedge(\dbar\xi_\tau \lrcorner u))}, \frac{\omega_b^n}{n!}   \right\}dS_b \\
    &= (-1)^{n-1}c_n  \left\{\Lambda_{\omega_b} \alpha\wedge \overline{\iota_{X_b}^*((\xi_\tau \lrcorner \dbar\partial \rho)\wedge u+ r\dbar \rho \wedge u)}, \frac{\omega_b^n}{n!}   \right\}dS_b \\
    &= c_n \left\{\iota_{X_b}^*(\ol{\xi}_\tau \lrcorner \partial \dbar \rho) \wedge \Lambda_{\omega_b} \alpha \wedge \ol{u}, \frac{\omega_b^n}{n!}   \right\}dS_b + c_n\left\{ r\partial \rho_b \wedge \Lambda_{\omega_b}\alpha \wedge\ol{u}, \frac{\omega_b^n}{n!}\right\}dS_b, 
\end{align*}
where the penultimate equality follows from \eqref{morrey4} since $\rho|_{\partial X_b} = 0$.  
Thus, the equation \eqref{morreybintegral} reduces to 
\begin{align*}
    &c_n \int_{\partial X_b} \hwedge{\Lambda_{\omega_b}\alpha}{\iota_{X_b}^*(\dbar\xi_\tau\lrcorner u)} \\
    &= c_n \int_{\partial X_b} \left\{\iota_{X_b}^*(\ol{\xi}_\tau \lrcorner \partial \dbar \rho) \wedge \Lambda_{\omega_b} \alpha \wedge \ol{u}, \frac{\omega_b^n}{n!}   \right\}dS_b + c_n \int_{\partial X_b}\left\{ r\partial \rho_b \wedge \Lambda_{\omega_b}\alpha \wedge\ol{u}, \frac{\omega_b^n}{n!}\right\}dS_b \\
     &= \int_{\partial X_b}\left\{\iota_{X_b}^*(\ol{\xi}_\tau \lrcorner \partial \dbar \rho)\wedge\Lambda_{\omega_b}\alpha ,u   \right\}dS_b.
\end{align*}
Here the last equality follows from the fact that $\alpha = \sqi N^{(n,1)}_b\iota_{X_b}^*\dbar\nabla^{1,0}(\xi_\tau\lrcorner u) \in \text{Dom}(\dbarstar)$ and $\partial \rho_b \wedge \Lambda_{\omega_b}\alpha =0$ is the condition for a smooth $(n,1)$-form to be in $\text{Dom}(\dbarstar)$.

To simplify the third term on right hand side of \eqref{secwmorrey}, we choose a lift $\xi_\tau$ as in Lemma \ref{primlift} so that $\iota_{X_b}^*(\dbar \xi_\tau \lrcorner u)$ is primitive. Note that $\dbar \Lambda_{\omega_b}\alpha$ is primitive. Indeed since $\dbar \omega_b = 0$, we have 
\begin{align*}
    \omega_b \wedge \dbar \Lambda_{\omega_b}\alpha = \dbar(\omega_b \wedge \Lambda_{\omega_b}\alpha) = \dbar L_{\omega_b}\Lambda_{\omega_b}\alpha = \dbar \alpha = 0,
\end{align*}
since $L_{\omega_b}\Lambda_{\omega_b} = \text{Id}$ on $E$-valued $(n,1)$-forms and $\alpha = \sqi \dbar N^{n,0}_b\iota_{X_b}^*\nabla^{1,0}(\xi_\tau \lrcorner u)$. 
Thus we have
\[
- c_n\int_{X_b} \hwedge{\dbar \Lambda_{\omega_b}\alpha}{\iota_{X_b}^*(\dbar\xi_\tau \lrcorner u)} = \int_{X_b} \{\dbar \Lambda_{\omega_b}\alpha, \iota_{X_b}^*(\dbar\xi_\tau \lrcorner u)\} dV_{\omega_b} = \ipr{\dbar \Lambda_{\omega_b}\alpha, \iota_{X_b}^*(\dbar\xi_\tau \lrcorner u)}_{\ch_b}.
\]
Putting all this together, we get
\begin{equation*}
   \norm{P_b^{\perp}L^{1,0}_{\xi_\tau}u}^2_{\ch_b} = \ipr{\alpha, \iota_{X_b}^*(\xi_\tau\lrcorner \Theta^E)u}_{\ch_b} + \ipr{\dbar \Lambda_{\omega_b}\alpha, \iota_{X_b}^*(\dbar\xi_\tau \lrcorner u)}_{\ch_b} + \int_{\partial X_b}\left\{\iota_{X_b}^*(\ol{\xi}_\tau \lrcorner \partial \dbar \rho)\wedge\Lambda_{\omega_b}\alpha ,u   \right\}dS_b. 
\end{equation*}
\end{proof}



\bibliographystyle{alpha}
\bibliography{DirectImagesRef}

\begin{thebibliography}{Wan17}

\bibitem[Ber09]{B2009}
Bo~Berndtsson.
\newblock Curvature of vector bundles associated to holomorphic fibrations.
\newblock {\em Ann. of Math. (2)}, 169(2):531--560, 2009.

\bibitem[Ber11]{BerndtssonStrict}
Bo~Berndtsson.
\newblock Strict and nonstrict positivity of direct image bundles.
\newblock {\em Math. Z.}, 269(3-4):1201--1218, 2011.

\bibitem[BL16]{BL2016}
Bo~Berndtsson and L{\'a}szl{\'o} Lempert.
\newblock A proof of the {O}hsawa--{T}akegoshi theorem with sharp estimates.
\newblock {\em Journal of the Mathematical Society of Japan}, 68(4):1461--1472,
  2016.

\bibitem[BS91]{boasstraube}
Harold~P Boas and Emil~J Straube.
\newblock Sobolev estimates for the $\overline{\partial}$-neumann operator on
  domains in $\mathbb{C}^n$ admitting a defining function that is
  plurisubharmonic on the boundary.
\newblock {\em Mathematische Zeitschrift}, 206(1):81--88, 1991.

\bibitem[CS01]{chenshaw}
So-Chin Chen and Mei-Chi Shaw.
\newblock {\em Partial differential equations in several complex variables},
  volume~19 of {\em AMS/IP Studies in Advanced Mathematics}.
\newblock American Mathematical Society, Providence, RI; International Press,
  Boston, MA, 2001.

\bibitem[Huy05]{huyb}
Daniel Huybrechts.
\newblock {\em Complex geometry}.
\newblock Universitext. Springer-Verlag, Berlin, 2005.

\bibitem[KN65]{KoNi}
J.~J. Kohn and L.~Nirenberg.
\newblock Non-coercive boundary value problems.
\newblock {\em Comm. Pure Appl. Math.}, 18:443--492, 1965.

\bibitem[LS14]{LS2014}
L\'{a}szl\'{o} Lempert and R\'{o}bert Sz\H{o}ke.
\newblock Direct images, fields of {H}ilbert spaces, and geometric
  quantization.
\newblock {\em Comm. Math. Phys.}, 327(1):49--99, 2014.

\bibitem[LTS20]{SL-T2020}
Christine Laurent-Thi\'{e}baut and Mei-Chi Shaw.
\newblock Holomorphic approximation via {D}olbeault cohomology.
\newblock {\em Math. Z.}, 296(3-4):1027--1047, 2020.

\bibitem[Str01]{straube2001}
Emil~J. Straube.
\newblock Good {S}tein neighborhood bases and regularity of the
  {$\overline\partial$}-{N}eumann problem.
\newblock {\em Illinois J. Math.}, 45(3):865--871, 2001.

\bibitem[Str10]{Strbook}
Emil~J. Straube.
\newblock {\em Lectures on the {$\mathscr{L}^2$}-{S}obolev theory of the
  {$\overline{\partial}$}-{N}eumann problem}.
\newblock ESI Lectures in Mathematics and Physics. European Mathematical
  Society (EMS), Z\"{u}rich, 2010.

\bibitem[Var22]{V2021}
Dror Varolin.
\newblock Berndtsson-{L}empert-{S}z{\H{o}}ke fields associated to proper
  holomorphic families of vector bundles.
\newblock {\em arXiv:2201.12802}, 2022.

\bibitem[Wan17]{W2017}
Xu~Wang.
\newblock A curvature formula associated to a family of pseudoconvex domains.
\newblock {\em Ann. Inst. Fourier (Grenoble)}, 67(1):269--313, 2017.

\end{thebibliography}

\end{document}